\documentclass[11pt]{article}

\usepackage{amsmath,amsthm,amsfonts,amssymb,enumerate,tikz}
\usepackage{color}
\usepackage{mathdots}
\usepackage[vcentermath,enableskew]{youngtab}
\usepackage[ocgcolorlinks, linkcolor=red!90!black, citecolor=blue!80!black]{hyperref}
\numberwithin{equation}{section}

\theoremstyle{definition}\newtheorem{thm}{Theorem}[section]
\theoremstyle{definition}\newtheorem*{thm*}{Theorem}
\newtheorem{lem}[thm]{Lemma}
\theoremstyle{definition} \newtheorem{defn}[thm]{Definition}
\newtheorem*{defn*}{Definition}
\newtheorem{ex}[thm]{Example}
\newtheorem*{lem*}{Lemma}
\newtheorem{cor}[thm]{Corollary}
\newtheorem*{cor*}{Corollary}
\newtheorem*{rem}{Remark}
\newtheorem{remark}[thm]{Remark}
\newtheorem*{conj*}{Conjecture}

\newtheorem{prop}[thm]{Proposition}

\newcommand{\CC}{\mathbb{C}}
\newcommand{\NN}{\mathbb{N}}	
\newcommand{\ZZ}{\mathbb{Z}}

\newcommand{\KK}{\mathbb{K}}

\newcommand{\I}{\mathcal{I}}
\newcommand{\K}{\mathcal{K}}

\newcommand{\RP}{\mathcal{RP}}

\newcommand{\FP}{\mathcal{FP}}
\newcommand{\RFP}{\mathcal{FP}}
\newcommand{\NFP}{\FP^+}

\DeclareMathOperator{\word}{word}
\DeclareMathOperator{\init}{in}
\DeclareMathOperator{\Ess}{Ess}
\DeclareMathOperator{\dom}{dom}
\DeclareMathOperator{\rad}{rad}

\DeclareMathOperator{\GL}{GL}
\DeclareMathOperator{\Mat}{Mat}
\DeclareMathOperator{\rank}{rank}
\DeclareMathOperator{\sort}{sort}
\DeclareMathOperator{\link}{lk}
\DeclareMathOperator{\del}{del}
\DeclareMathOperator{\pf}{pf}
\DeclareMathOperator{\sgn}{sgn}
\newcommand{\ch}{\operatorname{char}}
\newcommand{\setsgn}[2]{\sgn\!\tbinom{#1}{#2}}

\newcommand{\Hilb}{\mathcal{H}}

\newcommand{\Umat}{\mathcal{U}}
\newcommand{\SSU}{\mathcal{U}^{\text{ss}}}
\newcommand{\SSu}{u^{\text{ss}}}

\newcommand{\openX}{\mathring{X}}
\newcommand{\openSSX}{\mathring{X}^{\text{ss}}}
\newcommand{\SSX}{X^{\text{ss}}}
\newcommand{\SSM}{\operatorname{Mat}^{\text{ss}}}

\newcommand{\subword}{\Sigma}

\newcommand{\fkS}{\mathfrak{S}}
\DeclareMathOperator{\Sp}{Sp}
\renewcommand{\O}{\operatorname{O}}

\newcommand{\FPF}{\operatorname{FPF}}
\renewcommand{\ss}{\operatorname{ss}}
\newcommand{\SSI}{I^{\text{ss}}}
\newcommand{\SSD}{D^{\text{ss}}}
\newcommand{\SSr}{{\bf r}^{\text{ss}}}
\newcommand{\SSJ}{J^{\text{ss}}}

\def\r{{\bf r}}
\def\s{{\bf s}}

\def\SD{\mathsf{SD}}

\def\Gfpf{\mathfrak{G}^{\mathsf{Sp}}}

\def\cP{\mathcal{P}}

\def\untwist{\mathsf{untwist}}

\newcommand{\ltriang}{\raisebox{-0.5pt}{\tikz{\draw (0,0) -- (.25,0) -- (0,.25) -- (0,0);}}}

\def\quand{\quad\text{and}\quad}
\def\DesR{\mathrm{Des}_R}
\def\DesL{\mathrm{Des}_L}

\def\be{\begin{equation}}
\def\ee{\end{equation}}

\def\fpf{\mathsf{FPF}}
\def\ellfpf{\ell_{\fpf}}

\newcommand{\Ifpf}{\mathcal{I}^{\fpf}}

\def\cM{\mathcal{M}}
\def\OC{\mathsf{OC}}
\def\dbl{\mathsf{dbl}}

\def\cX{\mathcal{X}}

\newcommand{\J}[1]{\mathsf{Ideal}(#1)}

\def\deltafpf{\delta_{\fpf}}
\def\mon{\mathsf{mon}}

\usepackage{fullpage}

\usepackage[colorinlistoftodos]{todonotes}

\begin{document}

\title{
Gr\"obner geometry for skew-symmetric matrix Schubert varieties}

 \author{
Eric Marberg 
\\ {\tt eric.marberg@gmail.com}
\and
Brendan Pawlowski 
\\ {\tt br.pawlowski@gmail.com}
}

\date{}

\maketitle

\begin{abstract}
Matrix Schubert varieties are the closures of the orbits of $B\times B$ acting on 
    all $n\times n$ matrices, where $B$ is the group of invertible 
    lower triangular matrices.
Extending work of Fulton,    Knutson and Miller identified a Gr\"obner  basis for the 
prime ideals of these varieties.
They also showed that the corresponding initial ideals are Stanley-Reisner ideals of  shellable simplicial complexes, and derived a related primary decomposition in terms of reduced pipe dreams. These results lead to a geometric proof of the Billey-Jockusch-Stanley formula for a Schubert polynomial, among many other applications.

We define skew-symmetric matrix Schubert varieties to be the nonempty intersections of matrix Schubert varieties with the subspace of skew-symmetric matrices.
In analogy with Knutson and Miller's work, we describe a natural generating set for the prime ideals of these varieties. We then compute a related Gr\"obner basis. Using these results, we identify a primary decomposition for the corresponding initial ideals involving certain fpf-involution pipe dreams. We show that these initial ideals are likewise the Stanley-Reisner ideals of shellable simplicial complexes. As an  application, we give a geometric proof of an explicit generating function for symplectic Grothendieck polynomials.
Our methods differ from Knutson and Miller's and can be used to give new proofs of some of their results, as we explain at the end of this article.
\end{abstract}

\tableofcontents

\section{Introduction}
Let $\KK$ be an algebraically closed field and write $B_n \subseteq \GL_n := \GL_n(\KK)$ for the Borel group of $n \times n$ invertible lower triangular matrices over $\KK$. The \emph{flag variety} over $\KK$ is the quotient $\GL_n/B_n$ equipped with the structure of a projective variety. If $\KK = \CC$, then a subvariety $X \subseteq \GL_n/B_n$ determines a class $[X]$ in the cohomology ring $ H^*(\GL_n/B_n, \ZZ)$, which Borel identified with a quotient of $\ZZ[x_1, x_2,\ldots, x_n]$. One can therefore ask for a polynomial representing $[X]$.

When $K \subseteq \GL_n$ is an algebraic subgroup, the Zariski closures of the $K$-orbits on $\GL_n/B_n$ give a particularly interesting family of subvarieties. For $K = B_n$, there are finitely many orbits, naturally indexed by the elements of the symmetric group $S_n$. The closures of these orbits are the \emph{Schubert varieties} and their cohomology classes can be represented by the family of \emph{Schubert polynomials} $\{\fkS_w : w \in S_n\}$. Other subgroups $K$ acting with finitely many orbits (the so-called \emph{spherical subgroups}) include the symplectic group $\Sp_n$ when $n$ is even, the orthogonal group $\O_n$, and the block diagonal subgroups $\GL_p \times \GL_{q}$ for $p+q=n$.

The $K$-orbits on $\GL_n/B_n$ are in bijection with the $B_n$-orbits on $\GL_n/K$, as well as with the $B_n \times K$-orbits on $\GL_n$,
and it can be fruitful to consider the orbits from these alternate points of view. For example, take $K = B_n$ and consider the $B_n \times B_n$-action on the space of all $n \times n$ matrices $\Mat_{n \times n} := \Mat_{n \times n}(\KK)$, by $(g,h) \cdot A = g A h^T$. The orbits for this action are called \emph{matrix Schubert cells}, and their Zariski closures are the \emph{matrix Schubert varieties}. 

The $B_n\times B_n$ orbits contained in the open dense subset $ \GL_n\subset \Mat_{n\times n}$ 
are in bijection with Schubert varieties  via pulling back along the quotient map $\GL_n \to \GL_n/B_n$. It is useful to consider the full space $\Mat_{n \times n}$ since matrix Schubert varieties are then affine subvarieties of an affine space, and one can work with their ideals using tools of commutative algebra that are unavailable in the flag variety setting.

Let $u_{ij}$ for $i,j \in [n]:= \{1,2,\dots,n\}$ be commuting indeterminates and 
write $\Umat$ for the $n\times n$ matrix with $u_{ij}$ in position $(i,j)$.
We identify the coordinate ring $\KK[\Mat_{n \times n}]$ with $\KK[u_{ij} : 1 \leq i, j \leq n]$. Given a closed subvariety $X \subseteq \Mat_{n \times n}$, let $I(X) \subseteq \KK[\Mat_{n \times n}]$ be the ideal of polynomials vanishing on $X$. If $A$ is an $n\times n$ matrix and $R, C \subseteq  \{1,2,\ldots,n\}$, then we write $A_{RC}$ for the submatrix $[A_{ij}]_{(i,j) \in R \times C}$. Finally, let $X_A$ denote the closure of the $B_n \times B_n$-orbit of  $A \in \Mat_{n \times n}$.
 Each orbit closure $X_A$ is a subvariety, and one can describe the 
ideal $I(X_A)$ explicitly as follows.

 \begin{thm}[{\cite[Thms. A and B]{KnutsonMiller}}] \label{thm:KM-ideal}
    For each $A \in \Mat_{n \times n}$,  the collection of minors $\det(\Umat_{RC})$, for all $R \subseteq [i]$ and $C \subseteq [j]$ with $|R| = |C| = \rank A_{[i][j]}+1$ for some $i,j \in [n]$, generates the (prime) ideal $I(X_A)$ and forms a Gr\"obner basis with respect to any \emph{antidiagonal term order}.
 \end{thm}
 
In this statement, 
 a \emph{term order} on a polynomial ring means a total ordering of the monomials in which $1$ is minimal and multiplication by a fixed monomial preserves order.
 For the ring $\KK[\Mat_{n \times n}]=\KK[u_{ij} : 1 \leq i, j \leq n]$, a term order
 is \emph{antidiagonal} if the initial term (defined as the maximal monomial with nonzero coefficient) of any minor $\det(\Umat_{RC})$ is the product of the antidiagonal entries of $\Umat_{RC}$.
  The prototypical example of an antidiagonal term order on $\KK[\Mat_{n \times n}]$ is  \emph{(graded) reverse lexicographic order}, whose definition is reviewed
 in Example~\ref{lex-term2}. 

 If $I$ is an ideal, then a generating set $S \subseteq I$ is a \emph{Gr\"obner basis} when the initial terms of the elements of $S$ are a generating set for the \emph{initial ideal} $\init(I)$ generated by the initial terms of all elements of $I$.
   The assertion in Theorem~\ref{thm:KM-ideal} that the minors $\det(\Umat_{RC})$ generate $I(X_A)$ is originally due to Fulton \cite[Prop. 3.3]{FultonEssentialSet}.
   Knutson and Miller \cite{KnutsonMiller} reprove this result using different techniques to obtain the stronger statement that 
   these minors form a Gr\"obner basis. 

 If a subvariety $X \subseteq \Mat_{n \times n}(\CC)$ is invariant under the left $B_n$-action, then it defines a class $[X]_{B_n}$ in the equivariant cohomology ring $H_{B_n}^*(\Mat_{n \times n}(\CC)) \simeq \ZZ[x_1, \ldots, x_n]$. The polynomial $[X]_{B_n}$ can also be computed algebraically as the \emph{multidegree} of the ideal $I(X)$, and this definition works more generally over any algebraically closed field $\KK$. 
 
 Identify a permutation $w \in S_n$ with the $n \times n$ permutation matrix with 1 in each position $(i, w(i))$ and 0 everywhere else. The $B_n \times B_n$-orbits contained in $\GL_n \subseteq \Mat_{n \times n}$ are precisely the closures of the matrix Schubert varieties $\{X_w : w \in S_n\}$. For simplicity, and to emphasize the connection to Schubert varieties in $\GL_n/B_n$, we state the next two theorems only for those orbits; both statements can be generalized to all matrix Schubert varieties without serious difficulty.
When we write $\init(I(X_w))$ we mean the initial ideal under any fixed antidiagonal term order.

 \begin{thm}[{\cite[Thms. A and B]{KnutsonMiller}}]\label{thm:KM-primary-decomp} 
If $w \in S_n$ then $[X_w]_{B_n}$ is equal to the polynomial $\fkS_w$,
and the initial ideal of $I(X_w)$ has primary decomposition
    $\init(I(X_w))=\bigcap_{D \in \RP(w)} (u_{ij} : (i,j) \in D)$,
    where $\RP(w)\subseteq \binom{[n]\times [n]}{\ell(w)}$ is the set of \emph{(reduced) pipe dreams} for $w$ (see \S\ref{subsec:pipe-dreams}),
    and 
     $(u_{ij} : (i,j) \in D)$ denotes the (prime) ideal of $\KK[\Mat_{n \times n}]$ generated by $u_{ij}$ 
for all $(i,j) \in D$.
\end{thm}

As an application of this theorem, by using an algebraic interpretation of $[X_w]_{B_n}$ as the multidegree of  $I(X_w)$, Knutson and Miller 
provide a geometric proof of 
the Billey-Jockusch-Stanley formula \cite[Thm. 1.1]{BJS} for
the Schubert polynomial $\fkS_w$. 
Knutson and Miller extract more detailed information about the ideals $I(X_w)$, such as their $K$-polynomials, 
from this related result:

\begin{thm}[\cite{SubwordComplexes}] \label{thm:subword-complexes} 
For each $w \in S_n$, 
the ideal $\init(I(X_w))$ is square-free (that is, it contains $f$ if it contains $f^2$) and equal to the Stanley-Reisner ideal of a 
 shellable simplicial complex.
\end{thm}

Our goal in this paper is to prove similar theorems about the analogues of matrix Schubert varieties in another 
space of fundamental interest, namely,
the subspace of all skew-symmetric matrices in $\Mat_{n\times n}$. 
The rest of this introduction outlines our main new results in this direction.

Over an arbitrary field $\KK$,  an $n\times n$ matrix $A$ 
is \emph{skew-symmetric} if $A_{ij} = -A_{ji}$ and $A_{ii} = 0$ for all $i,j \in [n]$; note that the second condition is redundant if $\ch(\KK) \neq 2$. 
Let $\SSM_n := \SSM_n(\KK) $ denote the subset of such matrices in $\Mat_{n\times n}(\KK)$.
We define
the \emph{skew-symmetric matrix Schubert variety} associated to $A \in \SSM_n$ 
to be the intersection $\SSX_A := X_A \cap \SSM_n$.

Identify the coordinate ring $\KK[\SSM_n]$ with $\KK[u_{ij} : n \geq i > j \geq 1]$, and write $\SSU$ for the $n \times n$ skew-symmetric matrix with $\SSU_{ij} = u_{ij} = -\SSU_{ji}$ for $i > j$
and $\SSU_{ii} = 0$ for all $i$.
Each $\SSX_A$ is the zero locus of the obvious skew-symmetric
analogue of the minors in Theorem~\ref{thm:KM-ideal}, that is, the family consisting of $\det(\SSU_{RC})$ for all $R \subseteq [i]$ and $C \subseteq [j]$ with $|R| = |C| = \rank A_{[i][j]}+1$ for some $i,j\in [n]$.
These polynomials do not always generate the ideal $I(\SSX_A)$, however.
There is nevertheless a natural generating set for $I(\SSX_A)$, which we describe as follows:

\begin{thm}[See Definition~\ref{ssi-def} and Theorem~\ref{ss-main-thm2}] \label{thm:fpf-ideal-intro} For each $A \in \SSM_n$, the collection of  Pfaffians $\pf(\SSU_{RR})$, as $R$ ranges over all even sized subsets of $[n]$ such that $R \subseteq [i]$ and $|R \cap [j]| > \rank A_{[i][j]}$ for some $i,j \in [n]$ with $i\geq j$, generate the (prime) ideal $I(\SSX_A)$.
\end{thm}

In another point of departure from ordinary matrix Schubert varieties,
finding a generating set for $I(\SSX_A)$ does not immediately lead to a Gr\"obner basis. 
For example, the generating set  in the preceding theorem is generally \emph{not} a Gr\"obner basis for $I(\SSX_A)$ relative to an antidiagonal term order. This is because the initial ideal $\init(I(\SSX_A))$ can have generators divisible by $u_{ab}u_{ac}$ for some $a, b, c \in \NN$, which cannot be the leading term of any Pfaffian $\pf(\SSU_{RR})$; see Example~\ref{ex:SSI-generators}. 

Our second main result resolves the nontrivial problem of finding a Gr\"obner basis for $I(\SSX_A)$. Specifically, in 
Theorem~\ref{ss-main-thm3} we show that a Gr\"obner basis for $I(\SSX_A)$ with respect to the reverse lexicographic term order is provided by the Pfaffians of the
block diagonal matrices  \[\left[ \begin{array}{ll} \SSU_{CC} & \SSU_{CR} \\ \SSU_{RC} & 0
    \end{array}  \right]\]
    for certain subsets $R,C\subseteq [n]$.
Experimental evidence suggests that these Pfaffians may also form a Gr\"obner basis for other antidiagonal term orders. This more general claim does not follow from our present methods, however,
and will not be pursued in this article.

Suppose $n$ is even and $z \in S_n$ is a fixed-point-free involution,
that is, a permutation with $z(z(i)) = i \neq z(i)$ for all $i \in [n]$.
Associated to such a permutation is 
a set $\FP(z)$ of \emph{fpf-involution pipe dreams}, whose elements are certain subsets
of $\{ (i,j) \in [n] \times [n] : i > j\}$; see Definition~\ref{fp-def}
for the full details.
We let $\SSX_z := \SSX_A$ where $A$ is the skew-symmetric $n\times n$ matrix with
 $A_{ij} = 1$ if $z(j) = i<j =z(i)$ and $A_{ij} = -1$ if $z(j) = i>j = z(i)$.
Whenever we write $\init(I(\SSX_z))$, we mean the initial ideal under the reverse lexicographic term order defined in Example~\ref{lex-term2}.

\begin{thm}[See Theorems~\ref{ss-main-thm} and \ref{thm:primary-decomposition-ss}] \label{thm:fpf-primary-decomp-intro} For each fixed-point-free involution $z \in S_n$,
the initial ideal of $I(\SSX_z)$ has primary decomposition $\init(I(\SSX_z))=\bigcap_{D \in \RFP(z)} (u_{ij} : (i,j) \in D)$,
where $(u_{ij} : (i,j) \in D)$ denotes the (prime) ideal of $\KK[\SSM_n]$ generated by $u_{ij}$ 
for all $(i,j) \in D$.
 \end{thm}

The varieties $\SSX_z$ as $z$ ranges over the fixed-point-free involutions in $S_n$ are exactly the $B_n$-orbit closures in $\SSM_n \cap \GL_n$,
which is nonempty only if $n$ is even. For simplicity, we have stated Theorem~\ref{thm:fpf-primary-decomp-intro} only for these special cases of $\SSX_A$.  Theorem~\ref{thm:primary-decomposition-ss} below will extend this result to all skew-symmetric matrix Schubert varieties in $\SSM_n$ and all positive integers $n$.

Using Theorem~\ref{thm:fpf-primary-decomp-intro},
one can show that $[\SSX_z]_{B_n} = \sum_{D \in \RFP(z)} \prod_{(i,j) \in D} (x_i + x_j)$ under the identification $H_{B_n}^*(\SSM_n) \simeq \ZZ[x_1, \ldots, x_n]$. This formula was proven combinatorially in \cite{HMPdreams}. It was also shown in \cite{HMPdreams} that the polynomials $[\SSX_z]_{B_n}$ are the same as the \emph{fpf-involution Schubert polynomials} introduced by Wyser and Yong \cite{WyserYongOSp}, which represent the ordinary cohomology classes of the $\Sp_n(\CC)$-orbit closures on $\GL_n(\CC)/B_n$. Briefly, the connection to our situation is that $\Sp_n(\CC)$-orbits on $\GL_n(\CC)/B_n$ are in bijection with $B_n$-orbits on $\GL_n(\CC)/\Sp_n(\CC)$, which can be identified with $\SSM_n(\CC) \cap \GL_n(\CC)$.

Finally, we prove a skew-symmetric version of Theorem~\ref{thm:subword-complexes}. As with Theorem~\ref{thm:fpf-primary-decomp-intro}, we state the next result just for the special case when $\SSX_A \subseteq \SSM_n \cap \GL_n$, but we will
extend the theorem in Section~\ref{subsec:K} to all skew-symmetric matrix Schubert varieties.

\begin{thm}[See Theorems~\ref{thm:SR-ideal} and \ref{thm:shellable}]
 \label{thm:fpf-subword-complexes-intro}
For each fixed-point-free involution
 $z \in S_n$, the ideal $\init(I(\SSX_z))$ is square-free and equal to the Stanley-Reisner ideal of a shellable simplicial complex. 
\end{thm}

We use this result to give a new geometric proof of 
a combinatorial formula \cite[Thm. 4.5]{MacdonaldNote}  for 
the $B_n$-equivariant $K$-theory representative of $\SSX_z$; see Theorem~\ref{4.5-thm}.

Ideals generated by Pfaffians of a generic skew-symmetric matrix have been well-studied \cite{DeNegri,DeNegriSbarra,HerzogTrung,JonssonWelker,RaghavanUpadhyay}, and there is some overlap between our results and prior work. For example, Herzog and Trung showed in \cite{HerzogTrung} that for a fixed positive integer $r$, the Pfaffians of all $2r \times 2r$ submatrices of a generic skew-symmetric matrix form a Gr\"obner basis with respect to an appropriate term order. De Negri and Sbarra \cite{DeNegriSbarra} considered a larger family of ideals which, translated into our language, turns out to be a subfamily of the ideals $I(\SSX_A)$ for $A \in \SSM_n$; see Remark~\ref{rem:pfaffian-ideals} for a precise description. They observed that the Pfaffian generators of Theorem~\ref{thm:fpf-ideal-intro} need not form a Gr\"obner basis with respect to an antidiagonal term order, and then classified the ideals that do enjoy this property. Raghavan and Upadhyay \cite{RaghavanUpadhyay} examined the same family of ideals, computing their initial ideals and realizing the latter as Stanley-Reisner ideals of shellable complexes, as we do in Theorems~\ref{thm:fpf-primary-decomp-intro} and \ref{thm:fpf-subword-complexes-intro}. There is no easy way to translate between their results and ours, however, because they use term orders which are far from antidiagonal.

The techniques we use to prove Theorems~\ref{thm:fpf-ideal-intro} and \ref{thm:fpf-primary-decomp-intro} differ from those in \cite{KnutsonMiller}, and in fact lead to new proofs of Theorems~\ref{thm:KM-ideal} and \ref{thm:KM-primary-decomp}. We discuss these applications in Section~\ref{new-proofs-sect}. Knutson and Miller use a method in \cite{KnutsonMiller} called \emph{Bruhat induction}, which is an induction on weak Bruhat order on $S_n$ leveraging divided difference recurrences for Schubert classes. We instead induct on (strong) Bruhat order and use the transition recurrences of Lascoux and Sch\"utzenberger from \cite{LascouxSchutzenbergerTransitions}. A similar inductive approach to studying matrix Schubert varieties appears in \cite{HPW}.

The questions studied here are equally interesting to consider for 
\emph{symmetric matrix Schubert varieties}, where $\SSM_n$ is replaced by the space of symmetric matrices in $\Mat_{n\times n}$. 
A number of technical difficulties arise in that setting, however.
For example, the expected analogue of
Theorem~\ref{thm:fpf-primary-decomp-intro} no longer gives a pure primary decomposition into ideals of the same dimension, and this poses a fundamental obstruction to the techniques of Knutson and Miller from \cite{KnutsonMiller}. 
We hope that a variation of our new inductive approach can
be adapted in future work to avoid these issues.

\section*{Acknowledgements}
The first author was partially supported by Hong Kong RGC Grant ECS 26305218.
We thank Zach Hamaker for a much improved proof of Proposition~\ref{prop:rM-equals-rw} and for many other helpful conversations. We are also grateful to Allen Knutson for explaining his perspective on Schubert transition recurrences, which inspired some of the key techniques used here.

\section{Preliminaries}

Throughout, we write $\ZZ$ for the set of integers, $\NN = \{1,2,3,\dots\}$ 
for the set of natural numbers, and $[n] = \{ i \in \ZZ : 0 < i \leq n\}$ for the first $n$
positive integers.

\subsection{Initial ideals}

In this section we work over an arbitrary field $\KK$. (Outside this section, we will always assume that $\KK$ is algebraically closed.)
Suppose $x_1,x_2,\dots,x_N$ are commuting variables, and consider the polynomial ring $\KK[{\bf x}] := \KK[x_1,x_2,\dots,x_N]$. 
For us, a \emph{monomial} in $\KK[{\bf x}]$ is an element  $x_{i_1}x_{i_2}\dots x_{i_l}$ for some 
possibly empty sequence of (not necessarily distinct) indices $i_1,i_2,\dots,i_l \in [N]$.

A \emph{term order} on $\KK[{\bf x}] $ is a total order on the set of  all monomials,
such that $1$ is the unique minimum and such that if $\mon_1, \mon_2, \mon_3$ are monomials and $\mon_1 \leq \mon_2$, 
then $\mon_1  \mon_3 \leq \mon_2 \mon_3$. 
If $N=1$ then there is a unique term order, namely, $1 < x_1<x_1^2 < \dots$.

\begin{ex} \label{lex-term1}
The \emph{lexicographic term order} on $\KK[{\bf x}]$ declares
that  $x_1^{a_1}\cdots x_N^{a_N} \leq x_1^{b_1} \cdots x_N^{b_N}$ whenever $(a_1, \ldots, a_N) \leq (b_1, \ldots, b_N)$ in lexicographic order.
The \emph{(graded) reverse lexicographic term order} declares that 
$x_1^{a_1}\cdots x_N^{a_N} \leq x_1^{b_1} \cdots x_N^{b_N}$ whenever $\sum_i a_i \leq \sum_i b_i$ and $(a_N, \ldots, a_1) \geq (b_N, \ldots, b_1)$
in lexicographic order; note the double reversal. 
\end{ex}

Fix a term order and suppose $f = \sum_{\mon} c_\mon \cdot \mon \in \KK[{\bf x}]$ where the sum is over monomials $\mon$ and each $c_\mon \in \KK$.  
If $f$ is nonzero, then 
its \emph{initial term} (or \emph{leading term}) is the maximal monomial $\mon$ such that $c_\mon \neq 0$.
If $f=0$ then its initial term is also defined to be zero. In either case, we write $\init(f)$ for the corresponding initial term.

The \emph{initial ideal} of an ideal $I$ in $\KK[{\bf x}]$ is then $\init(I) := \KK\text{-}\mathrm{span}\{\init(f) : f \in I\}$. This abelian group is itself an ideal
 in $\KK[{\bf x}]$.
A \emph{Gr\"obner basis} $G$ for an ideal $I\subseteq \KK[{\bf x}]$, relative to a fixed term order,
is a generating set  whose set of initial terms $\{ \init(g) :g \in G\}$ generates $\init(I)$.

\begin{ex}\label{lex-term2}
In our applications, we will usually take $x_1,x_2,\dots,x_N$ to be 
either the commuting variables $u_{ij}$ indexed by all positions $(i,j) \in [m]\times [n]$ for some $m$ and $n$,
or the subset of these variables indexed by positions strictly below the main diagonal.

An \emph{antidiagonal term order} on $\KK[u_{ij} : (i,j) \in [m]\times [n]]$
is one with the property that the initial monomial of the determinant of any square submatrix $A$ of $[u_{ij}]_{(i,j) \in [m]\times [n]}$ is the product of the antidiagonal entries of $A$. 
Under an antidiagonal term order, one has, for example,
    \begin{equation*}
   \init\left(  \det \begin{bmatrix} u_{11} & u_{12} & u_{13} \\ u_{21} & u_{22} & u_{23} \\ u_{31} & u_{32} & u_{33} \end{bmatrix}\right)=u_{13}u_{22}u_{31}.
    \end{equation*}
The prototypical example of an antidiagonal term order is the (graded) reverse lexicographic order 
from Example~\ref{lex-term2}
with $u_{ij}$ identified with $x_{n(i-1) + j}$.
This means that we order the variables $u_{ij}$ lexicographically, so that $u_{ij} < u_{i'j'}$
if $i<i'$ or if $i=i'$ and $j<j'$.
Then, we declare that $\mon_1 < \mon_2$  if either $\deg (\mon_1)< \deg (\mon_2)$
or $\deg (\mon_1 )= \deg (\mon_2)$ and the following holds: there is some variable $u_{ij}$ whose exponent $e_1$ in $\mon_1$
differs from its exponent $e_2$ in $\mon_2$, and 
when $u_{ij}$ is the (lexicographically) largest such variable one has $e_1> e_2$.
If $\mon_1$ and $\mon_2$ are both square-free of the same degree,
then we have $\mon_1 < \mon_2$ if and only there is some variable $u_{ij}$ that does not divide both monomials, and 
the largest such variable divides $\mon_1$ but not $\mon_2$.
We refer to this order as the \emph{reverse lexicographic term order} on $\KK[u_{ij} : (i,j) \in [m]\times [n]]$.
\end{ex}

\begin{rem}
Unless otherwise indicated, all results concerning initial ideals for subrings of $\KK[u_{ij} : (i,j) \in [m]\times [n]]$
will be relative to the reverse lexicographic term order just described.
\end{rem}

We note a few basic facts in the generic setting of $\KK[{\bf x}]$ with a fixed term order.

\begin{lem}\label{lem:finite-chains} Every strictly descending chain of monomials in a term order on $\KK[{\bf x}]$ is finite. \end{lem} 
    \begin{proof}
    Since $1 \leq \mon$ for any monomial,  it follows that $\mon_1 \leq \mon_2$ whenever 
    $\mon_1, \mon_2 \in \KK[{\bf x}]$ are monomials with $\mon_1 \mid \mon_2
    $. 
        Dickson's lemma asserts that if $v_1, v_2, \ldots$ is an infinite sequence in $\NN^N$, then there are indices $i < j$ such that $v_i \leq v_j$ component-wise. Applying this to the exponent vectors of an infinite sequence of monomials $\mon_1, \mon_2, \ldots \in \KK[{\bf x}]$ shows that there are always 
        indices $i < j$ with $\mon_i \mid \mon_j$, so an infinite sequence is not strictly descending.
    \end{proof}

    \begin{prop} \label{prop:init-facts} Suppose $I$ and $J$ are ideals in $\KK[{\bf x}]$. Fix a term order on $\KK[{\bf x}]$. Then:
        \begin{enumerate}[(a)]
            \item If $I \subseteq J$, then $\init(I) \subseteq \init(J)$.
            \item It holds that $\init(I)+\init(J) \subseteq \init(I+J)$.
            \item It holds that $\init(I \cap J) \subseteq \init(I) \cap \init(J)$.
            \item If $I \subseteq J$ and $\init( I) = \init (J)$, then $I = J$.
        \end{enumerate}
    \end{prop}
        \begin{proof}
        Part (a) is clear. It implies $\init(I) \subseteq \init(I+J)$ and $\init(J) \subseteq \init(I+J)$, so part (b) follows. Similarly, part (a) implies $\init(I \cap J) \subseteq \init(I)$ and $\init(I \cap J) \subseteq \init(J)$, so part (c) follows.
    
       For part (d), suppose the hypotheses hold and $f \in J$. Since $\init(f) \in \init (I)$, there is some $g \in I$ with $\init(g) = \init(f)$. For some $c \in \KK$, $cf-g$ will therefore have a smaller initial term than $f$. Since $cf-g \in J$, we can iterate this process. Lemma~\ref{lem:finite-chains} implies that the process must terminate, and when it does we will have written $f$ as a linear combination of elements of $I$, so $J \subseteq I$.
        \end{proof}

\begin{lem} \label{lem:plus-vs-intersection} Suppose $I$ and $J$ are homogeneous ideals in $\KK[{\bf x}]$. 
If $\init(I+J) = \init(I) + \init(J)$ relative to a given term order on $\KK[{\bf x}]$, then $\init(I \cap J) = \init(I) \cap \init(J)$. \end{lem}

    \begin{proof}
        The \emph{Hilbert series} of a graded vector space $V = \bigoplus_{n \geq 0} V_n$ with each $\dim V_n$ is finite is the generating function $\Hilb(V) = \sum_{n \geq 0} \dim(V_n)t^n \in \ZZ[[t]]$. 
       A homogeneous ideal $I$ in $\KK[{\bf x}]$ is a graded vector space of finite graded dimension, as is its initial ideal,
       and it holds that $\Hilb(\init (I)) = \Hilb(I)$ \cite[Theorem 15.3]{Eisenbud}. 
        Since 
        $\init(I \cap J) \subseteq \init(I) \cap \init(J)$ by Proposition~\ref{prop:init-facts}, we 
        may consider the quotient ${\init(I) \cap \init(J)}/{\init(I \cap J)}$,
        which is itself a graded vector space.
        
        If $0 \to U \to V \to W \to 0$ is a short exact sequence of degree-preserving linear maps between graded vector spaces, then $\Hilb(U) = \Hilb(V) - \Hilb(W)$. 
        Using this fact with the exact sequences
        $ 0 \to A \hookrightarrow B \to B /A \to 0$ (when $A\subseteq B$) and
        $0 \to A \cap B \xrightarrow{x \mapsto (x,-x)} A \oplus B \to A + B \to 0,$ it is straightforward to check that
        $  \Hilb\left(\tfrac{\init(I) \cap \init(J)}{\init(I \cap J)}\right) = \Hilb(I\oplus J)  - \Hilb(I \cap J) - \Hilb(\init( I) + \init (J))$.
        If $\init(I + J) = \init(I) + \init(J)$, then the last term becomes 
        $\Hilb(\init(I+J)) = \Hilb(I+J) = \Hilb(I\oplus J) - \Hilb(I\cap J) $ and the whole expression simplifies to zero,
        which can only happen if $\init(I) \cap \init(J) = \init(I \cap J)$.
    \end{proof}

 Given an element $f \in \KK[{\bf x}]$, let $(f)$ denote the principal ideal it generates.

    \begin{lem} \label{lem:hilbert} Suppose $J \subseteq I$ are homogeneous ideals in $\KK[{\bf x}]$. Let $f \in \KK[{\bf x}]$ be a non-constant homogeneous polynomial such that $I \cap (f) = fI$ and $I + (f) = J + (f)$. Then $I = J$. \end{lem}

\begin{proof} Write $I = \bigoplus_{n \geq 0} I_n$ where $I_n$ is the set of homogeneous elements of $I$ with degree $n$.
Write
 $J = \bigoplus_{n \geq 0} J_n$ where $J_n$ is defined similarly. Choose $i \in I_n$ for some $n$, and assume by induction that $I_m = J_m$ for $m < n$. By hypothesis we can write $i = j + fg$ for some $j \in J$ and $g \in \KK[{\bf x}]$. 
 The homogeneous parts of $j$ and $fg$ outside degree $n$ must 
sum to zero. As $f$ is homogeneous, we can therefore assume that $j$ is homogeneous of degree $n$ and that $g$ is homogeneous of degree $n-\deg(f)$.
Now we have $fg = i-j \in I \cap (f) = fI$, so $g \in I_{n-\deg f}$. Since $f$ is non-constant,
we may assume by induction that $g \in J_{n-\deg f}$, and hence $i = j + fg \in J_n$.
\end{proof}

\subsection{Permutations}

Let $S_{\infty}$ be the group of permutations of the positive integers $\NN$ that fix all but finitely many points. 
We realize $S_n$ as the subgroup of permutations $w \in S_\infty$ with $w(i) =i$ for all $i>n$.
Let $s_i = (i,i+1)$ for $i \in \NN$. Then $S_\infty$ and $S_n$ are Coxeter groups relative to the
simple generating sets
$\{ s_i : i \in \NN\}$ and $\{ s_i : i \in [n-1]\}$, respectively.

Most of the time we will work with elements of $S_\infty$ or $S_n$,
but some relevant constructions make sense for arbitrary permutations $w$ of $\NN$.
For example, the descent sets of
$w$ are defined as
\be\DesR(w) := \{ i : w(i) > w(i+1)\}\quand \DesL(w) :=\DesR(w^{-1}),\ee
and the \emph{Rothe diagram} of $w$ is  
\be D(w) := \{(i,j) \in \NN\times \NN : 0 < j < w(i)\text{ and } 0 < i < w^{-1}(j)\}.\ee
For example, $D(3142) = \{(1,1),(1,2),(3,2)\}$ is the set of $\underline{0}$'s in the permutation matrix $\left[\begin{smallmatrix} \underline{0} & \underline{0} & 1 & 0 \\ 1 & 0 & 0 & 0 \\ 0 & \underline{0} & 0 & 1 \\ 0 & 1 & 0 & 0 \end{smallmatrix} \right]$.
A permutation $w$ of $\NN$ belongs to $ S_\infty$ if and only if $D(w)$ is a finite set, in which case the length function of $S_\infty$
 has value  
$\ell(w) = |D(w)|$.
Given integers $m,n\geq 0$, we define
\be        S^{m,n}_{\infty} := \{w \in S_{\infty} : \DesR(w) \subseteq [m]\text{ and }\DesL(w) \subseteq [n]\}.\ee
This set (which is not a subgroup) has a simple characterization in terms of Rothe diagrams:

\begin{prop}\label{rothe-prop}
One has $S^{m,n}_\infty= \left\{ w \in S_\infty : D(w) \subseteq [m]\times [n]\right\}$.
\end{prop}

We include a short, self-contained proof.

\begin{proof}
Let $w \in S_\infty$.
For $(i,j) \in \NN \times \NN$, consider the set of pairs $(i,y)$ with $j \leq y$
plus all pairs $(x,j)$ with $i \leq x$. The result follows as a straightforward exercise
on observing that $D(w)$ is the complement in $\NN \times \NN$
of the union of these ``hooks'' as $(i,j)$ ranges over all pairs with $j=w(i)$.
\end{proof}

A \emph{partial permutation matrix} is a $(0,1)$-matrix with at most one $1$ in each row and column. The \emph{$m \times n$ partial permutation matrix} of an arbitrary permutation $w$ of $ \NN$
is the $m \times n$  partial permutation matrix
whose nonzero entries occur in the positions $(i,j) \in [m]\times [n]$ with $j  = w(i)$. 
We denote this matrix by $w_{[m][n]}$.
One can check that the (non-injective) map
 $w \mapsto  w_{[m][n]}$ 
restricts to a bijection from $S_\infty^{m,n}$ to the set of all $m \times n$ partial permutation matrices.

\begin{rem}
This notion of partial permutation matrix is ubiquitous in the literature.
Note that the map sending $w \in S_n$ to its $n\times n$ partial permutation matrix
is an antiautomorphism, however.
\end{rem}

\subsection{Involutions}

We write $\I_n := \{ w \in S_n: w^2=1\}$ for the set of involutions in $S_\infty$ fixing all $i >n$,
and $\Ifpf_n$ for the set of elements in $\I_n$ with no fixed points in $[n]$.
The latter set is nonempty only if $n$ is even, in which case it is the $S_n$-conjugacy class
of $(1,2)(3,4)\cdots(n-1,n)$.
Define $\Ifpf_\infty$ to be the set of permutations that are conjugate by an element of $S_\infty$ to the  involution
\be\label{1fpf-eq}
 1_\fpf := (1,2)(3,4)(5,6)(7,8)\cdots,\text{ i.e., the map sending } i \mapsto i - (-1)^i.
 \ee  More explicitly, the set $\Ifpf_\infty$ consists of the fixed-point-free involutions of $\NN$
 that agree with $1_\fpf$ at all sufficiently large inputs---in other words, all products of the form
\[
 z \cdot(2m+1,2m+2)(2m+3,2m+4)(2m+5,2m+6)\cdots
 \] where $z \in \Ifpf_{2m}$ and $m \geq 0$.
 Note that the sets $\Ifpf_n\subset \I_n \subset S_n \subset S_\infty$ are disjoint from $\Ifpf_\infty$.

\begin{defn}
    Suppose $z \in \I_n$ has fixed points $i_1 < i_2< \cdots < i_k$ in $[n]$. The \emph{fixed-point-free standardization} of $z$ is 
    the product $  \FPF_n(z)  := z \cdot z' \cdot z''$ where 
\[ \begin{aligned} z ' &= (i_1,n+1)(i_2,n+2)\cdots(i_k,n+k) \\
z '' &= (n+k+1,n+k+2)(n+k+3,n+k+4)(n+k+5,n+k+6)\cdots.
\end{aligned}
\]
Since $k$ and $n$ have the same parity, $n+k+1$ is odd so $\FPF_n(z) \in \Ifpf_\infty$.
\end{defn}

For example, if $z = (1,4) \in \I_5$ then $\FPF_5(z) = (1,4)(2,6)(3,7)(5,8)(9,10)(11,12)\cdots$.
If $y \in \Ifpf_m$ and $z \in \Ifpf_n$ for $m \leq n$, then $\FPF_m(y) = \FPF_n(z)$
if and only if $z = y s_{m+1}s_{m+3}\cdots s_{n-1}$.
There are more interesting situations in which one can have $\FPF_m(y) = \FPF_n(z)$ 
if $y$ or $z$ have fixed points in $[m]$ and $[n]$.

There are at least two natural ways of characterizing the image of  $\FPF_n : \I_n \to \Ifpf_\infty$.
First, say that a positive integer $i$ is a \emph{visible descent} of $z \in \Ifpf_\infty$ if $z(i+1) < \min\{i,z(i)\}$.

\begin{prop} \label{prop:fpf-injection} 
The map $\FPF_n$ is a bijection from $\I_n$ to the set  of elements in $ \Ifpf_\infty$ with no visible descents greater than $n$.
\end{prop}

\begin{proof} If we write $\FPF_n(z) = (a_1,b_1)\cdots (a_p,b_p)(c_1,n+1)(c_2,n+2)\cdots$ where $a_i, b_i \in [n]$ for each $i$, then $z$ can be recovered as $\FPF_n(z)\cdot (c_1,n+1)(c_2,n+2) \cdots = (a_1,b_1)\cdots (a_p,b_p),$  so $\FPF_n$ is injective.
Next, observe that if $y \in \I_n$ then $\FPF_n(y)$ has no visible descents greater than $n$.

Assume conversely that $z \in \Ifpf_\infty$ has no visible descents greater than $n$.
One has $z(a) \in \{a-1,a+1\}$ for all sufficiently large integers $a$.
If $a$ is maximal with $z(a)+1<a$,
then some $i$ with $z(a) \leq i  <a $ must have 
$z(i+1) < \min\{i,z(i)\}$. 
Therefore if $n\leq z(a) < a$ then $a=z(a) +1$.
Let $i$ be maximal with $z(n+i) \leq n$, or set $i=0$ if no
such integer exists.  Then we must have $z(n+1)<z(n+2)<\dots<z(n+i)\leq n$ since otherwise some element of $ \{n+1,n+2,\dots,2n-1\}$ would 
be a visible descent. In this case $z = \FPF_n(y)$ 
for the element $y \in \I_n$ with
 $y(j) = z(j)$ for integers $j \in [n]\setminus\{z(n+1),z(n+2),\dots,z(n+i)\}$ and $y(j) =j$ for all other $j$.
\end{proof}

\begin{rem}
The image $\FPF_n(\I_n)$ of the map in Proposition~\ref{prop:fpf-injection}
turns out to be the natural skew-symmetric analogue of $S^{m,n}_\infty$ and 
will frequently recur throughout this article. We stick to the slightly cumbersome
notation ``$\FPF_n(\I_n)$'' when referring to this, since anything more concise risks confusion with the set of all fixed-point-free involutions
in $S_n$.
\end{rem}

Next, define the \emph{skew-symmetric Rothe diagram} of an arbitrary involution $z=z^{-1}$ of $\NN$ to be
\[
\SSD(z)  := \{ (i,j) \in \NN\times \NN : 0 < j < z(i)\text{ and } j < i < z(j) \}
.
\] 
Since we assume $z=z^{-1}$, we can also write $\SSD(z) = D(z) \cap \{(i,j) \in \NN \times \NN : i > j\}$.
The set $\SSD(z)$ is denoted as $\hat D_{\FPF}(z)$ in \cite{HMP5} and as $D^{\Sp}(z)$ in \cite{MP2019,MP2019b}.

\begin{prop}\label{ssd-prop}
It holds that $\FPF_n(\I_n) = \left\{ z \in \Ifpf_\infty : \SSD(z) \subseteq [n]\times [n]\right\}$.
\end{prop}

\begin{proof}
By \cite[Lemma 5.2]{HMP5}, the index of the last nonempty row in $\SSD(z)$ is the largest visible descent of $z$,
so this result follows from Proposition~\ref{prop:fpf-injection}.
\end{proof}

\section{Skew-symmetric matrix Schubert varieties}\label{ss-sect}

Here, we give the definitions and then examine the basic properties of our key objects of interest,
which consist of the \emph{skew-symmetric matrix Schubert varieties} and some related ideals.
Throughout, $m$ and $n$ are fixed positive integers and $\KK$ is any algebraically closed field.
We write $\Mat_{m \times n}=\Mat_{m \times n}(\KK)$ for the set of $m\times n$ matrices over $\KK$, and $A^T$ for the transpose of
 $A$.

In this setting, an element $A \in \Mat_{n\times n}$ is \emph{skew-symmetric} if $A^T = - A$ and all diagonal entries of $A$ are zero.
The second condition is redundant if $\ch(\KK)\neq 2$.
Let $\SSM_n = \SSM_n(\KK)$ be the set of skew-symmetric  matrices in $\Mat_{n \times n}$.
Finally, write $\GL_n=\GL_n(\KK)$ for the $n\times n$ general linear group over $\KK$ 
and write $B_n$ for its Borel subgroup of invertible lower triangular matrices.

\subsection{Matrix Schubert varieties }

The main results in this article concern certain varieties of skew-symmetric matrices.
One is led to study such spaces for their analogy with the following classical objects:
 
\begin{defn}
    Given a permutation $w$ of $\NN$, 
    the associated \emph{$m\times n$ matrix Schubert cell} $\openX_{w} $
    and
     \emph{$m\times n$ matrix Schubert variety} $X_{w}$ 
    are the subsets of $\Mat_{m\times n}$ given by
   \[ \begin{aligned}
    \openX_w &:= \{ A \in \Mat_{m\times n} : \text{$\rank A_{[i][j]} = \rank w_{[i][j]}$ for $(i,j) \in [m]\times [n]$}\},
    \\
    X_w &:= \{ A \in \Mat_{m\times n} : \text{$\rank A_{[i][j]} \leq \rank w_{[i][j]}$ for $(i,j) \in[m]\times [n]$}\}.
\end{aligned}
\]
\end{defn}

Note that $\rank w_{[i][j]}$ is just the number of nonzero entries in the $i\times j$ partial permutation matrix of $w$.
The product group $\GL_m \times \GL_n$ acts on $\Mat_{m\times n}$,
by the formula $(g,h) : A \mapsto gAh^T$
and matrix Schubert cells arise as the orbits of this action restricted to $B_m\times B_n$:

\begin{thm}[{See \cite[Ch. 15]{MillerSturmfels}}]
\label{x-thm}
 Let $w$ be a permutation of $\NN$.
Then $\openX_w$ is the $B_m \times B_n$-orbit of the $m\times n$ partial permutation matrix of $w$.
Moreover, $X_w$ is the Zariski closure of $\openX_w$ and is an irreducible variety.
Finally, each $B_m \times B_n$-orbit in $\Mat_{m\times n}$ is equal to $\openX_w$ for a unique $w \in S^{m,n}_\infty$.
\end{thm}

The last part of the theorem shows that we lose no generality in our definition of $X_w$
by restricting the index $w$ to the set $S^{m,n}_\infty$, and this will be our usual practice.
    Matrix Schubert varieties have been extensively studied, for example, in  \cite{FultonEssentialSet,KnutsonMiller,MillerSturmfels}.


We are interested in the following related varieties:
    
\begin{defn}
    Given an involution $z \in \Ifpf_\infty$, 
    the associated \emph{$n\times n$ skew-symmetric matrix Schubert cell} $\openSSX_z$
    and
     \emph{$n\times n$ skew-symmetric matrix Schubert variety} $\SSX_{z}$ 
    are the intersections 
    \begin{align*}
    &\openSSX_z = \{ A \in \SSM_n : \text{$\rank A_{[i][j]} = \rank z_{[i][j]}$ for $i,j \in [n]$}\} = \openX_z \cap \SSM_n,
    \\
    &\SSX_z = \{ A \in \SSM_n : \text{$\rank A_{[i][j]} \leq \rank z_{[i][j]}$ for $i,j \in [n]$}\} = X_z \cap \SSM_n.
\end{align*}
\end{defn}

These definitions would still make sense if $z$ were an arbitrary permutation of $\NN$,
but then it could happen that $\openSSX_z =\emptyset$. We require $z \in \Ifpf_\infty$ to exclude this degenerate case.
As in the classical setting,
 skew-symmetric matrix  Schubert cells arise as the orbits of a certain group action.
 Specifically,
observe that the general linear group $\GL_n$ acts on $\SSM_n$ by $g: A\mapsto gAg^T$.

\begin{rem}
This formula obviously preserves $\SSM_n$ when $\ch(\KK)\neq 2$.
If $\ch(\KK)=2$, then it is still clear that $(gAg^T)^T = -gAg^T$ for all $g \in \GL_n$ and $A \in \SSM_n$,
and for $i \in [n]$ one has $(gAg^T)_{ii}= \sum_{j=1}^n \sum_{k=1}^n g_{ij}g_{ik} A_{jk}$. This is zero since
$A_{jj}=0$ and $A_{jk} = - A_{kj}$, so 
$gAg^T\in \SSM_n$.
\end{rem}

One can check that if $(g,h) \in \GL_n \times \GL_n$
then $gAh^T \in \SSM_n$ for all $A \in \SSM_n$
if and only if $g= \lambda h$ for some $0\neq \lambda \in \KK$.
If $g=\lambda h$ for $0\neq \lambda \in \KK$,
then 
there exists $\mu \in \KK$ with $\mu^2=\lambda$ since $\KK$ is algebraically closed, 
so
we can define
$\tilde g := \mu h$ 
 and then have  $g Ah^T = \tilde g A \tilde g^T$
for all $A \in \SSM_n$.
Thus, the $\GL_n$-orbits in $\SSM_n$ for the action $g: A\mapsto gAg^T$
are the same as the orbits of the subgroup of $\GL_n\times \GL_n$ stabilizing $\SSM_n$.
 
Given any permutation $w$ of $\NN$, let $\ss_{n}(w)$ be the $n \times n$ matrix
 whose entry in position $(i,j)$
 is $1$ if $w(j) = i<j =w(i)$, $-1$ if $w(j) = i>j = w(i)$, and $0$ otherwise; see Example~\ref{ex:SSM-orbits}.


 \begin{thm}[\cite{Cherniavsky}]
 \label{ssx-thm}
  Suppose $z\in \Ifpf_\infty$.
 Then $\openSSX_z$ is the $B_n$-orbit of the skew-symmetric matrix $\ss_{n}(z)$.
 Moreover, $\SSX_z$ is the Zariski closure of $\openSSX_z$ and an irreducible variety.
Finally, each $B_n$-orbit in $\SSM_{n}$ is equal to $\openSSX_z$ for a unique element $z \in \FPF_n(\I_n)$.
\end{thm}

We include a short proof sketch to explain how to extract our statement from
 \cite{Cherniavsky}.

\begin{proof}[Proof sketch]
The discussion in \cite[\S2]{Cherniavsky}
shows that the matrices $\ss_{n}(y)$ for $y \in \I_n$ represent
the distinct $B_n$-orbits in $\SSM_n$.
Cherniavsky works over the complex numbers but his arguments apply equally well over arbitrary fields.
Since $\ss_{n}(y) = \ss_{n}(\FPF_n(y))$ for $y \in \I_n$, it follows that the $B_n$-orbits in $\SSM_n$ are equivalently represented by 
$\ss_{n}(z)$ for $z \in \FPF_n(\I_n)$.
The set $ \openSSX_z $ is 
 the $B_n$-orbit of $\ss_{n}(z)$ by \cite[Prop. 4.2]{Cherniavsky},
and \cite[Prop 4.3]{Cherniavsky} asserts that $\SSX_z$ is the closure of $ \openSSX_z $.
%
%
%
The orbit $\openSSX_z$ is irreducible since it is the image of the irreducible variety $B_n$
under the algebraic map $b \mapsto b\cdot \ss_{n}(z) \cdot b^T$,
so its closure $\SSX_z$ is also irreducible.
\end{proof}

\begin{ex} \label{ex:SSM-orbits}
    Suppose $A = \left[ \begin{smallmatrix} 0 & a & b \\ -a & 0 & c \\ -b & -c & 0 \end{smallmatrix} \right] \in \SSM_3$. If $a \neq 0$, then adding an appropriate multiple of row $1$ to row $3$ and adding the same multiple of column $1$ to column $3$ transforms $A$ into $\left[ \begin{smallmatrix} 0 & a & b' \\ -a & 0 & 0 \\ -b' & 0 & 0 \end{smallmatrix} \right]$; further operations clear out the entries $\pm b'$ and rescale $a$ to $1$. Thus:
    \begin{itemize}
    \item If $a\neq 0$, then $A$ is in the $B_3$-orbit of $\left[ \begin{smallmatrix} 0 & 1 & 0 \\ -1 & 0 & 0 \\ 0 & 0 & 0 \end{smallmatrix} \right] = \ss_{3}((1,2))$.
\end{itemize}
    Similarly:
    \begin{itemize}
        \item If $a = 0\neq b$, then $A$ is in the $B_3$-orbit of $\left[ \begin{smallmatrix} 0 & 0 & 1 \\ 0 & 0 & 0 \\ -1 & 0 & 0 \end{smallmatrix} \right] = \ss_{3}((1,3))$.
        \item If $a = b = 0 \neq c$, then $A$ is in the $B_3$-orbit of $\left[ \begin{smallmatrix} 0 & 0 & 0 \\ 0 & 0 & 1 \\ 0 & -1 & 0 \end{smallmatrix} \right] = \ss_{3}((2,3))$.
        \item If $a = b = c = 0$, then $A$ is in the $B_n$-orbit of $\left[ \begin{smallmatrix} 0 & 0 & 0 \\ 0 & 0 & 0 \\ 0 & 0 & 0 \end{smallmatrix} \right] = \ss_{3}(1)$.
    \end{itemize}
    In each case $A$ belongs to the orbit of $\ss_{3}(y)$ for some $y \in \I_3$; this orbit is $\openSSX_z$ for $z = \FPF_3(y)$.
\end{ex}

\begin{remark}\label{doesnt-matter-remark}
A corollary of Theorems~\ref{x-thm} and \ref{ssx-thm} is that if $v$ is any permutation of $\NN$ and 
$y$ is any element of $\Ifpf_\infty$, 
then 
there are unique elements $w \in S^{m,n}_\infty$ and $z \in \FPF_n(\I_n)$
such that $\rank v_{[i][j]} = \rank w_{[i][j]}$ for all $(i,j) \in [m]\times [n]$
and $\rank y_{[i][j]} = \rank z_{[i][j]}$ for all $(i,j) \in [n]\times [n]$.
This is not hard to show directly.
We will encounter a few other objects described in terms of these rank conditions.
As with $X_w$ and $\SSX_z$,
such objects may be defined for arbitrary permutations or involutions,
but they will always be uniquely indexed by finite sets like $S^{m,n}_\infty$ or $\FPF_n(\I_n)$.
Beyond this subsection we will usually stick to the unique indexing sets.
\end{remark}

Many of the rank conditions defining $X_w$ and $\SSX_z$ are redundant, 
and it is sometimes useful to work with a smaller set of sufficient conditions.
The \emph{essential set} $\Ess(D)$ of $D \subseteq \NN \times \NN$ is the subset of pairs $(i,j) \in D$ such that $(i+1,j) \notin D$ and $ (i,j+1) \notin D$. 
\begin{prop}[{\cite{FultonEssentialSet,MP2019}}]  \label{thm:fpf-essential-set} If $w \in S^{m,n}_\infty$ and $z \in \FPF_n(\I_n)$ then
\[
\begin{aligned}
 X_w&= \{ A \in \Mat_{m\times n} : \text{$\rank A_{[i][j]} \leq \rank w_{[i][j]}$ for $(i,j) \in \Ess(D(w))$}\}, \\
\SSX_z&= \{ A \in \SSM_n : \text{$\rank A_{[i][j]} \leq \rank z_{[i][j]}$ for $(i,j) \in \Ess(\SSD(z))$}\}.
\end{aligned}
\]
\end{prop}

The formula for $X_w$ here is \cite[Lem. 3.10]{FultonEssentialSet}; the formula for $\SSX_z$ is included in \cite[Prop. 2.16]{MP2019}.

\subsection{Bruhat orders}\label{bruhat-sect}

Recall that $\ell(w) = |D(w)|$ is the length of $w \in S_\infty$.
The \emph{Bruhat order} on $S_\infty$ is the transitive closure $<$ of the relation with $v <w$ if $w = v(i,j)$ for 
positive integers $ i < j$ and $\ell(w) = \ell(v)+1$.
This order is related to matrix Schubert varieties by the following well-known proposition.

\begin{prop}\label{prop:bruhat-order}
For all $v,w \in S^{m,n}_\infty$, the following properties are equivalent:
\begin{itemize}
\item[(a)] One has $v \geq w$ in the Bruhat order on $S_\infty$.
\item[(b)] It holds that $\rank v_{[i][j]} \leq \rank w_{[i][j]}$ for all $(i,j) \in [m]\times [n]$.
\item[(c)] We have the containment $X_v\subseteq X_w$ of matrix Schubert varieties.
\end{itemize}
\end{prop}

\begin{proof}[Proof sketch]
The equivalence of (a) and (b) follow as a straightforward consequence of the \emph{tableau criterion} for Bruhat order 
on symmetric groups \cite[Thm. 2.6.3.]{BjornerBrenti}. Properties (b) and (c) are equivalent since the matrix Schubert cells $\openX_v$ and $\openX_w$ are nonempty.
\end{proof}

There is a skew-symmetric analogue of this proposition.
This will involve a version of Bruhat order for the set $\Ifpf_\infty$, which we introduce 
after the following technical lemma.
Here, for any permutation $w$ of $\NN$ and integers $i,j \in \NN$, let  $\r_w(i,j) := \rank w_{[i][j]}$.

\begin{lem}\label{fpf-bruhat-pre-lem} Suppose $y,z \in \FPF_n(\I_n)$ 
are involutions with $\r_y(i,j) \leq \r_z(i,j)$
for all $i,j \in [n]$. Then $\r_y(i,j) \leq \r_z(i,j)$ for all $i,j \in [2n]$.
\end{lem}

\begin{proof}
Consider the numbers $b_i := \r_z(i,n)$ for $i=0,1,2,\dots,n$.
Let $i_1 < i_2 < \dots < i_k$ be the indices $i \in [n]$ with $b_i = b_{i-1}$ where $b_0:=0$, so that $b_n = n -k $.
Then the permutation matrix of $z$ has an entry equal to $1$ in positions $(i_j,n+j)$ and $(n+j,i_j)$ for each $j \in [k]$
and in positions $(i,i+1)$ and $(i+1,i)$ for each (odd) index $i \in \{n + k+1, n+k+3,\dots,2n-1\}$,
and zeros in all other positions $(i,j) \in ([2n]\times [2n]) \setminus ([n]\times [n])$.
It follows that 
\begin{itemize}
\item $\r_z(i,n+j)= \min\{ i, b_i+j\}$ for $i,j \in [n]$,
\item $\r_z(n+i,n+j) = b_n + i + \min\{j, k\}$ for $i \in [k]$ and $j \in [n]$,
\item $\r_z(n+i+k,n+j+k) = n + k + \r_{1_\fpf}(i,j)$ if $i,j \in [n-k]$, with $1_\fpf \in \Ifpf_\infty$ as in \eqref{1fpf-eq}.
\end{itemize}
This tells us how to compute $\r_z(i,j)$ whenever $\max\{i,n+1\} \leq j \leq 2n$.

Now let $a_0=0$ and $a_i = \r_y(i,n)$ for $i \in [n]$, and suppose $a_i = a_{i-1}$ for exactly $l$ indices $i \in [n]$.
Since $a_n = n-l \leq n-k = b_n$, we have $k \leq l$.
The itemized formulas above immediately imply that $\r_y(i,n+j) \leq \r_z(i,n+j)$ for all $i,j \in [n]$ 
and that $\r_y(n+i,n+j) \leq \r_z(n+i,n+j)$ for all $i,j \in [k]$.
Checking that $\r_y(i,j) \leq \r_z(i,j)$ for the remaining values of $i,j \in [2n]$ is straightforward using the above formulas.

Here are the explicit details.
When $i \in [k]$ and $j \in [n]\setminus[k]$ we have
$ 
\r_y(n+i,n+j) =  a_n + i + \min\{j,l\}\leq n + i = \r_z(n+i,n+j)
$
since $a_n = n-l$ and $b_n=n-k$. Also, it is easy to see that 
$\r_y(n+i+l,n+j+l) = \r_z(n+i+l,n+j+l)= n + l + \r_{1_\fpf}(i,j)$ for $i,j \in [n-l]$.
It remains to check that 
$
\r_y(n+i+k,n+j+k) 
\leq \r_{1_\fpf}(i,j)= \r_z(n+i+k,n+j+k)
$
for $i \in [l-k]$ and $j \in [n-k]$ with $i\leq j$. Since in this regime we have
\[
\begin{aligned}
\r_y(n+i+k,n+j+k) &=n-l + i+k + \min\{j+k,l\}, \text{ and }\\
 \r_z(n+i+k,n+j+k) &= n + k + \r_{1_\fpf}(i,j),
\end{aligned}
\]
the desired inequality is equivalent to 
$i + \min\{j-(l-k),0\}   \leq \r_{1_\fpf}(i,j).
$
This holds when $i \in [l-k]$ and $j \in [n-k]$ and $i\leq j$,
since then $\r_{1_\fpf}(i,j)=i$ unless $i=j$ is odd, in which case $r_{1_\fpf}(i,j) = i-1$ and $j < l-k$.
Combining these observations confirms that $\r_y(i,j) \leq \r_z(i,j)$ for all $ 1 \leq i \leq j \leq 2n$,
so by symmetry the same inequality holds for all $i,j \in [2n]$.
\end{proof}

 Define 
$\ellfpf(z) = |\SSD(z)|$ for any involution $z$ of $\NN$. This is guaranteed to be a finite quantity if $z \in S_\infty$ or $z \in \Ifpf_\infty$.
One can show that $\ellfpf$ restricts to the unique map $ \Ifpf_\infty \to \{0,1,2,\dots\}$ with $\ellfpf(1_\fpf) = 0$
and $\ellfpf(szs) =\ellfpf(z) + 1$ whenever $s=(i,i+1)$ and $z(i) < z(i+1)$; see  \cite[\S2.3 and \S4]{HMP5}.

The \emph{Bruhat order} on $\Ifpf_\infty$ 
is the partial order $<$
given as the transitive closure of the relation with $z < tzt$ if $t=(i,j) \in S_\infty$ 
is a transposition and $\ellfpf(tzt) = \ellfpf(z) + 1$.
For our applications, all we need to know about this order is summarized by the following lemma.
\begin{lem}[{\cite[\S4]{HMP3}}]
\label{ss-bruhat-lem}
 Fix positive integers $i<j$. The following properties then hold:
\begin{itemize}
\item[(P1)] If $z \in \Ifpf_\infty$ then we have $z < (i,j) z(i,j)$ if and only if
 $z(i) < z(j)$.
\item[(P2)]  If $z \in \Ifpf_\infty$ then we have
$\ellfpf((i,j)z(i,j)) = \ellfpf(z)+1$ only if no integer $e$ exists with $i<e<j$ and $z(i)<z(e) < z(j)$.

\item[(P3)] If $y,z\in \Ifpf_n\subset S_\infty$ for some $n \in 2\NN$, then 
$\FPF_n(y) \leq \FPF_n(z)$ in the Bruhat order on $\Ifpf_\infty$
if and only if $y \leq z$ in the Bruhat order on $S_\infty$.
\end{itemize}
\end{lem}

\begin{rem}
In this lemma, property (P1) follows from \cite[Cor. 4.10]{HMP3} and property (P2) 
follows from \cite[Prop. 4.9]{HMP3}.
Property (P3) follows from  \cite[Thm. 4.6]{HMP3}, which summarizes some specific consequences 
of the general theory of quasiparabolic sets developed in \cite{RainsVazirani}. 
\end{rem}

Here is our skew-symmetric analogue of Proposition~\ref{prop:bruhat-order}:

\begin{prop}\label{ss-bruhat-prop}
For all $y,z \in \FPF_n(\I_n)$, the following properties are equivalent:
\begin{itemize}
\item[(a)] One has $y \geq z$ in the Bruhat order on $\Ifpf_\infty$.
\item[(b)] It holds that $\rank y_{[i][j]} \leq \rank z_{[i][j]}$ for all $(i,j) \in [n]\times [n]$.
\item[(c)] We have the containment $\SSX_y\subseteq \SSX_z$ of skew-symmetric matrix Schubert varieties.
\end{itemize}
\end{prop}

\begin{proof} 
We have $y = \FPF_{2n}(y')$ and $z = \FPF_{2n}(z')$ for some permutations $y',z' \in \Ifpf_{2n}$,
and
property (P3) in Lemma~\ref{ss-bruhat-lem}
asserts 
that  $y \geq z$ in the Bruhat order on $\Ifpf_\infty$ if and only if $y' \geq z'$ in the Bruhat order on $S_{2n}$.
The latter holds if and only if $\r_{y'}(i,j) \leq \r_{z'}(i,j)$ for all $i,j \in [2n]$ by Proposition~\ref{prop:bruhat-order}.
By Lemma~\ref{fpf-bruhat-pre-lem}, this is equivalent to having
$\r_{y}(i,j) \leq \r_{z}(i,j)$ for all $i,j \in [n]$, as $\r_y(i,j) = \r_{y'}(i,j)$ and $\r_z(i,j) = \r_{z'}(i,j)$ for all $i,j \in [2n]$.
This shows that  properties (a) and (b) are equivalent.
Properties (b) and (c) are obviously equivalent since the cells $\openSSX_y$ and $\openSSX_z$ are nonempty;
this equivalence is also noted as \cite[Cor. 4.4]{Cherniavsky}.
 \end{proof}

\subsection{Pfaffian generators for prime ideals}

As in the introduction, we identify the coordinate ring $\KK[\SSM_n]$ with $\KK[u_{ij} : i,j\in [n], i>j]$
where $u_{ij}$ represents the function $A \mapsto A_{ij}$.
If $A$ is a matrix and $I$ and $J$ are subsets of indices,
then we write $A_{IJ}:= [A_{ij}]_{(i,j) \in I \times J}$ for the corresponding $|I|\times |J|$ submatrix. 
We often apply this notation to the 
$n\times n$ skew-symmetric matrix of variables $\SSU$ with entries defined by
\begin{equation}
\label{SSU-eq}
  \SSU_{ij} = \begin{cases} -u_{ji} & \text{if $i < j$} \\ u_{ij} & \text{if $i > j$} \\ 0 & \text{if $i = j$.} \end{cases}
\end{equation}

If $A$ is a matrix then $\rank A \leq r$ if and only if all size $r+1$ minors of $A$ vanish.
Hence, by Proposition~\ref{thm:fpf-essential-set},  $\SSX_z$ is the zero locus of the ideal in 
$\KK[u_{ji} : i,j\in [n], i>j]$ 
 generated by all size $\rank z_{[i][j]} + 1$ minors of $\SSU_{[i][j]}$ for  
   $(i,j) \in \Ess(\SSD(z))$.
 This ideal is often \emph{not} prime, however.
 
\begin{ex}\label{not-prime-ex}
Take $n=6$ and  let $z$ be the image of $ (1,2)(3,5)(4,6) \in \Ifpf_6$ under $\FPF_6$. Then
\[
\Ess(\SSD(z)) = \SSD(z) = \{(4,3)\}  =
\left\{\begin{smallmatrix}
 \cdot& \cdot & \cdot &\cdot  &  \cdot&\cdot  \\
\cdot &   \cdot&\cdot  &  \cdot& \cdot &\cdot  \\
\cdot & \cdot &  \cdot&\cdot  & \cdot & \cdot \\
\cdot & \cdot & \square & \cdot & \cdot &\cdot  \\
\cdot & \cdot & \cdot & \cdot& \cdot &\cdot  \\
\cdot & \cdot & \cdot & \cdot & \cdot & \cdot \\
\end{smallmatrix}\right\}
=
\left\{\begin{smallmatrix}
 \square& 1 & \cdot &\cdot  &  \cdot&\cdot  \\
1 &   \cdot&\cdot  &  \cdot& \cdot &\cdot  \\
\cdot & \cdot &  \square&\square  & 1 & \cdot \\
\cdot & \cdot & \square & \square & \cdot &1  \\
\cdot & \cdot & 1 & \cdot & \cdot &\cdot  \\
\cdot & \cdot & \cdot & 1 & \cdot & \cdot \\
\end{smallmatrix}\right\}
\cap 
\left\{\begin{smallmatrix}
\cdot& \cdot & \cdot &\cdot  &  \cdot & \cdot  \\
\square& \cdot & \cdot &\cdot  &  \cdot & \cdot  \\
\square& \square & \cdot &\cdot  &  \cdot & \cdot  \\
\square& \square & \square &\cdot  &  \cdot & \cdot  \\
\square& \square & \square &\square  &  \cdot & \cdot  \\
\square& \square & \square &\square  &  \square & \cdot  
\end{smallmatrix}\right\}
\]
so 
$\SSX_z = \{ A \in \SSM_6 : \rank A_{[4][3]} \leq 2\}$ by Proposition~\ref{thm:fpf-essential-set}.
The ideal described above is therefore generated by the four $3 \times 3$ minors in $\SSU_{[4][3]}$, one of which is 
    \begin{equation*}
    \begin{aligned}
        \det ( \SSU_{\{1,2,4\},\{1,2,3\}}) =  \det \left[ \begin{smallmatrix}
            0 & -u_{21} & -u_{31}\\
            u_{21} & 0  & -u_{32}\\
            u_{41} & u_{42} & u_{43} \end{smallmatrix}\right] 
            &
            = u_{21}(u_{21}u_{43} - u_{31}u_{42} + u_{32}u_{41}).
            \end{aligned}
        \end{equation*}
        This ideal, being  generated by homogeneous polynomials of degree $3$, cannot contain either of the factors $u_{21}$ or $u_{21}u_{43} - u_{31}u_{42} + u_{32}u_{41}$, and therefore is not prime.

\end{ex}

In this section, we identify a different set of natural generators for an ideal $\SSI_z$ whose zero locus is $\SSX_z$.
Later,  in Section~\ref{tra-sect}, we will show that $\SSI_z$ is actually the prime ideal $I(\SSX_z)$ of the variety $\SSX_z$.
The key idea in our construction is to replace minors of matrices by Pfaffians.
Recall that the \emph{Pfaffian} of a skew-symmetric $n\times n$ matrix $A$ is
\be\label{pf-def-eq}
 \pf (A)
= \sum_{z \in \Ifpf_n} (-1)^{\ellfpf(z)} \prod_{z(i)<i \in [n] } A_{z(i),i}
\ee
where once again $\ellfpf(z) = |\SSD(z)|$. For example, we have
$
\pf( \SSU_{[2][2]}) = \pf \left[\begin{smallmatrix} 0 & - u_{21} \\ u_{21} & 0 \end{smallmatrix}\right] = -u_{21}
$.
If $n$ is odd then
the outer summation in \eqref{pf-def-eq} is empty so $\pf(A) = 0$. This is consistent with the well-known fact that
$ \pf (A)^2 = \det (A)$, which is zero if $A$ is skew-symmetric of odd size.

\begin{defn}\label{ssi-def} Given $z \in \FPF_n(\I_n)$, let $\SSI_z$ be the ideal in 
$\KK[\SSM_n] = \KK[u_{ij} : i,j \in [n], i>j]$ generated by the Pfaffians $\pf (\SSU_{RR})$ for all nonempty sets $R \subseteq [n]$ of even size 
for which there exist 
 indices $i,j \in [n]$ with $i\geq j$ such that $R \subseteq [i]$ and $|R \cap [j]| > \rank z_{[i][j]}$.  \end{defn}

We discuss an example where we can compute directly that $\SSI_z = I(\SSX_z)$.

\begin{ex} Suppose $n=6$ and $z$ is the image of $(1,4)(2,6)(3,5) \in \Ifpf_6$ under $\FPF_6$. Then
\begin{equation*}
\left[  
        \rank z_{[i][j]} 
   \right]_{1\leq i,j \leq 6}
    =
\left[    \begin{array}{cccccc}
        0 & 0 & 0 & 1 & 1 & 1 \\
        0 & 0 & 0 & 1 & 1 & 2 \\
        0 & \boxed{0} & 0 & 1 & 2 & 3 \\
        1 & 1 & 1 & 2 & 3 & 4 \\
        1 & \boxed{1} & 2 & 3 & 4 & 5 \\
        1 & 2 & 3 & 4 & 5 & 6
    \end{array}\right]
\end{equation*}
where the boxed cells correspond to the positions in $\Ess(\SSD(z))$. The rank condition at position $(i,j) = (3,2)$ says that $\rank A_{[3][2]} = 0$
for all $A \in \SSX_z$, so 
$\SSX_z$ is contained in the zero locus of the ideal $(u_{21}, u_{31},u_{32}) $ in $ \KK[u_{ij} : 6 \geq i>j \geq 1] $. Thus every matrix $A \in \SSX_z$ has the form
\begin{equation*}
    \begin{bmatrix}
        0 & 0 & 0 & -A_{41} & -A_{51} & -A_{61}\\
        0 & 0 & 0 & -A_{42} & -A_{52} & -A_{62}\\ 
        0 & 0 & 0 & -A_{43} & -A_{53} & -A_{63}\\
        A_{41} & A_{42} & A_{43} & 0 & -A_{54} & -A_{64}\\
        A_{51} & A_{52} & A_{53} & A_{54} & 0 & -A_{65}\\
        A_{61} & A_{62} & A_{63} & A_{64} & A_{65} & 0
    \end{bmatrix}.
\end{equation*}
The rank condition at position $(i,j)=(5,2)$ says that if $A \in \SSX_z$ then $\rank A_{[5][2]} \leq 1$, which holds if and only if $\det \left[\begin{smallmatrix} A_{41} & A_{42} \\ A_{51} & A_{52} \end{smallmatrix} \right] = 0$. We conclude that $\SSX_z$ is the zero locus of the ideal 
$(u_{21}, u_{31}, u_{32}, u_{41}u_{52} - u_{42}u_{51}) $ in $ \KK[u_{ij} : 6 \geq i>j \geq 1].$ This ideal turns out to be prime,
so it is in fact equal to the ideal $I(\SSX_z)$ of the (irreducible) variety $\SSX_z$.

We now describe the ideal $\SSI_z$. The generators of $\SSI_z$ corresponding to position $(i,j) = (3,2)$ 
have the form $\pf(\SSU_{RR})$ for all subsets 
 $R \subseteq \{1,2,3\}$ of size $2$ such that $|R \cap \{1,2\}| > 0$; the last condition is vacuous so $R$ can be any of $\{1,2\}$, $\{2,3\}$, or $\{1,3\}$, giving generators 
 \[\pf \left[ \begin{smallmatrix} 0 & -u_{21} \\ u_{21} & 0 \end{smallmatrix}\right] = -u_{21},\quad \pf \left[ \begin{smallmatrix} 0 & -u_{23} \\ u_{32} & 0 \end{smallmatrix}\right] = -u_{32},
 \quand \pf \left[ \begin{smallmatrix} 0 & -u_{31} \\ u_{31} & 0 \end{smallmatrix}\right] = -u_{31}.\]
The generators of $\SSI_z$ corresponding to position $(i,j) = (5,2)$ 
have the form $\pf(\SSU_{RR})$ for all subsets $R \subseteq \{1,2,3,4,5\}$ of size $2$ or $4$ such that $|R \cap \{1,2\}| > 1$: these are $\{1,2\}$, $\{1,2,3,4\}$, $\{1,2,3,5\}$, and $\{1,2,4,5\}$, giving generators $\pf(\SSU_{\{1,2\},\{1,2\}}) =\pf \left[ \begin{smallmatrix} 0 & -u_{21} \\ u_{21} & 0 \end{smallmatrix}\right] = -u_{21}$ along with
\[
\begin{aligned}
  \pf(\SSU_{\{1,2,3,4\},\{1,2,3,4\}}) &=  \pf \left[\begin{smallmatrix}
        0 & -u_{21} & -u_{31} & -u_{41} \\ 
        u_{21} & 0 & -u_{32} & -u_{42} \\ 
        u_{31} & u_{32} & 0 & -u_{43} \\ 
        u_{41} & u_{42} & u_{43} & 0 \\ 
    \end{smallmatrix}\right] = u_{21}u_{43} - u_{31}u_{42} + u_{41}u_{32},
    \\
  \pf(\SSU_{\{1,2,3,5\},\{1,2,3,5\}}) &=   \pf \left[\begin{smallmatrix}
        0 & -u_{21} & -u_{31} & -u_{51} \\ 
        u_{21} & 0 & -u_{32} & -u_{52} \\ 
        u_{31} & u_{32} & 0 & -u_{53} \\ 
        u_{51} & u_{52} & u_{53} & 0 \\ 
    \end{smallmatrix}\right] = u_{21}u_{53} - u_{31}u_{52} + u_{51}u_{32},
    \\
\pf(\SSU_{\{1,2,4,5\},\{1,2,4,5\}}) &=    \pf \left[\begin{smallmatrix}
        0 & -u_{21} & -u_{41} & -u_{51} \\ 
        u_{21} & 0 & -u_{42} & -u_{52} \\ 
        u_{41} & u_{42} & 0 & -u_{54} \\ 
        u_{51} & u_{52} & u_{54} & 0 \\ 
    \end{smallmatrix}\right] = u_{21}u_{54} - u_{41}u_{52} + u_{51}u_{42}.
\end{aligned}
\]
One can check that these Pfaffians are already a generating set for $\SSI_z$.
(In fact, by mimicking the proof of \cite[Lem. 3.10]{FultonEssentialSet}, one can show that $\SSI_z$ is always generated by the 
generators in Definition~\ref{ssi-def} corresponding to just the essential positions $(i,j) \in \Ess(\SSD(z))$, but we will not need this
in any arguments.)
For this example, one can compute directly that
\[ 
{\small
\begin{aligned}
   \SSI_z &= (-u_{21}, -u_{32}, -u_{31}, u_{21}u_{43} - u_{31}u_{42} + u_{41}u_{32}, u_{21}u_{53} - u_{31}u_{52} + u_{51}u_{32}, u_{21}u_{54} - u_{41}u_{52} + u_{51}u_{42})\\
    &=(u_{21}, u_{32}, u_{31}, -u_{41}u_{52} + u_{51}u_{42}) \\&= I(\SSX_z).
\end{aligned}
}
\]
\end{ex}

\begin{remark} \label{rem:pfaffian-ideals} One can show that the Pfaffian ideals considered in \cite{DeNegriSbarra, RaghavanUpadhyay} are 
precisely the ideals $\SSI_z$ 
indexed by all $z \in \FPF_n(\I_n)$ for which the essential set of the skew-symmetric Rothe diagram has the following property:
one can write $\Ess(\SSD(z)) = \{(a_1, b_1),\ldots, (a_k, b_k)\}$ 
such that the 
sequences of numbers $a_i$, $b_i$, and $\rank z_{[a_i][b_i]}$ for $i=1,2,\dots,k$
are respectively decreasing, increasing, and increasing (weakly in each case). Such  conditions appear in certain geometric contexts \cite{AndersonFulton, KnutsonMillerYong}, as well as combinatorially: the permutations $w \in S_\infty$ such that $\Ess(D(w))$ has the above property are exactly the ones that are \emph{vexillary} (i.e., 2143-avoiding). 
Fixed-point-free involutions $z$ satisfying our conditions also arise from geometry \cite{UniversalGraphSchubert}, but are no longer always vexillary permutations.
\end{remark}

To show that $\SSI_z$ is the prime ideal of $\SSX_z$,
we first need to prove that $\SSI_z \subseteq I(\SSX_z)$.
For this, it suffices to check that 
$A \in \SSX_z$ if and only if $f(A) = 0$ for all $f \in \SSI_z$.

It is well-known that if $B$ is the matrix obtained from a skew-symmetric matrix $A$ by multiplying row $i$ and column $i$ by a scalar $c$, then $\pf(B) = c\cdot \pf(A)$. Likewise, if $A^v$, $A^w$, and $A^{v+w}$ are skew-symmetric matrices that are equal outside row and column $i$, 
and it holds that column $i$ contains vectors $v$, $w$, and $v+w$ in the three respective matrices, then $\pf(A^{v+w}) = \pf(A^v) + \pf(A^w)$.
More generally, we have $\pf(XAX^T) = \det(X) \pf(A)$ as long as $X$ and $A$ are square matrices of the same size with $A$ skew-symmetric.

\begin{lem}\label{lem-pf2}
Fix integers $i,j \in [n]$ with $i\geq j$ and $r\geq 0$. If $A \in \SSM_n$
then $\rank A_{[i][j]} \leq r$ if and only if $\pf(A_{RR}) = 0$ for all subsets $\emptyset \neq R \subseteq [i]$ of even size with $|R \cap [j]| > r$. 
\end{lem}

\begin{proof} 
  First assume that $A \in \SSM_n$ is a \emph{monomial matrix}, that is, a product of a permutation matrix and a diagonal matrix.
Suppose $\rank A_{[i][j]} > r$. Then there exist nonzero entries of $A_{[i][j]}$ in some positions $(i_0, j_0), (i_1,j_1),\ldots, (i_r, j_r)$ such that $R = \{i_0,i_1, \ldots, i_r\}$ and $S = \{j_0, j_1,\ldots, j_r\}$ both have size $r+1$. The skew-symmetry of $A$ means that every row and column of $A_{R \cup S, R \cup S}$ contains a nonzero entry. Since $A_{R \cup S, R \cup S}$ is again monomial, this means that $|R \cup S|$ is even and $\pf(A_{R \cup S, R \cup S}) \neq 0$,
but in this case we have $|(R \cup S) \cap [j]| \geq |S \cap [j]| = |S| = r+1 > r$. 
    Conversely, suppose that $\rank A_{[i][j]} \leq r$. Then, since $A$ is a monomial matrix, any selection of more than $r$ columns of $A_{[i][j]}$ must include a zero column. In particular, if $|R \cap [j]| > r$ and $R \subseteq [i]$ then $A_{RR}$ has a zero column, so $\pf(A_{RR}) = 0$.

This proves the lemma when $A \in \SSM_n$ is a monomial matrix. 
We now explain how to reduce the general form of the lemma to the monomial case. 

Say that a set $R$ is \emph{valid} if $\emptyset \neq R \subseteq [i]$ and $|R|$ is even and $|R \cap [j]| > r$.
Let $\mathcal{V}$ be the vector space with a basis given by the collection of all valid sets. For a matrix $A \in \SSM_n$, let
$[A]_R := \pf(A_{RR})$ and define
 $[A] \in \mathcal{V}$ to be formal linear combination $\sum_R [A]_R \cdot R$ over all all valid sets $R$.
 In this notation, what we wish to show is that $\rank A_{[i][j]} \leq r$ if and only if $[A]$ is the zero vector in $\mathcal{V}$.
Using the properties of Pfaffians noted above, we observe that for any indices $n\geq a > b \geq 1$:
\begin{itemize}
    \item[(i)] If $X \in B_n$ is the $n\times n$ diagonal matrix with $v \in \KK$ in position $(a,a)$ and all other diagonal entries  $1$, then $[XAX^T] =  v \sum_{a \in R}  [A]_R \cdot R + \sum_{a \notin R} [A]_R \cdot R$, both sums over valid sets $R$.
    
    \item[(ii)] If $R$ is a valid set with $a \in R$ and $b \notin R$, then the set $(R\setminus\{a\}) \cup \{b\}$ is also valid.
    
    \item[(iii)] If $X=1 + E_{ab} \in B_n$ is the $n\times n$ lower triangular matrix with 1's on the diagonal, with 1 in position $(a,b)$, and with 0's in all other positions, 
then 
\[[XAX^T] = \sum_{a\in R\text{ and } b \notin R} \left([A]_R \pm [A]_{(R \setminus \{a\}) \cup \{b\}}\right) \cdot R + \sum_{a\notin R\text{ or } b \in R} [A]_R \cdot R,\] 
with both sums over valid sets $R$. More precisely, the sign $\pm$  is $ (-1)^{|R \cap \{ a-1, a-2,\dots,b+1\}|}$.
\end{itemize}
Any matrix in $B_n$ is a product of matrices of the form described in item (i) or (iii).
Thus, we deduce that 
for each $X \in B_n$
there exists an invertible linear map $L_X : \mathcal{V} \to \mathcal{V}$ with $L_X([A]) = [XAX^T]$ and $(L_{X})^{-1} = L_{X^{-1}}$. 
Fix $A \in \SSM_n$.
Then for any given $X \in B_n$ we have $[A]=0$ if and only if $[XAX^T] = 0$.
In view of Theorem~\ref{ssx-thm},
there exists $X \in B_n$
such that  $XAX^T$ is the monomial matrix $\ss_{n}(z)$ for some $z \in \FPF_n(\I_n)$.
By the monomial case, this means that
 $[A] =0$ if and only if $\rank (XAX^T)_{[i][j]}   \leq r$.
However, $\rank (XAX^T)_{[i][j]} = \rank A_{[i][j]}$ for all $X \in B_n$.
\end{proof}

We can now briefly derive the following result, which implies that $\SSI_z \subseteq I(\SSX_z)$.

\begin{thm}
\label{ss-zero-locus-thm}
 If $z \in \FPF_n(\I_n)$ then the zero locus of $\SSI_z$ in $\SSM_n$ is $\SSX_z$.
  \end{thm}
\begin{proof} 
The zero locus of $\SSI_z$ in $\SSM_n$ is the zero locus of the generating set in Definition~\ref{ssi-def}.
By Lemma~\ref{lem-pf2}, this is just the set of all $A \in \SSM_n$ satisfying the rank conditions defining $\SSX_z$.
\end{proof}

\subsection{Growth diagrams and rank tables}

We would also like to describe the initial ideal of $\SSI_z$ (with respect to the reverse lexicographic term order), 
which will turn out to be the initial ideal of $I(\SSX_z)$.
In preparation for this, we take a short digression to prove some technical facts about \emph{growth diagrams} and \emph{rank tables}.

Given integer partitions $\mu = (\mu_1 \geq \mu_2 \geq \dots \geq 0)$ and $\lambda = (\lambda_1 \geq \lambda_2 \geq \dots \geq 0)$, 
we write $\mu \subseteq \lambda$ if $\mu_i \leq \lambda_i$ for all $i \in \NN$. When this occurs we write $\lambda/\mu := \{ (i,j) \in \NN\times \NN : \mu_i < j \leq \lambda_i\}$ for 
the corresponding skew shape. Finally, let $\ell(\lambda) := \max(\{0\} \sqcup \{ i \in \NN : \lambda_i\neq 0\})$.

Suppose $X$ is an $m\times n$ matrix with nonnegative integer entries.
As first defined by Fomin \cite{fomin}, the \emph{growth diagram} of $X$ 
is the unique family of partitions $\lambda(i,j)$ 
for
$i \in \{0,1,2,\dots,m\}$ and $j \in \{0,1,2,\dots,n\}$ satisfying the following inductive rules.
First, if $i=0$ or $j=0$ then $\lambda(i,j) := \emptyset$. 
If $i>0$ and $j>0$, then we write
\begin{equation}
\label{lambda-eq} \left[\begin{array}{cc} \rho & \mu \\ \nu & \lambda \end{array} \right] := \left[\begin{array}{ll} \lambda(i-1,j-1) & \lambda(i-1,j) \\ \lambda(i,j-1) & \lambda(i,j) \end{array}\right]
\end{equation}
and define $\lambda = \lambda(i,j)$ by the following algorithm:
\begin{enumerate}
\item[(F0)] Set $\mathrm{CARRY} := X_{ij}$ and $k:=1$.
\item[(F1)] Set $\lambda_k := \max\{\mu_k,\nu_k\} + \mathrm{CARRY}$.
\item[(F2)] If $\lambda_k =0$, then return $\lambda = (\lambda_1,\lambda_2,\dots,\lambda_{k-1})$. Otherwise, set $\mathrm{CARRY} := \min\{\mu_k,\nu_k\} - \rho_k$ and then set $k:=k+1$ and go back to step  (F1).
\end{enumerate}
See \cite[\S4.1]{Krattenthaler} for further discussion and examples.
One can check  directly that in the notation of
\eqref{lambda-eq}, one always has $\rho \subseteq \mu$ and $\rho \subseteq \nu$ 
such that $\mu/\rho$ and $\nu/\rho$ are horizontal strips, so 
\begin{equation}
\label{growth-eq1}
\ell(\mu) -\ell(\rho) \in \{0,1\}\quad\text{and}\quad \ell(\nu) -\ell(\rho) \in \{0,1\}.
\end{equation}
Moreover, it holds that 
\begin{equation}
\label{growth-eq2}
 \ell(\mu) = \ell(\nu) = \ell(\rho) + 1\quad\Rightarrow\quad\ell(\lambda) = \ell(\rho) + 2.
 \end{equation}
An essential nontrivial property of the growth diagram  is 
this analogue of Greene's theorem:

\begin{lem}[{\cite[Thm. 8]{Krattenthaler}}]\label{greene-lem}
For each $(i, j)\in [m] \times [n]$,
the number of parts
$l = \ell(\lambda(i,j)) $ is the length of the longest sequence of nonzero positions $(i_1,j_1),(i_2,j_2),\dots,(i_l,j_l)$ in $X$ 
with $ i \geq i_1  >i_2> \dots > i_l \geq 1$ and $1 \leq j _1<j_2  < \dots < j _l \leq j$. 
\end{lem}

We will need one other lemma related to this general setup.

\begin{lem}\label{odd-columns-lem}
Suppose $m=n$ and $X=X^T$ is symmetric.
Then for each $i \in[n]$, the number of odd columns in the partition $\lambda(i,i)$ is $X_{11} +X_{22} + \dots +X_{ii}$.
\end{lem}

\begin{proof}
Fix $ i \in[n]$.
It is clear that $\lambda(j,k) = \lambda(k,j)$, so we can write 
\[
\left[\begin{array}{cc} \rho & \mu \\ \mu & \lambda \end{array} \right] := \left[\begin{array}{ll} \lambda(i-1,i-1) & \lambda(i-1,i) \\ \lambda(i,i-1) & \lambda(i,i) \end{array}\right]
\]
and in steps (F1) and (F2) to compute $\lambda$ one can replace $\max\{\mu_k,\nu_k\}$ and $\min\{\mu_k,\nu_k\}$ by just $\mu_k$.
Now consider the sequence of partitions $\lambda^0,\lambda^1,\lambda^2,\dots$ where $\lambda^k$ is formed by replacing the first $k$ parts of $\rho$ by $\lambda_1,\ldots, \lambda_k$, so that $\lambda^0 = \rho$.
Let $\OC(\lambda^k)$ be the number of odd columns in $\lambda^k$. If $\alpha$ and $\beta$ are two partitions all of whose parts are equal except that $\alpha_i < \beta_i$ for some $i$, then $\OC(\beta) - \OC(\alpha) = (-1)^{i+1}(\beta_i - \alpha_i)$. Thus 
$\OC(\lambda^1) -\OC(\lambda^0) = X_{ii} + (\mu_1-\rho_1)$ and 
for $k>1$  it follows by examining (F1) and (F2) that
$\OC(\lambda^{k}) - \OC(\lambda^{k-1}) = (-1)^{k+1}((\mu_{k-1} - \rho_{k-1}) + (\mu_{k} - \rho_{k})). $
Since $\lambda = \lambda^k$ for the first value of $k$ with $\rho_k - \mu_k =\rho_{k+1}= \mu_{k+1}=0$,
adding things up gives \[\OC(\lambda) -\OC(\rho)= \sum_{j=1}^k (\OC(\lambda^j) -\OC(\lambda^{j-1})) = X_{ii}\]
and in turn
$\OC(\lambda(i,i)) =  \sum_{j=1}^i (\OC(\lambda(j,j)) -\OC(\lambda(j-1,j-1)))=\sum_{j=1}^iX_{ii}$.
\end{proof}

We define a \emph{rank table} to be a map 
 ${\r} :\{0,1,2,\dots\} \times \{0,1,2,\dots\}  \to\{0,1,2,\dots\} $
satisfying $\r(0,0) = \r(i,0) = \r(0,j) = 0$ for all $i,j \in \NN$.
We have already seen that a permutation $w$ of $\NN$ defines a rank table by the formula 
${\r}_w(i,j) = \rank w_{[i][j]}$ for $i,j>0$. 
Recall that $\r_w(i,j)$ is just the number of nonzero entries in the $i\times j$ partial permutation matrix $w_{[i][j]}$.
 
\begin{lem} \label{lem:rw-characterization} Let ${\r}$ be a rank table.
There exists $w \in S^{m,n}_\infty$ with $\r(i,j) = \r_w(i,j)$ for all $(i,j) \in [m] \times [n]$
 if and only if 
  for each $(i,j) \in [m] \times [n]$ both of the following conditions hold:
    \begin{enumerate}[(a)]
        \item One has ${\r}(i,j) - {\r}(i,j - 1) \in \{0,1\}$ and ${\r}(i,j) - {\r}(i - 1,j) \in \{0,1\}$.
        \item If ${\r}(i-1,j) - {\r}(i-1,j-1) = 1$, then ${\r}(i,j) - {\r}(i,j-1) = 1$.
    \end{enumerate}
\end{lem}

\begin{proof}
First assume $\r$ satisfies the given conditions.
    Say that $(i,j) \in [m] \times [n]$ is \emph{marked} if ${\r}(i,j) - {\r}(i,j-1) = 1$ but ${\r}(i-1,j) - {\r}(i-1,j-1) = 0$. Let  $X$ be the $m \times n$ matrix with $1$ in each marked position and $0$ everywhere else. Condition (b) implies that each column of $X$ contains at most nonzero entry. If some row $i$ contains two marked positions $(i,j), (i,k)$ with $j < k$, and there are no marked positions between these two, then we have        \begin{equation*}
        \r(i,j') - \r(i,j'-1) = \begin{cases}
            \r(i-1,j') - \r(i-1,j'-1) & \text{if $j < j' < k$}\\
            \r(i-1,j') - \r(i-1,j'-1) + 1 & \text{if $j' \in \{j,k\}$}
        \end{cases}
    \end{equation*}
    which implies that $\r(i,k) - \r(i-1,k) \geq 2$. This contradicts condition (a), so each row of $X$ must also contain at most one nonzero entry, and $X$ is an $m \times n$ partial permutation matrix. 
    
    Let $w \in S^{m,n}_\infty$ be the element  whose 
    $m\times n$ partial permutation matrix is $X$.
    Condition (b) implies that ${\r}(i,j) - {\r}(i,j-1) = 1$ if and only if (exactly) one of the positions $(i-1,j), (i-2,j), \ldots, (1,j)$ is marked,
    so  ${\r}(i,j) = \sum_{j'=1}^j ({\r}(i,j') - \r(i,j'-1))$ for $(i,j) \in [m]\times [n]$ is exactly the number of marked positions $(i',j')\in [i]\times [j]$. That is, 
    we have ${\r}(i,j) = {\r}_w(i,j)$ for $(i,j) \in [m]\times [n]$, as desired.
    If conversely there exists $w \in S^{m,n}_\infty$ with $\r(i,j) = \r_w(i,j)$ for all $(i,j) \in [m] \times [n]$,
    then it is straightforward to check that conditions (a) and (b) hold.
\end{proof}

Rank tables also arise from monomials in polynomial rings.
As above, let $\{u_{ij}\}$ be a family of  commuting indeterminates indexed by $(i,j) \in \NN\times \NN$.
If
$A = \{a_0 <a_1< \cdots < a_r\}$
and 
$ B = \{b_0 < b_1<\cdots < b_r\}$ are two sets of $r+1$
positive integers,
then we let \be\label{odot-eq}
A \odot B :=\{ (a_0,b_r),(a_1,b_{r-1}), \dots,(a_r,b_0)\}
\quand
u_{AB} := \prod_{(a,b) \in A\odot B} u_{ab}
.\ee
Given a set $S$ and an integer $r\geq 0$, let $\binom{S}{r}$ denote the set of $r$-element subsets of $S$.
\begin{defn}
For a monomial $M \in \KK[u_{ij} : i,j \in \NN]$,
let $\r_M$ be the rank table in which
${\r}_M(i,j)$ for $i,j>0$ is the largest value $r$ such that  $u_{AB} \mid M$ for some $(A,B) \in \binom{[i]}{r}\times \binom{[j]}{r}$.
\end{defn}
For example, 
    if $M =  u_{21}u_{31}$ then  
    \[
    \left[\begin{array}{llll}
    \r_M(1,1)  & \r_M(1,2)  & \r_M(1,3)   \\
    \r_M(2,1)  & \r_M(2,2)  & \r_M(2,3)   \\
    \r_M(3,1)  & \r_M(3,2)  & \r_M(3,3)   \\
    \end{array}
    \right]
    =
      \left[\begin{array}{llll}
    0  & 0 & 0   \\
    1  & 1  & 1   \\
    1  & 1  & 1   \\
    \end{array}
    \right]
    \]
and we have $\r_M(i,j) = \r_M(3,j)$ for $i\geq 3$ and $\r_M(i,j) = \r_M(i,3)$ for $j\geq 3$.
    
    \begin{prop} \label{prop:rM-equals-rw} If 
$ M \in \KK[u_{ij} : i \in [m], j \in [n]]
$
is a monomial then there exists $w \in S^{m,n}_{\infty}$ such that ${\r}_M(i,j) = {\r}_w(i,j)$ for all $i \in [m]$ and $j \in [n]$. 
 \end{prop}

\begin{proof}
Consider the growth diagram $\{\lambda(i,j)\}_{i \in \{0,1,\dots,m\}, j\in\{0,1,\dots,n\}}$ associated to the $m\times n$ matrix whose entry in position $(i,j)$ is $1$ if $u_{ij}$ divides $M$ and $0$ otherwise.
Lemma~\ref{greene-lem} implies that ${\r}_M(i,j) =\ell(\lambda(i,j)) $ for all $(i,j) \in [m]\times [n]$.
With this interpretation, the conditions in Lemma~\ref{lem:rw-characterization} obviously hold in view of
properties
\eqref{growth-eq1} and \eqref{growth-eq2}.
\end{proof}

We will need
a skew-symmetric version of $\r_M$, whose definition is slightly more involved.
If $A$ and $B$ are  finite subsets of $\NN$ with $|A| = |B|$, then we let
\[
\SSu_{AB} := \begin{cases} 0 &\text{if some $(a,b) \in A\odot B$ has $a=b$} \\ 
 \prod \{ u_{ij} : \text{$i > j$, $(i,j)$ in $A \odot B$ or $B\odot A$}\} &\text{otherwise}.
 \end{cases}
\]
For example, if $A = \{3,4\}$ and $B = \{1,2\}$ then $\SSu_{AB} = \SSu_{A\sqcup B, A\sqcup B}= u_{AB} = u_{41}u_{32}$.
In general, if $\phi : \KK[u_{ij} : i, j \in \NN] \to \KK[u_{ij} : i,j \in \NN, i > j]$  is the ring homomorphism
with 
\[\phi(u_{ij}) = \begin{cases} u_{ij} &\text{if } i > j \\ u_{ji}&\text{if }i<j \\ 0&\text{if }i=j,\end{cases}\] then   $\SSu_{AB} = \rad(\phi(u_{AB}))$ where $\rad(f)$ is the square-free radical of a polynomial $f$.

\begin{defn}
For a monomial $M \in \KK[u_{ij} : i,j \in \NN, i>j]$,
let $\SSr_M$  be the rank table in which
$\SSr_M(i,j)$ for $i,j>0$ is the largest value $r$ such that $\SSu_{AB} \mid M$ for some $(A,B) \in \binom{[i]}{r}\times \binom{[j]}{r}$.
\end{defn}

The rank table  $\SSr_M$ is always \emph{symmetric} in the sense that $\SSr_M(i,j) =\SSr_M(j,i)$ for all $i,j$.

\begin{ex}
    If $M = u_{21}u_{31}$ then 
    \[
    \left[\begin{array}{llll}
    \SSr_M(1,1)  & \SSr_M(1,2)  & \SSr_M(1,3)   \\
    \SSr_M(2,1)  & \SSr_M(2,2)  & \SSr_M(2,3)   \\
    \SSr_M(3,1)  & \SSr_M(3,2)  & \SSr_M(3,3)   \\
    \end{array}
    \right]
    =
      \left[\begin{array}{llll}
    0  & 1 & 2   \\
    1  & 2  & 2   \\
    2  & 2  & 2   \\
    \end{array}
    \right].
\]
 We have $\SSr_M(1,1)=0$ since $0$ does not divide $M$,
and we have $\SSr_M(1,2) = \SSr_M(2,1) = 1$ and $\SSr_M(2,2) = 2$
since $\SSu_{AB} = u_{21}$ for $A = \{1\}$ and $B=\{2\}$ (or vice versa)
as well as for $A=B=\{1,2\}$.
\end{ex}

Let $\dbl :  \KK[u_{ij} : i,j \in \NN] \to \KK[u_{ij} : i,j \in \NN]$ be the ring homomorphism with 
$\dbl(u_{ij}) = u_{ij} u_{ji}$.

\begin{lem}\label{dbl-lem}
Suppose $A,B \in \binom{ [n]}{r}$ for some $r \in[n]$ and $M \in  \KK[u_{ij} :  i,j \in \NN, i > j] $ is a monomial.
Then $ \SSu_{AB} \mid M$ if and only if $u_{AB} \mid \dbl(M)$.
\end{lem}

\begin{proof}
If $\SSu_{AB} = 0$ then we have both $\SSu_{AB} \nmid M$ and $u_{AB} \nmid \dbl(M)$, since no variable of the form $u_{ii}$ divides $\dbl(M)$. Assume $\SSu_{AB} \neq 0$.
Let 
$u'_{AB} = \prod_{(i,j) \in A\odot B, i> j} u_{ij}$
and
$u''_{AB} = \prod_{(i,j) \in A \odot B, i<j} u_{ij}$ so that 
$u_{AB} = u'_{AB}\cdot  u''_{AB}$ and $ \SSu_{AB} = \rad( u'_{AB} \cdot \phi(u''_{AB})).$
Write $N$ for the image of $M$ under the ring homomorphism sending each $u_{ij} \mapsto u_{ji}$,
so that $\dbl(M)= M N$.
Clearly $\SSu_{AB} \mid M$  if and only if $u'_{AB} \mid M$ and $u''_{AB}\mid N$,
which holds if and only if 
 $u_{AB} \mid \dbl(M)$.
\end{proof}

Our main use of the previous lemma is to prove this analogue of Proposition~\ref{prop:rM-equals-rw}.

\begin{prop} \label{ss-prop:rM-equals-rw} 
If 
$
M \in \KK[u_{ij} : i,j \in \NN, i>j]
$ is a monomial then there exists $z \in \FPF_n(\I_n)$ such that $\SSr_M(i,j) = {\r}_z(i,j)$ for all $i,j \in [n]$.
\end{prop}

\begin{proof}
By Proposition~\ref{prop:rM-equals-rw},
there exists an element $w \in S^{n,n}_\infty$ with ${\r}_{\dbl(M)}(i,j) = {\r}_w(i,j)$ for all
$i,j \in [n]$,
and we have $\SSr_M(i,j) = \r_{\dbl(M)}(i,j)= {\r}_w(i,j)$ for all $i,j \in [n]$ by Lemma~\ref{dbl-lem}. Since the rank table ${\r}_{\dbl(M)}$ is symmetric,
we may assume that $ w=w^{-1}$.
We claim that $w(i) \neq i$ for $i \in [n]$.
Indeed, we can only have $w(i)=i$ for some $i \in [n]$
if 
\begin{equation}
\label{impossible-eq}
{\r}_{w}(i-1,i-1) = {\r}_{w}(i,i-1) = {\r}_{w}(i-1,i) < {\r}_{w}(i,i).
\end{equation} 
To exclude this possibility, consider the symmetric $n\times n$ matrix $Y$ whose entry in position $(i,j)$ is $1$ if $u_{ij}$ or $u_{ji}$ divides $M$, and $0$ otherwise. If $\{\lambda(i,j)\}_{i \in \{0,1,\dots,n\}, j\in\{0,1,\dots,n\}}$ is the growth diagram associated to $Y$,
then we have ${\r}_{w}(i,j) = \r_{\dbl(M)}(i,j)=\ell(\lambda(i,j))$ for all $i,j \in [n]\times [n]$
by the argument in the proof of Proposition~\ref{prop:rM-equals-rw}. Since $Y$ is symmetric with all zeros on the main diagonal,
Lemma~\ref{odd-columns-lem} implies that ${\r}_{w}(i,i)$ is even for all $i \in [n]$.
As adjacent entries in ${\r}_w$ differ by at most one, we conclude that \eqref{impossible-eq} cannot occur.
Now 
define $y \in \I_n$ to have $y(i) = w(i)$ if it holds that $\{i,w(i)\}\subseteq [n]$
and $y(i)=i$ otherwise. As $w$ has no fixed points in $[n]$,
the element $z = \FPF_n(y)$ satisfies
$\r_z(i,j) = \r_w(i,j) = \SSr_M(i,j)$ for all $i,j \in [n]$.
%
\end{proof}

\subsection{Monomial generators for initial ideals}\label{mon-sect}

We can now describe explicitly a monomial ideal in $\KK[u_{ij} :i,j \in [n], i >j]$
that will turn out to be  the initial ideal $\init(\SSI_z)$ of $\SSI_z$ under the reverse lexicographic term order.

\begin{defn}\label{ssj-def}
For a symmetric rank table $\r$, let $\SSJ_{\r}$ be the ideal 
in $ \KK[u_{ij} : i,j \in [n], i > j]$
generated by all monomials of the form $\SSu_{AB}$,
where  $(A, B) \in \binom{[i]}{q}\times \binom{[j]}{q}$
for some $(i,j) \in [n]\times [n]$ with $i\geq j$ and $q = \r(i,j)+1$.
If $z \in \FPF_n(\I_n)$, then we define $\SSJ_z := \SSJ_{{\r}_z}$.
\end{defn}

It can be shown that in the definition of $\SSJ_z$ above, one only needs those generators corresponding to $(i,j) \in \Ess(\SSD(z))$. We will not prove this fact since we do not need it for any arguments, but we do assume it in the next example.

\begin{ex} \label{ex:SSI-generators}
Take $n=6$ and  let $z$ be the image of $ (1,2)(3,6)(4,5) \in \Ifpf_6$ under $\FPF_6$. Then
\[
\SSD(z) = \{(4,3),(5,3)\}  =
\left\{\begin{smallmatrix}
 \cdot& \cdot & \cdot &\cdot  &  \cdot&\cdot  \\
\cdot &   \cdot&\cdot  &  \cdot& \cdot &\cdot  \\
\cdot & \cdot &  \cdot&\cdot  & \cdot & \cdot \\
\cdot & \cdot & \square & \cdot & \cdot &\cdot  \\
\cdot & \cdot & \square & \cdot& \cdot &\cdot  \\
\cdot & \cdot & \cdot & \cdot & \cdot & \cdot \\
\end{smallmatrix}\right\}
=
\left\{\begin{smallmatrix}
 \square& 1 & \cdot &\cdot  &  \cdot&\cdot  \\
1 &   \cdot&\cdot  &  \cdot& \cdot &\cdot  \\
\cdot & \cdot &  \square&\square  & \square & 1 \\
\cdot & \cdot & \square & \square & 1 &\cdot  \\
\cdot & \cdot & \square & 1 & \cdot &\cdot  \\
\cdot & \cdot & 1 & \cdot & \cdot & \cdot \\
\end{smallmatrix}\right\}
\cap 
\left\{\begin{smallmatrix}
\cdot& \cdot & \cdot &\cdot  &  \cdot & \cdot  \\
\square& \cdot & \cdot &\cdot  &  \cdot & \cdot  \\
\square& \square & \cdot &\cdot  &  \cdot & \cdot  \\
\square& \square & \square &\cdot  &  \cdot & \cdot  \\
\square& \square & \square &\square  &  \cdot & \cdot  \\
\square& \square & \square &\square  &  \square & \cdot  
\end{smallmatrix}\right\}
\]
so $\Ess(\SSD(z)) = \{(5,3)\}$ and 
$\SSX_z = \{ A \in \SSM_6 : \rank A_{[5][3]} \leq 2\}$ by Proposition~\ref{thm:fpf-essential-set}. The ideal $\SSJ_z$ is generated by
$\SSu_{AB}$ where $A$ ranges over the ten $3$-element subsets of $\{1,2,3,4,5\}$ and $B=\{1,2,3\}$, the unique $3$-element subset of $\{1,2,3\}$. The
relevant monomials 
 are listed below:
    \begin{equation*}
        \begin{array}{r|r|r|r}
            A & u_{AB} & \phi(u_{AB}) &  \SSu_{AB}\\
            \hline
            \{1,2,3\} & u_{13}u_{22}u_{31} & 0  & 0\\
            \{1,2,4\} & u_{13}u_{22}u_{41} & 0 & 0\\
            \{1,2,5\} & u_{13}u_{22}u_{51} & 0 & 0\\
            \{1,3,4\} & u_{13}u_{32}u_{41} & u_{31}u_{32}u_{41} & u_{31}u_{32}u_{41} \\
            \{1,3,5\} & u_{13}u_{32}u_{51} & u_{31}u_{32}u_{51} & u_{31}u_{32}u_{51} \\

            \{1,4,5\} & u_{13}u_{42}u_{51} & u_{31}u_{42}u_{51} & u_{31}u_{42}u_{51} \\
            \{2,3,4\} & u_{23}u_{32}u_{41} & u_{32}^2 u_{41} & u_{32}u_{41}\\
            \{2,3,5\} & u_{23}u_{32}u_{51} & u_{32}^2 u_{51} & u_{32}u_{51}\\
            \{2,4,5\} & u_{23}u_{42}u_{51} & u_{32}u_{42}u_{51} & u_{32}u_{42}u_{51}\\
            \{3,4,5\} & u_{33}u_{42}u_{51} & 0 & 0
        \end{array}
    \end{equation*}
 In particular,   $\SSJ_z=(u_{32}u_{41}, u_{32}u_{51}, u_{31}u_{42}u_{51}) \subseteq \KK[u_{ij}: i,j \in [6], i>j]$ is the ideal generated by the righthand column.
The ideal $\SSI_z$ in this case is generated by the Pfaffians of the submatrices of $\SSU_6$ with row and column sets $\{1,2,3,4\}$ and $\{1,2,3,5\}$, which are the polynomials
    \begin{equation*}
        f := u_{32}u_{41} - u_{31}u_{42} + u_{21}u_{43} \quad \text{and} \quad g := u_{32}u_{51} - u_{31}u_{52} + u_{21}u_{53}.
    \end{equation*}
    The set $\{f,g\}$ is \emph{not} a Gr\"obner basis under the reverse lexicographic term order, as the ideal $(\init(f), \init(g)) = (u_{32}u_{41}, u_{32}u_{51})$ does not contain
    $u_{31}u_{42}u_{51} =\init( u_{41}g  - u_{51}f)  \in \init(\SSI_z).$
It does hold that $u_{31}u_{42}u_{51} \in \SSJ_z$ and one can check directly for this example that $\SSJ_z = \init(\SSI_z)$.
\end{ex}

We state three basic properties of the ideals $\SSJ_\r$ 
before proving a more substantial identity. 

\begin{prop} \label{m-ss-prop}
If $M$ is a monomial in $ \KK[u_{ij} : i,j \in [n], i > j]$
and $\r = \SSr_M$, then $M \notin \SSJ_{\r}$. 
\end{prop}

\begin{proof} 
The monomials in $\SSJ_{{\r}}$ are those divisible by some $\SSu_{AB}$ where $A \subseteq [i]$, $B \subseteq [j]$ and $|A| = |B| = {\r}(i,j) + 1$ for some $i,j$. All such monomials dividing $M$ have degree at most ${\r}(i,j)$. 
\end{proof}

\begin{prop}\label{ssj-bruhat-prop}
  If $y,z \in \FPF_n(\I_n)$ and $y \leq z$, then $\SSJ_y \subseteq \SSJ_z$. 
  \end{prop}

    \begin{proof}
If $y\leq z$ then $\r_y(i,j) \geq \r_z(i,j)$ for all $i,j \in [n]$ by Proposition~\ref{ss-bruhat-prop}, and 
it is clear from the definitions that  this implies that  $\SSJ_y \subseteq \SSJ_z$.
    \end{proof}
    
    Given rank tables $\r$ and $\s$, let $\min(\r,\s)$ be the rank table 
    mapping $(i,j) \mapsto \min\{\r(i,j),\s(i,j)\}$.
    
    \begin{prop} \label{prop:min-to-sum-ss} 
If ${\r}$ and ${\s}$ are symmetric rank tables then $\SSJ_{\min({\r}, {\s})} = \SSJ_{\r} + \SSJ_{{\s}}$. 
\end{prop}

    \begin{proof} It is clear that $\SSJ_{\r}$ and $ \SSJ_{{\s}}$ are both contained in $ \SSJ_{\min({\r}, {\s})}$,
    so $ \SSJ_{\r} + \SSJ_{{\s}} \subseteq \SSJ_{\min({\r}, {\s})} $.
     Conversely, suppose $M$ is a monomial in $\SSJ_{\min({\r}, {\s})}$. Then $M$ is divisible by some $\SSu_{AB}$ where $A \subseteq [i], B \subseteq [j]$ and $|A| = |B| = \min({\r}(i,j), {\s}(i,j))+1$ for some $i$ and $j$. But this means that $M$ is in either $\SSJ_{{\s}}$ or $\SSJ_{{\r}}$ depending on whether ${\r}(i,j)$ exceeds ${\s}(i,j)$ or not.
    \end{proof}

We will need the following theorem to prove the key lemma in Section~\ref{key-lemma-sect}.

\begin{thm}\label{thm:bruhat-minimal-intersection-ss}
 Let ${\r}$ be a symmetric rank table. Then $\SSJ_{\r} = \bigcap_z \SSJ_z$ where $z$ runs over the 
elements of $\FPF_n(\I_n)$ that are  Bruhat-minimal among those with ${\r}_z(i,j) \leq {\r}(i,j)$ for all $i,j \in [n]$.
\end{thm}

\begin{proof} 
It is clear that $\SSJ_{\r}$ is contained in $\SSJ_{\s}$ for any rank table $\s$ with $\s(i,j) \leq \r(i,j)$
 for all $i,j \in [n]$.
Conversely, suppose $M$ is a monomial in $\KK[u_{ij} : i,j \in [n],i>j]$ that is not contained in $\SSJ_{\r}$.  
We must have $\SSr_M(i,j)=\SSr_M(j,i) \leq {\r}(i,j)= {\r}(j,i)$ for each $(i,j) \in [n]\times [n]$ with $i\geq j$,
since otherwise $M$ would be divisible by a monomial $\SSu_{AB}$ with 
$(A,B) \in \binom{[i]}{r} \times \binom{[j]}{r}$ for $r = {\r}(i,j) + 1$, which would imply that $M \in \SSJ_{\r}$. 
By Propositions~\ref{ss-prop:rM-equals-rw} and \ref{m-ss-prop}, we have $\SSr_M = {\r}_z$ for some $z \in \FPF_n(\I_n)$
and  $M \notin \SSJ_z$. This means that $M$ is also not contained in the intersection given in the theorem statement.

We conclude that $\SSJ_{\r} = \bigcap_z \SSJ_z$ 
where $z$ runs over all (not necessarily 
 Bruhat-minimal) elements of $\FPF_n(\I_n)$ satisfying ${\r}_z(i,j) \leq {\r}(i,j)$ for all $i,j \in [n]$.
 By Proposition~\ref{ssj-bruhat-prop}, this intersection is unchanged if we require $z$ to be Bruhat-minimal
 with the desired property. 
 \end{proof}

\subsection{Initial terms of Pfaffians}

We continue to write $\init( \SSI_z)$ for the initial ideal of $\SSI_z$ relative to the reverse lexicographic term order on $\KK[u_{ij} : i,j \in [n], i>j]$.
Our ultimate goal is to prove  that 
$\SSJ_z=\init( \SSI_z)$. At this point, however, it is not even clear that $\SSJ_z \subseteq \init( \SSI_z)$,
since the leading terms of the defining generators of $\SSI_z$ do not generate $\SSJ_z$. 
We devote this subsection to proving this containment.

This section is the only place where our results depend on specific properties of 
the reverse lexicographic term order on $\KK[u_{ij} : i,j \in [n], i>j]$, as opposed to properties of term orders in general.
The reader may wish to review the definition of this term order 
from Example~\ref{lex-term2}.

Our strategy is to identify another set of  elements in $\SSI_z$ whose leading terms 
 recover the generators $ \SSu_{AB}\in \SSJ_z$ (and which we will eventually prove to be a Gr\"obner basis for $\SSI_z$).
Suppose $A$ and $B$ are finite sets of positive integers with $|A| \leq |B|$. 
Recall the definition of the matrix $\SSU$ from \eqref{SSU-eq} and 
define $g_{AB}$ to be the Pfaffian of the skew-symmetric block matrix
\be\label{gAB-eq}
g_{AB} := \pf \left[ \begin{array}{ll} \SSU_{BB} & \SSU_{BA} \\ \SSU_{AB} & 0
    \end{array}  \right] \in \KK[u_{ij} :i,j \in \NN, i> j].
\ee
For example, if $A = \{1,6,7\}$ and $B=\{1,2,3,4,5\}$ then
\[
g_{AB} = 
 \pf \left[\begin{array}{rrrrr|rrr}
            0 & -u_{21} & -u_{31} & -u_{41} & -u_{51} & 0 & -u_{61} &- u_{71}\\
            u_{21} & 0 & -u_{32} & -u_{42} & -u_{52} & u_{21} & -u_{62} & -u_{72}\\
            u_{31} & u_{32} & 0 & -u_{43} &- u_{53} & u_{31} &- u_{63} & -u_{73}\\
            u_{41} & u_{42} & u_{43} & 0 & -u_{54} & u_{41} & -u_{64} & -u_{74}\\
            u_{51} & u_{52} & u_{53} & u_{54} & 0 & u_{51} & -u_{65} & -u_{75}\\
            \hline 
            0 & -u_{21} & -u_{31} & -u_{41} & -u_{51} & 0 & 0 & 0\\
            u_{61} & u_{62} & u_{63} & u_{64} & u_{65} & 0 & 0 & 0\\
            u_{71} & u_{72} & u_{73} & u_{74} & u_{75} & 0 & 0 & 0
        \end{array}\right].
\]
To show that (some of) these polynomials belong to $\SSI_z$, we need an alternate formula
expressing each $g_{AB}$ as a sum of products of Pfaffians. 
This requires a little more notation.

Given a finite set of integers $T = \{t_1 < t_2 < \dots < t_m\}$, let $\sort(T)$ denote the word $t_1t_2 \cdots t_m$.
When $S \subseteq T$, define $\setsgn{T}{S}$ to be the sign of the permutation taking $\sort(T)$ to $\sort(T \setminus S)\sort(S)$.
By counting inversions, one checks that if $S = \{ t_i : i \in I\}$ for some set $I \subseteq [m]$, then 
\begin{equation}\label{stst-eq}
\setsgn{T}{S} = (-1)^{|S||T| +\binom{|S|}{2} + \sum I}.
\end{equation}
We will need this identity in the proof of the following:

\begin{prop}
\label{pfaffian-sum-prop}
 If $A, B$ are finite sets of positive integers with $|A| \leq |B|$, then
\[
    g_{AB} =  \sum_{\substack{S \subseteq A \setminus B \\ \text{$|B \sqcup S|$ even}}}\!\!\! (-1)^{|A \setminus S|/2}\setsgn{A}{S} \setsgn{B \sqcup S}{S}\pf (\SSU_{A \setminus S, A \setminus S})\pf(\SSU_{B \sqcup S, B \sqcup S}).
\]
\end{prop}

\begin{proof}
Write $A = \{ a_1 <a _2  <\dots <a_{|A|}\}$ and set $m=|A|+|B|$.
If  $X$ and $Y$ are two $m\times m$ skew-symmetric matrices,
then \cite[Corollary 2.2]{IshikawaWakayama} states that
\begin{equation}\label{iw0-eq} \pf(X+Y) = \sum_{\substack{[m] =I\sqcup J \\ |I|\text{ even}}} (-1)^{|I|/2 + \sum I} \pf(X_{II}) \pf( Y_{JJ})\end{equation}
where the sum is over disjoint decompositions of $[m]$.
Fix a decomposition $[m] =I\sqcup J$ with $|I|$ even.
The result will follow by applying \eqref{iw0-eq} to  
\[ X = \left[ \begin{array}{cc} \SSU_{BB} & \SSU_{BA} \\ \SSU_{AB} & \SSU_{AA}
    \end{array}  \right]
    \quad\text{and}\quad
     Y = \left[ \begin{array}{cc}0& 0 \\ 0 & -\SSU_{AA}
    \end{array}  \right].
    \]
For these matrices, we have $\pf(Y_{JJ}) = 0$ unless $J\subseteq |B| + [|A|]$. Assume this holds and let
$S = \{ a_i : i+|B| \notin J\}$. Then
$Y_{JJ} = -\SSU_{A \setminus S,A \setminus S}$ so 
\begin{equation}\label{y-eq}
\pf(Y_{JJ}) = (-1)^{|A\setminus S|/2} \pf(\SSU_{A \setminus S,A \setminus S}).
\end{equation}
If $S \cap B$ is nonempty, then $X_{II}$ has two repeated rows 
so $\pf(X_{II}) =0$. Assume further that $S\cap B=\emptyset$. Then $S \subseteq A\setminus B$ and $|S| = |A| - |J| = |I|-|B| \equiv |B| \mod 2$, so
\begin{equation}
\label{i-eq}
(-1)^{|I| /2}= (-1)^{\frac{|B| + |S|}{2}} = (-1)^{\binom{|B|+1}{2} + \binom{|S|}{2}}.
\end{equation}
As
$X_{II}$ is obtained from $\SSU_{B\sqcup S, B\sqcup S}$ by applying the permutation 
taking the word $\sort(B\sqcup S)$ to $\sort(B) \sort(S)$
to all row and column indices, we have
\begin{equation}
\label{x-eq}
\pf(X_{II}) = \setsgn{B \sqcup S}{S} \pf(\SSU_{B\sqcup S,B\sqcup S}).
\end{equation}
Finally, since  $\sum_{i \in I} i =\binom{|B|+1}{2} + |B||S| + \sum_{a_i \in S} i$, it follows from
\eqref{stst-eq} that
\begin{equation}
\label{s-eq} (-1)^{\sum I}  = (-1)^{\binom{|B|+1}{2}  + |B||S| + |A||S| + \binom{|S|}{2}} \setsgn{A}{S} .
\end{equation}
Combining \eqref{y-eq}, \eqref{i-eq}, \eqref{x-eq}, and \eqref{s-eq}, we can rewrite the right hand side of \eqref{iw0-eq} as
\[
 \sum_{\substack{S\subseteq A \setminus B \\ |B\sqcup S| \text{ even}}} (-1)^{|A\setminus S|/2 +  |A||S| + |B||S|} \setsgn{A}{S}     \setsgn{B \sqcup S}{S} \pf(\SSU_{A\setminus S,A\setminus S})\pf(\SSU_{B\sqcup S,B\sqcup S}).
 \]
Finally, we can replace the sign by $(-1)^{|A\setminus S|/2}$
since all Pfaffians in the sum are zero unless $|A|\equiv |B| \mod 2$, in which case $(-1)^{|A||S| + |B||S|} = 1$.
 \end{proof}

If $A$ and $B$ are finite sets of positive integers with $|A| = |B|$,
then we write
\be A \ominus B := \{a \in A : \text{there is no $b$ with $(a,b)\in A \odot B$ and $ (b,a)\in A \odot B$}\}.
\label{ominus-eq}
\ee
For example, if $A = \{2 < 3<4<5\}$ and $B = \{6>4>3>1\}$ then $A\ominus B = \{2,5\}$. Next, let 
\be\label{fAB-eq}
f_{AB} := g_{A\ominus B, B}  \in \KK[u_{ij} :i,j \in \NN, i> j].
\ee

\begin{cor}\label{pfaffian-sum-cor}
Suppose $z \in \FPF_n(\I_n)$. Choose $k,l \in [n]$ with $k\geq l$, write $q = \rank z_{[k][l]}+1$,
and let $(A,B) \in \binom{[k]}{q} \times \binom{[l]}{q}$.
Then  $f_{AB} \in \SSI_z$.
\end{cor}

\begin{proof}
If $S \subseteq (A\ominus B)\setminus B$ and $|B\sqcup S|$ is even, then $B \sqcup S \subseteq [k]$ and $|(B \sqcup S) \cap [l]| \geq |B \cap [l]| =|B|= q$, so  $\pf(\SSU_{B \sqcup S,B \sqcup S}) \in \SSI_z$ by definition
and $f_{AB}=g_{A\ominus B, B}  \in \SSI_z$ by Proposition~\ref{pfaffian-sum-prop}.
\end{proof}

If our sets $(A,B) \in \binom{\NN}{q} \times \binom{\NN}{q}$ are $A = \{ a_1 < a_2 < \dots < a_q\}$ and $B = \{ b_1 > b_2 > \dots > b_q\}$,
then we say that $(A,B)$ is \emph{untwisted} if it never holds that $b_i > a_j > a_i > b_j$ for any $1\leq i < j \leq q$.
This property fails when $A = \{2<3<4<5\}$ and $B = \{6>4>3>1\}$
as in our example above, but it does hold for the interchanged sets $A=\{1<3<4<6\}$ and $B=\{5>4>3>2\}$.

In the following lemma and any similar statements, the notation $\init(f)$ means the initial term of $f$ 
relative to
the reverse lexicographic term order on $\KK[u_{ij} : i,j \in [n], i>j]$.

\begin{lem}\label{untwist-lem1}
 If the pair $(A,B) \in \binom{\NN}{q} \times \binom{\NN}{q}$ is  untwisted, then $\init(f_{AB}) =\pm \SSu_{AB}$. 
 \end{lem}

For example,
if  $A = \{2<3<4<5\}$ and $B = \{6>4>3>1\}$ then $\SSu_{AB} = \SSu_{BA} = u_{43} u_{62} u_{51}$ while
$\init(f_{AB}) = \init(g_{\{2,5\}, \{1,3,4,6\}})= -u_{53} u_{61} u_{42} \neq \pm\SSu_{AB}$.
On the other hand, the pair $(B,A)$ is untwisted and 
$\init(f_{BA}) = \init(g_{\{1,6\}, \{2,3,4,5\}}) = - \SSu_{BA}$ as the lemma predicts.

\begin{proof}
The set
$A\odot B$ contains a pair $(a,b)$ with $a=b$ if and only if $|A\ominus B| + |B| $ is odd.
The lemma holds in this case since then
 $f_{AB} = 0$, being the Pfaffian of an odd sized matrix, and $\SSu_{AB} = 0$ by definition.
Instead assume that $A\odot B$ contains no pairs $(a,b)$ with $a=b$.

Let $ \mathcal{V}^{\mathrm{ss}}_{BB}$ denote the $q\times q$ matrix formed 
from $\SSU_{BB}$ by replacing all entries in row $i$ or in column $i$  
by zeros whenever $b_i \in(A\ominus B) \cap B$. Then the basic properties of Pfaffians imply that
\begin{equation}
\label{Vv-eq}
f_{AB} := \pf \left[ \begin{array}{ll} \SSU_{BB} & \SSU_{B,A\ominus B} \\ \SSU_{A\ominus B,B} & 0
    \end{array}  \right] =
    \pf \left[ \begin{array}{ll}  \mathcal{V}^{\mathrm{ss}}_{BB} & \SSU_{B,A\ominus B} \\ \SSU_{A\ominus B,B} & 0
    \end{array}  \right].
    \end{equation}
  Now recall that a \emph{matching} is an undirected graph in which no vertex belongs to two edges.
Define $\cM_{AB}$ to be the set of matchings with $2q$ vertices labeled $a_1,a_2,\dots,a_q,b_1,b_2,\dots,b_q$
such that 
\begin{itemize}
\item each $a_i \in A \setminus (A\ominus B)$ is an isolated vertex;
\item each $a_i \in A\ominus B$ is connected to some $b_j \in B$;
\item each $b_i \in B \cap (A\ominus B)$ is connected to some $a_j \in (A\ominus B) \setminus \{b_i\}$; 
\item each $b_i \in B \setminus (A\ominus B)$ is connected to some $a_j \in A\ominus B$ or some $b_j \in B \setminus(A\ominus B)$.
\end{itemize}
For example, if $A = \{2<3<4<5\}$ and $B = \{5>4>3>1\}$, then $A\ominus B = \{2,5\}$ and $\cM_{AB}$ has 
only three elements, which we represent by the following pictures:
\be\label{run-eq}
\vcenter{\hbox{\small\begin{tikzpicture}[scale=0.5]
\node (1a) at (0,0) {$5$}; \node (2a) at (1,0) {$4$}; \node (3a) at (2,0) {$3$};\node (4a) at (3,0) {$1$};
\node (1b) at (0,2) {$2$}; \node (2b) at (1,2) {{$3$}}; \node (3b) at (2,2) {{$4$}};\node (4b) at (3,2) {$5$};
\draw (1a) to (1b); \draw (4a) to (4b); \draw (2a) to[out=90,in=90] (3a);
\end{tikzpicture}
\qquad\quad
\begin{tikzpicture}[scale=0.5]
\node (1a) at (0,0) {$5$}; \node (2a) at (1,0) {$4$}; \node (3a) at (2,0) {$3$};\node (4a) at (3,0) {$1$};
\node (1b) at (0,2) {$2$}; \node (2b) at (1,2) {{$3$}}; \node (3b) at (2,2) {{$4$}};\node (4b) at (3,2) {$5$};
\draw (1a) to (1b); \draw (3a) to (4b); \draw (2a) to[out=90,in=90] (4a);
\end{tikzpicture}
\qquad\quad
\begin{tikzpicture}[scale=0.5]
\node (1a) at (0,0) {$5$}; \node (2a) at (1,0) {$4$}; \node (3a) at (2,0) {$3$};\node (4a) at (3,0) {$1$};
\node (1b) at (0,2) {$2$}; \node (2b) at (1,2) {{$3$}}; \node (3b) at (2,2) {{$4$}};\node (4b) at (3,2) {$5$};
\draw (1a) to (1b); \draw (2a) to (4b); \draw (3a) to[out=90,in=90] (4a);
\end{tikzpicture}}}
\ee
To each edge $\{ i > j\}$ we associate the variable $u_{ij}$. We 
then define $u_{AB}^{M}$ to be the product of the variables corresponding to all edges in $M \in \cM_{AB}$.
Comparing the definition of a generic Pfaffian \eqref{pf-def-eq} with \eqref{Vv-eq} shows that
$f_{AB} =\sum_{M \in \cM_{AB}} \varepsilon(M) u_{AB}^M$
 for some map $\varepsilon : \cM_{AB} \to \{\pm 1\}$.
Indeed, in our running example,
$ f_{\{2<3<4<5\},\{6>4>3>1\}} = -u_{43} u_{52} u_{51} + u_{41} u_{53} u_{52} - u_{54} u_{52} u_{31}.$
In principle there might be further cancellations among the terms in this formula for $f_{AB}$, but this will not affect our argument.

Observe that if $a_i \in A\setminus (A\ominus B)$ then there is a unique index $j \neq i$ with $b_i = a_{j} \in A\setminus (A\ominus B)$ and $b_{j} = a_i$; 
let $\cP$ denote the set of such pairs $(i,j)$ and write $\I$ for the set of indices $i \in [q]$ with $a_i \in A\ominus B$.
In our example with $A = \{2<3<4<5\}$ and $B = \{5>4>3>1\}$,
we have $\cP = \{ (2,3),(3,2)\}$ and $\I = \{1,4\}$ since $a_2 = b_3 = 3$ and $a_3=a_2 = 4$.
Next let $M_0 \in \cM_{AB}$ be the matching 
whose edges consist of $\{ a_i ,b_i\}$ for each $i \in \I$ 
along with $\{b_i, b_{j}\}$ for each $(i,j) \in \cP$.
For our running example, $M_0$ is the left-most matching shown in \eqref{run-eq}.
We always have $u_{AB}^{M_0} = \SSu_{AB} \neq 0$ since we assume that $A\odot B$ does not intersect the main diagonal.

It suffices to show that if $M \in \cM \setminus\{ M_0\}$ then $u_{AB}^M < \SSu_{AB}$
in the reverse lexicographic term order.
Fix such a matching $M \in \cM \setminus\{M_0\}$ and 
consider the edge $\{ r > s\}$ in $M_0$ whose associated variable $u_{rs}$ is maximal in lexicographic order.
We first prove that either $u_{AB}^M < \SSu_{AB}$ in the  reverse lexicographic term order
or $\{r>s\}$ is also an edge in $M$.
The variable $u_{rs}$ is $\pm u_{ab}$ for either $(a,b) =(a_1,b_1)$ or $(a,b) =(a_q,b_q)$.
Exactly one of the following must therefore hold: 
\begin{itemize}
\item[(a)] We have $(1,q) \in \cP$ and $\{ r>s\} = \{ b_1 > b_{q}\}$.
\item[(b)] We have $q \in \I$ and $\{r >s \} = \{a_q > b_q\}$.
\item[(c)] We have $1 \in \I$ and $\{ r > s\} = \{ b_1 > a_1\}$.
\end{itemize}
We consider each case in turn.

If we are in case (a)
then $b_q = a_1 < \dots < a_q = b_1$, 
so
 if
$\{b_1> b_{q}\}$ were not an edge in $M$ then 
the variable associated to the edge incident to $b_1$
would exceed $u_{rs}=u_{b_1b_q}$ in lexicographic order, giving  $u_{AB}^M < \SSu_{AB}$
in the reverse lexicographic term order.

Suppose we are in case (b). 
Then we must have  $(1,q) \notin \cP$
and $a_q \geq b_1$ to avoid cases (a) and (c). Thus, as the edge of $M$
incident to $a_q$ must have the form $\{a_q> b_i\}$,
 if this edge is not $\{a_q >b_q\}$ then $u^M_{AB} < \SSu_{AB}$.

Finally suppose we are in case (c). This case is 
where our untwisted hypothesis plays a role.
To avoid cases (a) and (b), we must have $(1,q) \notin \cP$ and
 $b_1 \geq a_q$, and if $b_1=a_q$ then $b_q < a_1$.
As $\{b_1,a_q\}$ cannot be an edge in $M$ if $b_1=a_q$,
it follows that $u^M_{AB} < \SSu_{AB}$ whenever $M$ has an edge of the form $\{b_1> a_i\}$ with $i>1$.
It remains to show that the same holds whenever $M$ has an edge of the form $\{b_1>b_i\}$ with $1<i\leq q$.
If $b_1=a_q$ then $b_1 \in  (A\ominus B) \cap B$ since $(1,q) \notin \cP$;
in this case $M$ can have no edges of the form $\{b_1>b_i\}$ so what we want to show holds vacuously.
Instead suppose $b_1 > a_q$. Since $(A,B)$ is untwisted, we must have $a_1 \leq b_i$ for all $1<i\leq q$.
This can only hold if $a_1 < b_i$ for all $1<i<q$ and $a_1 \leq b_q$.
Thus if $\{b_1 > b_i\}$ is an edge in $M$ with $1<i<q$ or if $\{b_1>b_q\}$ is an edge with $a_1 <b_q$, 
then we have  $u^M_{AB} < \SSu_{AB}$ in the reverse lexicographic term order.
The only remaining possibility is that $a_1=b_q$, but then
  $b_q \in (A\ominus B) \cap B$ so $\{b_1>b_q\}$ cannot be an edge in $M$ by definition.
  Thus our claim also holds in case (c).

To finish the proof of the lemma, define 
\[ A' = \begin{cases} 
A \setminus\{a_1,a_q\} &\text{in case (a)}\\
A \setminus\{a_1\} &\text{in case (b)}\\
A \setminus\{a_q\} &\text{in case (c)}
\end{cases}
\quad\text{and}
\quad
B' = \begin{cases} 
B \setminus\{b_1,b_q\} &\text{in case (a)}\\
B \setminus\{b_1\} &\text{in case (b)}\\
B \setminus\{b_q\} &\text{in case (c)}.
\end{cases}
\]
The pair $(A',B')$ is untwisted and $\SSu_{AB} = u_{rs} \cdot \SSu_{A'B'}$. 
Moreover, it follows from our claim that either $u^M_{AB} < \SSu_{AB}$ or $u^M_{AB} = u_{rs}\cdot u^{M'}_{A'B'}$ for some 
matching $M' \in \cM_{A'B'}$.
But this alternative implies that $u^M_{AB} < \SSu_{AB}$ in the reverse lexicographic term order in all cases, since 
we may assume by induction that $u^{M'}_{A'B'} < \SSu_{A'B'}$ as we originally chose $M \neq M_0$.
\end{proof}

Let $ { a_1\ a_2\ \cdots\ a_q  \choose b_1\ b_2\ \cdots\ b_q}$ stand for the pair of sets 
\begin{equation}\label{aabb-eq}
(A,B) = (\{a_1<a_2<\dots<a_q\},\{b_1>b_2>\dots>b_q\}).
\end{equation}
Given a set of indices $S = \{s_1 <s_2 <  \cdots < s_t\} \subseteq [q]$, define 
\[ (A,B)^S = { a_1\ a_2\ \cdots\ a_q  \choose b_1\ b_2\ \cdots\ b_q}^S := { a_1'\ a_2'\ \cdots\ a_q'  \choose b_1'\ b_2'\ \cdots\ b_q'}\] where 
$a_1' a_2'\cdots a_q'$ is obtained from $a_1a_2 \cdots a_q$ by replacing the subword $a_{s_1} a_{s_2}\cdots a_{s_t}$ with $b_{s_t} \cdots b_{s_2}b_{s_1}$, and likewise $b_1'b_2'\cdots b_q'$ is obtained from $b_1 b_2\cdots b_q$ by replacing the subword $b_{s_1}b_{s_2} \cdots b_{s_t}$ with $a_{s_t} \cdots a_{s_2}a_{s_1}$. For example, 
${a_1\ a_2\ a_3\ a_4\ a_5 \choose b_1\ b_2\ b_3\ b_4\ b_5}^{\{1,3,4\}} = {b_4\ a_2\ b_3\ b_1\ a_5  \choose a_4\ b_2\ a_3\ a_1\ b_5}.$
Write $\equiv$ for the relation on $\binom{\ZZ}{q} \times \binom{\ZZ}{q}$ that is
the transitive closure of the relations $(A,B) \equiv (A,B)^S$ for all $S \subseteq [q]$ such that the 
top word of $(A,B)^S$ is strictly increasing and bottom word is strictly decreasing.

\begin{lem}\label{untwist-lem2}
Suppose $(A,B) \in \binom{[k]}{q} \times \binom{[l]}{q}$ for positive integers $k\geq l$.
Then there exists an untwisted $(A',B') \in  \binom{[k]}{q} \times \binom{[l]}{q}$ with $(A,B) \equiv (A',B')$.
\end{lem}

\begin{proof}
Write $[k_1,k_2] := \{i \in \ZZ : k_1 \leq i \leq k_2\}$.
Fix positive integers  $k_1 \leq k_2$ and $l_1 \leq l_2$ with $k_1 \leq l_1\leq l_2 \leq k_2$.
We recursively define a map 
\[\untwist: \binom{[k_1,k_2]}{q} \times \binom{[l_1,l_2]}{q} \to \binom{[k_1,k_2]}{q} \times \binom{[l_1,l_2]}{q}\]
with the property that $(A,B) \equiv \untwist(A,B)$ for all pairs $(A,B)$.
Fix $(A,B) \in \binom{[k_1,k_2]}{q} \times \binom{[l_1,l_2]}{q}$ and write 
$ (A,B) = { a_1\ a_2\ \cdots\ a_q  \choose b_1\ b_2\ \cdots\ b_q} $ as in \eqref{aabb-eq}.
If $q < 2$ or $b_1 < a_2$ then we set 
$ \untwist(A,B) := (A,B).$
Otherwise, there exists a maximal index $1 < i \leq q$ such that $b_1 \geq a_i$. Note that $b_1 < a_{i+1}$ if $i<q$.
Set $(A',B') := (A,B)$ if $a_1 < b_i$ or else  define
\[
(A',B') :=  { b_i\ \cdots\ b_2\ b_1 \ a_{i+1}\ \cdots\ a_q  \choose a_i\ \cdots\ a_2\ a_1\ b_{i+1}\ \cdots b_q}.\]
This gives a well-defined element of $\binom{[k_1,k_2]}{q} \times \binom{[l_1,l_2]}{q}$ with $(A',B') \equiv (A,B)$.
Write $ { a_1'\ a_2'\ \cdots\ a_q'  \choose b_1'\ b_2'\ \cdots\ b_q'} := (A',B') $
and note that $ a_1' < b_i'$ and $\max\{b_1',a_i'\} < a_{i+1}'$, where we set $a'_{q+1} := k_2+1$ if $i=q$.
Let 
\[k'_1 := a'_1 +1 \leq k'_2 := a'_{i+1}-1\quad\text{and}\quad
l'_1 := b_i' \leq l'_2 := b'_1 -1 .\]
Then $k_1' \leq l_1' \leq l_2' \leq k_2'$ so we can recursively construct
\[ { a_2''\ a_3''\ \cdots\ a_i''  \choose b_2''\ b_3''\ \cdots\ b_i''} := \untwist{ a_2'\ a_3'\ \cdots\ a_i'  \choose b_2'\ b_3'\ \cdots\ b_i'}
\in \binom{[k_1',k_2']}{i-1} \times \binom{[l_1',l_2']}{i-1}
\]
and then define 
\[\untwist(A,B) := { a_1'\ a_2''\ a_3''\ \cdots\ a_i''\ a'_{i+1}\ \cdots \ a'_q  \choose b_1'\ b_2''\ b_3''\ \cdots\ b_i''\ b_{i+1}'\ \cdots\ b'_q}.\]
This is an element of 
$ \binom{[k_1,k_2]}{q} \times \binom{[l_1,l_2]}{q}$ that is equivalent to $(A,B)$ under $\equiv$.

To prove the lemma, we just need to check that $\untwist(A,B)$ is untwisted, as our notation suggests.
For the pair $(A',B')$, observe that one can only have $b'_s > a'_t > a'_s > b'_t$ for some $s<t$ if $1 < s < t \leq i$.
If follows that $\untwist(A,B)$ is untwisted if and only if ${ a_2''\ a_3''\ \cdots\ a_i''  \choose b_2''\ b_3''\ \cdots\ b_i''}$ is untwisted,
and we can assume that the latter property holds by induction.
\end{proof}

Putting everything together, we can finally prove the following lemma.

\begin{lem}\label{containment-ss-lem}
For all $z \in \FPF_n(\I_n)$, the ideal 
 $\SSJ_z$ is contained in the initial ideal  $\init (\SSI_z)$ defined relative to the reverse lexicographic term order on $\KK[u_{ij} : i,j \in [n], i>j]$.
\end{lem}

\begin{proof}
Choose $k,l \in [n]$ with $k\geq l$ and let $(A,B) \in \binom{[k]}{q} \times \binom{[l]}{q}$
for $q := \rank z_{[k][l]}+1$. 
Since $\SSJ_z$ is generated by all such elements $\SSu_{AB}$ (see Definition~\ref{ssj-def}),
it suffices to show that $\SSu_{AB}  \in \init(\SSI_z)$.
By Lemma~\ref{untwist-lem2}, there is an untwisted pair $(A',B') \in  \binom{[k]}{q} \times \binom{[l]}{q}$
with $(A,B) \equiv (A',B')$ and 
$\SSu_{AB} = \SSu_{A'B'}$.
Finally, we have  $\SSu_{A'B'} =\pm \init( f_{A'B'})$ by   Lemma~\ref{untwist-lem1} and
 $ f_{A'B'} \in \SSI_z$ by Corollary~\ref{pfaffian-sum-cor}.
\end{proof}

\section{Transition equations for ideals}\label{tra-sect}

In this section we prove our main results, relating  $\SSI_z$, $\SSJ_z$, and $I(\SSX_z)$. Here is a quick
outline of the content that follows.
We show in Section~\ref{indu-sect} that $\SSJ_z = \init(\SSI_z)$ and $\SSI_z = I(\SSX_z)$,
relying on a technical lemma that is proved in Section~\ref{key-lemma-sect}.
Section~\ref{subsec:pipe-dreams} gives an explicit primary decomposition for $\SSJ_z$ involving the \emph{involution pipe dreams}
studied in \cite{HMPdreams}.
In Section~\ref{subsec:K} we discuss some applications related to $K$-theory representatives for 
$\SSX_z$.
Section~\ref{new-proofs-sect}, finally, sketches how to adapt our arguments to get new proofs of 
the results of Knutson and Miller  mentioned in the introduction.

\subsection{The inductive argument}\label{indu-sect}

 The \emph{dominant component} of an arbitrary permutation $z$ of $\NN$ is the set 
    \begin{equation}\label{dom-eq}
        \dom(z) = \left\{(i,j) \in \NN \times \NN : \rank z_{[i][j]} = 0\right\}.
    \end{equation}
This  is 
always the Young diagram of a partition.
For a strict partition $\lambda = (\lambda_1 > \cdots > \lambda_k > 0)$, let 
\[\SD_\lambda = \{ (i,i+j-1) \in \NN\times \NN : 1 \leq j \leq \lambda_i\}.\]
As explained in \cite[\S4.3]{HMPdreams}, 
if $z \in \Ifpf_\infty$ then there is a strict partition $\lambda$ 
with 
\[ \{ (i,j) \in \dom(z) : i > j\} = \{  (j+1,i) : (i,j) \in \SD_\lambda\}.\] 
An \emph{outer corner} of $\dom(z)$ is a pair $(i,j) \in (\NN\times \NN)\setminus \dom(z)$ with
\be\label{outer-corner-eq}
(i-1, j) \in \dom(z) \sqcup (\{0\} \times \NN)
\quand
(i, j-1) \in \dom(z) \sqcup ( \NN \times \{0\}).
\ee
It is clear that if $(i, j)$ is an outer corner of $\dom(z)$ then $z(i) = j$. 

\begin{ex}
Suppose $z \in \Ifpf_\infty$ is the image of $(1,5)(2,4)(3,6)$ under $\FPF_6$. Then 
\[
\dom(z) = \left\{\begin{smallmatrix}
 \square & \square & \square &\square  &  1&\cdot  \\
\square &   \square& \square  &  1& \cdot &\cdot \\
\square & \square &  \square&\cdot  & \cdot & 1 \\
\square & 1 & \cdot & \cdot & \cdot &\cdot  \\
1 & \cdot & \cdot & \cdot& \cdot &\cdot  \\
\cdot & \cdot & 1 & \cdot & \cdot & \cdot \\
\end{smallmatrix}\right\}
\]
and $\{ (i,j) \in \dom(z) : i > j\}  =  \{  (j+1,i) : (i,j) \in \SD_\lambda\}$ for the strict partition $\lambda =(3,1)$.
The outer corners of $\dom(z)$ are the pairs $(5,1)$, $(4,2)$, $(2,4)$, and $(1,5)$.
\end{ex}

Given $z \in \Ifpf_\infty$ and a positive integer $p$, define 
\begin{equation}
    \Psi(z,p) = \Bigl\{ y=(p,r)z(p,r) \in \Ifpf_\infty : p < r \text{ and } \ellfpf(y) = \ellfpf(z)+1\Bigr\}
\end{equation}
where $\ellfpf(z) = |\SSD(z)|$ as in Section~\ref{bruhat-sect}.
These sets will play a key role in the inductive proof of our main theorem, via the formulas in the following lemma.
Here, we write $(u_{pq})$ for the principal ideal of $\KK[u_{ij} : n\geq i > j \geq 1]$ generated by $u_{pq}$ for $n \geq p >q \geq 1$.
Let $\SSJ_z$ be as in Section~\ref{mon-sect}, and recall that $I(\SSX_z)$ is the ideal of all polynomials in $\KK[u_{ij} : n\geq i > j \geq 1]$ vanishing on the set $\SSX_z$,
where $u_{ij}$ is identified with the function on skew-symmetric matrices sending $A \mapsto A_{ij}$.

\begin{lem}\label{lem:transition-containment-ss}  
 Suppose $z \in \FPF_n(\I_n)$ and $(p,q) $  is an outer corner of $\dom(z)$ with $n\geq p>q$. Then 
 $ \Psi(z,p) \subseteq \FPF_n(\I_n)$, 
$I(\SSX_z) + (u_{pq}) \subseteq \bigcap_{v \in \Psi(z,p)} I(\SSX_v)$,
and
$\SSJ_z + (u_{pq}) = \bigcap_{v \in \Psi(z,p)} \SSJ_v
.$
\end{lem}


We postpone the proof of this lemma to Section~\ref{key-lemma-sect}. Here is our first main theorem.

\begin{thm} \label{ss-main-thm}  If $z \in \FPF_n(\I_n)$ then  $\SSJ_z = \init (\SSI_z) = \init (I(\SSX_z))$,
where both initial ideals are computed with respect to the reverse lexicographic term order from Example~\ref{lex-term2}.
 \end{thm}

\begin{proof}
Let $z \in \FPF_n(\I_n)$ so that $\SSD(z)$ is a subset of 
$\ltriang_n := \{ (i,j) \in [n] \times [n] : i > j\}$ by Proposition~\ref{ssd-prop}. We prove the theorem by induction on the size of
the complement of the skew-symmetric Rothe diagram $\ltriang_n \setminus \SSD(z)$.

The unique element $z \in \FPF_n(\I_n)$ with $\SSD(z) = \ltriang_n$, which indexes the containment-minimal skew-symmetric 
 matrix Schubert variety defined by the all-zeros rank table, is
 the involution
\begin{equation}\label{z-square-eq}
z_{\square}:= \FPF_n(1) =  (1,n+1)(2,n+2)\cdots(n,2n)\cdot \prod_{i=n}^\infty (2i + 1, 2i + 2).
\end{equation}
If $n \geq i > j \geq 1$, then  
$u_{ij}  = \SSu_{AB} \in \SSJ_{z_\square}
$
and
$
u_{ij} =
-\pf (\SSU_{RR}) \in \SSI_{z_\square}
$
for $A = \{i\}$, $B = \{j\}$, and $R = \{i,j\}$,
so
 $\SSJ_{z_\square}=\SSI_{z_\square}= I(\SSX_{z_\square})=  (u_{ij} : (i,j) \in \ltriang_n)$.
Thus the theorem holds in this base case.

    For the inductive step, suppose $z \in \FPF_n(\I_n)\setminus\{ z_{\square}\}$. Then the set $\dom(z)$ is invariant under transpose and not all of $[n] \times [n]$, so it has an outer corner $(p,q) \in [n] \times [n]$ with $p\geq q$. We must have $p>q$ since $z$ has no fixed points and $z(p) = q$, as noted above.
Therefore
\begin{equation} \label{eq:I-vs-J-ss}
    \begin{aligned}
        \init(I(\SSX_z)) + (u_{pq})  &\subseteq \init(I(\SSX_z) + (u_{pq}))  & \quad\text{[by Prop. \ref{prop:init-facts}(b)]}\\
        &\subseteq \init \left(\bigcap_{v \in \Psi(z,p)} I(\SSX_z)\right) &\quad\text{[by Lem. \ref{lem:transition-containment-ss}]} \\
       & \subseteq  \bigcap_{v \in \Psi(z,p)} \init \left(I(\SSX_v) \right)&\quad \text{[by Prop. \ref{prop:init-facts}(c)]} \\
       & = \bigcap_{v \in \Psi(z,p)} \SSJ_v &\quad\text{[by induction]} \\
       &= \SSJ_z + (u_{pq}). &\qquad \text{[by Lem. \ref{lem:transition-containment-ss}]}
    \end{aligned}
    \end{equation}
The appeal to induction on the penultimate line is valid since every $v \in \Psi(z,p)$ belongs to $\FPF_n(\I_n)$ 
by Lemma \ref{lem:transition-containment-ss} and has 
$|\SSD(v)| = \ellfpf(v) = \ellfpf(z) + 1 = |\SSD(z)| + 1$.

On the other hand, we have
 $\SSJ_z \subseteq \init (\SSI_z)$ by Lemma~\ref{containment-ss-lem},
and  
 $\SSI_z \subseteq  I(\SSX_z)$ by Theorem~\ref{ss-zero-locus-thm}.
    Therefore $\SSJ_z \subseteq \init (\SSI_z) \subseteq  \init (I(\SSX_z))$, so all inclusions in \eqref{eq:I-vs-J-ss} must be equalities
    and
    $ \init(I(\SSX_z) + (u_{pq})) = \init(I(\SSX_z)) + (u_{pq}).$
     By Lemma~\ref{lem:plus-vs-intersection}, this implies that 
     \begin{equation}
     \label{hy1}
     \init(I(\SSX_z) \cap (u_{pq})) = \init(I(\SSX_z)) \cap (u_{pq}).
     \end{equation}
      Now suppose $f \in I(\SSX_z) \cap (u_{pq})$, so $f = u_{pq} g$ for some $g$. Since $u_{pq}$ is not identically zero on $\SSX_z$, as the permutation matrix of $z$ has a $1$ at the outer corner $(p,q)$ of $\dom(z)$, we have $u_{pq} \notin I(\SSX_z)$. But $I(\SSX_z)$ is a prime ideal as $\SSX_z$ is irreducible by Theorem~\ref{ssx-thm}, so $u_{pq} g \in I(\SSX_z)$ forces us to have $g \in I(\SSX_z)$. This proves 
     that
     $I(\SSX_z) \cap (u_{pq}) = u_{pq} I(\SSX_z)$, so
     \begin{equation}\label{porism-ss}
     \begin{aligned}
     \init(I(\SSX_z)) \cap (u_{pq}) &= \init(I(\SSX_z) \cap (u_{pq})) = \init(u_{pq} I(\SSX_z)) = u_{pq}  \init (I(\SSX_z)).
     \end{aligned}
     \end{equation}
     The chain of equalities \eqref{eq:I-vs-J-ss} shows that
     \begin{equation}\label{hy2}
     \init(I(\SSX_z)) + (u_{pq}) = \SSJ_z + (u_{pq}).
     \end{equation}
     Applying Lemma~\ref{lem:hilbert}, whose hypotheses hold by \eqref{porism-ss} and \eqref{hy2}, gives $\init(I(\SSX_z)) = \SSJ_z$.
    Since $\SSJ_z \subseteq \init (\SSI_z) \subseteq \init ( I(\SSX_z))$, both inclusions are equalities. 
    \end{proof}

\begin{cor} If $z \in \FPF_n(\I_n)$ and $(p,q)$ is an outer corner of $\dom(z)$ with $n \geq p >q$, then 
$\SSJ_z \cap (u_{pq})= u_{pq} \SSJ_z.$
 \end{cor}

\begin{proof}
Since now we know that $\init(I(\SSX_z)) = \SSJ_z$, this is just \eqref{porism-ss}.
\end{proof}

We state two more significant consequences of Theorem~\ref{ss-main-thm} as the following theorems.

\begin{thm}
\label{ss-main-thm2}
 If $z \in \FPF_n(\I_n)$ then $\SSI_z = I(\SSX_z)$, which  is a prime ideal. \end{thm}

    \begin{proof}
       We have $\SSI_z \subseteq I(\SSX_z)$  by Theorem~\ref{ss-zero-locus-thm} and $\init( \SSI_z) = \init (I(\SSX_z))$ by Theorem~\ref{ss-main-thm},
       so Proposition~\ref{prop:init-facts}(d) implies that $\SSI_z = I(\SSX_z)$, which is prime
    since  $\SSX_z$ is irreducible.
 \end{proof}


Given sets $A = \{a_1<\dots<a_q\}$ and $B = \{b_1 > \dots > b_q\}$,
recall from \eqref{fAB-eq} that  we define 
\[ f_{AB} :=  \pf \left[ \begin{array}{ll} \SSU_{BB} & \SSU_{B,A\ominus B} \\ \SSU_{A\ominus B, B} & 0
    \end{array}  \right] \in \KK[u_{ij} :i,j \in \NN, i> j] 
    \]
    where 
    $A \ominus B$ is the set of $a \in A$ for which no $b$ exists with $\{(a,b),(b,a)\} \subset \{(a_1,b_1),\dots,(a_q,b_q)\}$,
    and $\SSU$ is the skew-symmetric matrix whose $(i,j)$ entry 
    is $u_{ij}$ if $i>j$.
We also repeat that  the pair of sets $(A,B)$ is \emph{untwisted} if 
 it never holds that $b_i > a_j > a_i > b_j$ for any $1\leq i < j \leq q$.

\begin{thm}\label{ss-main-thm3}
If $z \in \FPF_n(\I_n)$ then 
the elements $f_{AB}$, as $(i,j)$ ranges over all pairs in $[n]\times [n]$ with $i \geq j$ and $(A,B)$ ranges over all untwisted pairs in
$ \tbinom{[i]}{q}\times \tbinom{[j]}{q} $ for $q = \rank z_{[i][j]}+1$,
form a 
Gr\"obner basis for $I(\SSX_z) $ with respect to the reverse lexicographic term order.
\end{thm}

\begin{proof}
Each such $f_{AB}$ belongs to $\SSI_z = I(\SSX_z) $ by Corollary~\ref{pfaffian-sum-cor},
and by the argument in the proof of Lemma~\ref{containment-ss-lem} the initial terms of these polynomials
generate $\SSJ_z = \init(I(\SSX_z))$.
\end{proof}

\subsection{The key lemma}\label{key-lemma-sect}

In this subsection we prove Lemma~\ref{lem:transition-containment-ss}.
Write $<$ for the Bruhat order on $\Ifpf_\infty$ from Section~\ref{bruhat-sect}.

\begin{lem}\label{ijr-prime-lem-fpf}
Suppose $z \in \Ifpf_{\infty}$
and $(p,q)$ is an outer corner of $\dom(z)$ with $p>q$.
Assume $i,j,r\in \NN$ are integers such that $i<j$, $p<r$,
and $z < (i,j)z(i,j) < (p,r)(i,j)z(i,j)(p,r).$ Then there are integers $i',j',r' \in \{i,j,p,r,z(i), z(j), z(p), z(r)\}$ 
with $i'<j'$, $p<r'$, and
\[z < (p,r')z(p,r') < (i',j')(p,r')z(p,r')(i',j') = (p,r)(i,j)z(i,j)(p,r).\]
\end{lem}

\begin{proof}
Let $S = \{i,j,p,r,z(i), z(j), z(p), z(r)\}$ and 
write $\phi : [n]\to S$ and $\psi : S \to [n]$ 
for the unique order-preserving bijections where $n = |S|$. Note that $n \leq 8$.
We have $z(p) = q$,
so $(p',q') := (\psi(p), \psi(q)) \in [n]\times [n]$ is an outer corner of $z' := \psi \circ z \circ\phi \in \Ifpf_n$
with $p'>q'$.
If the lemma fails for $z$ and $(p,q)$ then it must also fail for $z'$ and $(p',q')$,
but it is a finite calculation to check that the lemma holds when $z \in \Ifpf_n$ and $n\leq 8$.
\end{proof}

\begin{lem} \label{lem:bruhat-minimal-fpf} 
Suppose $z \in \Ifpf_{\infty}$
and $(p,q)$ is an outer corner of $\dom(z)$ with $p>q$. Then 
$\Psi(z,p)$ is the subset of Bruhat-minimal elements in the set
\begin{equation*}
\begin{aligned} 
\cX(z,p) &:= \left\{v \in \Ifpf_\infty : v \geq z, \text{ $\rank v_{[i][j]} = 0$ for $(i,j) \in [p] \times [q]$}\right\} 
\\
&\phantom{:}=  \left\{v \in \Ifpf_\infty : v \geq z, \ \dom(v)\supseteq [p] \times [q]\right\}.
\end{aligned}
\end{equation*}
In addition, if $ z\in \FPF_n(\I_n)$ and $p \in [n]$ then $\Psi(z,p) \subseteq \FPF_n(\I_n)$.
 \end{lem}

 \begin{proof}
As noted above,
 we have $z(p) = q$ so
$([p]\times[q])\setminus \{(p,q)\} \subseteq \dom(z)$ and 
 $z \notin \cX(z, p)$. 
Properties (P1) and (P2) of the Bruhat order on $\Ifpf_\infty$ from Lemma~\ref{ss-bruhat-lem},
along with Proposition~\ref{ss-bruhat-prop},
imply that if $i<j$ are positive integers with $z< (i,j)z(i,j)$, then $[p]\times[q] \subseteq \dom((i,j)z(i,j))$ if and only if $i\in \{p,q\}$.
Since $\Psi(z,p) \supseteq \Psi(z,q)$ by \cite[Lem. 4.15]{HMP3},
this shows that 
\[
 \Psi(z,p) = \left\{ v \in \cX(z, p) : \ellfpf(v) = \ellfpf(z) + 1\right\}.
\]
This a subset of Bruhat-minimal 
elements in $\cX(z,p)$.
Using Proposition~\ref{ssd-prop} and property (P2) from Lemma~\ref{ss-bruhat-lem},
it is likewise straightforward to deduce that
if $z \in \FPF_n(\I_n)$ and $p \in [n]$, then $\Psi(z,p) \subset  \FPF_n(\I_n)$.

Now suppose $w \in \cX(z,p)$ and $\ellfpf(w)-\ellfpf(z) \geq 2$.
We argue by induction that $w$ is not Bruhat-minimal in $\cX(z,p)$.
There must exist some positive integers $i<j$ such that if $v:=(i,j)y(i,j)$
then $z < v < w$ and $\ellfpf(v) =\ellfpf(z)+1$. If we have $i \in \{p,q\}$ then $v \in \cX(z,p)$ 
in which case $w$ is not Bruhat-minimal in $\cX(z,p)$.
Assume $i\notin \{p,q\}$. Then $(p,q)$ is an outer corner of $v$ and $w \in \cX(v,p)$.
If $\ellfpf(w) - \ellfpf(z)  = 2$, then  $\ellfpf(w) - \ellfpf(v) = 1$, so $w\in \Psi(v,p)$
and $w = (p,r)v(p,r)$ for some $r>p$.
In this case, Lemma~\ref{ijr-prime-lem-fpf} asserts that 
\[w = (i',j')(p,r')z(p,r')(i',j') > (p,r')z(p,r') \in \Psi(z,p)\subseteq \cX(z,p)\] for some integers $p<r'$ and $i'<j'$, so $w$ is not Bruhat-minimal in $\cX(z,p)$.
Alternatively, if $\ellfpf(w) - \ellfpf(z) \geq 3$, then 
$\ellfpf(w) - \ellfpf(v) \geq 2$ so
by induction $w$ is neither Bruhat-minimal in $\cX(v,p)$ nor in $\cX(z,p)$. 
 \end{proof}

Given $(p,q) \in [n] \times [n]$ with $p>q$, consider the linear subspace 
\begin{equation*}
 \SSX_{p,q}:=   \{A \in \SSM_n: \text{$A_{ij} = 0$ for all $(i,j) \in ([p] \times [q]) \cup ([q]\times [p])$}\}.
\end{equation*}
Since $\SSX_{p,q}$ is closed, irreducible, and $B_n$-stable, it is a single skew-symmetric matrix Schubert variety.
Moreover, $I(\SSX_{p,q}) $ is the ideal $( u_{ij} : (i,j) \in [p]\times[q],\ i>j) \subset \KK[u_{ij} : i,j \in[n], i>j]$.

\begin{cor}\label{cor:SSXw-transitions} Let $z \in \FPF_n(\I_n)$ and suppose $(p,q)$ is an outer corner of $\dom(z)$ with $n\geq p > q$. 
Then $\SSX_z \cap \SSX_{p,q} = \bigcup_{v \in \Psi(z,p)} \SSX_v$. \end{cor}
Although the equality here is meant set-theoretically,  the proof of Theorem~\ref{ss-main-thm} implies $I(\SSX_z) + (u_{pq}) = I(\SSX_z) + I(\SSX_{p,q}) = \bigcap_{v \in \Psi(z,p)} I(\SSX_v)$, so this corollary also holds at the level of schemes.

\begin{proof}
The intersection $\SSX_z \cap \SSX_{p,q}$ is closed and $B$-stable,
so it is a union of  $\SSX_v$ for certain $v \in \FPF_n(\I_n)$.
It follows from Proposition~\ref{ss-bruhat-prop} and Lemma~\ref{lem:bruhat-minimal-fpf} that
if $v \in \FPF_n(\I_n)$ then $\SSX_v \subseteq \SSX_z \cap \SSX_{p,q}$ if and only if $v  \in \Psi(z,p)$.
\end{proof}

We may now prove our key lemma.

    \begin{proof}[Proof of Lemma~\ref{lem:transition-containment-ss}]
    Let $z \in \FPF_n(\I_n)$ and suppose $(p,q) $  is an outer corner of $\dom(z)$ with $n\geq p>q$.
    The fact that $\Psi(z,p) \subseteq \FPF_n(\I_n)$ was noted in Lemma~\ref{lem:bruhat-minimal-fpf}.
     
    One has
    $\bigcap_{v \in \Psi(z,p)} I(\SSX_v) = I(\bigcup_{v \in \Psi(z,p)} \SSX_v) $ so Corollary~\ref{cor:SSXw-transitions} implies that this intersection 
    contains 
    $I(\SSX_z) + I(\SSX_{p,q})$. Since $I(\SSX_{p,q})$ is generated by the variables $u_{ij}$ for $(i,j) \in [p] \times [q]$ with $i>j$, all of which are in $I(\SSX_z)$ except for $u_{pq}$, we therefore have 
    \[\bigcap_{v \in \Psi(z,p)} I(\SSX_v)  \supseteq I(\SSX_z) + I(\SSX_{p,q}) = I(\SSX_z) + (u_{pq}).\]
    This proves the desired formula for $I(\SSX_z) + (u_{pq})$.

Next, let ${\r}_{p,q}$ be the rank table with $ {\r}_{p,q}(i,j)  = 0$ whenever $(i,j) \in [p] \times [q]$ or $(i,j) \in [q] \times [p]$
 and with $\r_{p,q}(i,j) = n$ for all other positive integer pairs $(i,j)$.
Suppose $v \in \FPF_n(\I_n)$. As ${\r}_v(i,j) \leq \min({\r}_z, {\r}_{p,q})(i,j)$ for all $i,j \in [n]$ if and only if $v \geq z$ and $\rank v_{[i][j]} = 0$ for $(i,j) \in [p] \times [q]$, 
we have $\SSJ_{\min({\r}_z, {\r}_{p,q})} = \bigcap_{v \in \Psi(z,p)} \SSJ_v$ by Theorem~\ref{thm:bruhat-minimal-intersection-ss} and Lemma~\ref{lem:bruhat-minimal-fpf}.
Proposition~\ref{prop:min-to-sum-ss} shows that $\SSJ_{\min({\r}_z, {\r}_{p,q})} = \SSJ_z + \SSJ_{{\r}_{p,q}}$. The ideal $\SSJ_{{\r}_{p,q}}$
is generated by the variables $u_{ij}$ for $(i,j) \in [p] \times [q]$ with $i>j$, and all of these except
 $u_{pq}$
are contained in $\SSJ_z$. We conclude that  
\[\bigcap_{v \in \Psi(z,p)} \SSJ_v =\SSJ_{\min({\r}_z, {\r}_{p,q})} = \SSJ_z + \SSJ_{{\r}_{p,q}} = \SSJ_z + (u_{pq})\]
which is the desired formula for $ \SSJ_z + (u_{pq})$.
\end{proof}

\subsection{Pipe dreams}
\label{subsec:pipe-dreams}

In this subsection, we derive an alternate expression for $\SSJ_z$
as an intersection of certain natural ideals indexed by the \emph{involution pipe dreams} corresponding to $z$, as introduced in \cite{HMPdreams}.

Define the \emph{pipe dream reading word} of a finite set $D \subset \NN \times \NN$ to be the word whose letters list the numbers $i+j-1$ as $(i,j)$ ranges over all elements of $D$ in the order that makes $(i,-j)$ increase lexicographically, that is, reading the rows in order but going right to left.
For example, the pipe dream reading word of $D = \{(1,4),(1,3),(2,6),(5,5),(5,4),(5,3)\}$ is $437987$.
We write $\word(D)$ to denote the pipe dream reading word of $D$.

\begin{defn} \label{rp-def}
A \emph{(reduced) pipe dream} for a permutation $w \in S_{\infty}$ is a finite set $D \subset \NN \times \NN$ 
whose pipe dream reading word is a reduced word for $w$, that is, a sequence of positive integers $i_1i_2\cdots i_l$ of shortest possible length such that $w=s_{i_1}s_{i_2}\cdots s_{i_l}$.
\end{defn}

This definition first appeared in work of Bergeron and Billey \cite{bergeron-billey}, who referred to pipe dreams as \emph{rc-graphs}.
The set $D = \{(1,4),(1,3),(2,6),(5,5),(5,4),(5,3)\}$ is a pipe dream for the permutation $w= (3,5,4)(7,10,9)$.
Let $\RP(w)$ denote the set of all pipe dreams for $w \in S_\infty$.

Recall that $\Ifpf_m$ is the set of involutions  $y=y^{-1} \in S_m$ with no fixed points in $[m]$, and that 
$ \FPF_m(y) = y \cdot (m+1,m+2)(m+3,m+4)(m+5,m+6)\cdots$
for $y \in \Ifpf_m$ if $m$ is even.

\begin{defn}\label{fp-def}
Suppose $z = \FPF_m(y) \in \Ifpf_\infty$ for some even $m \in 2\NN$ and $y \in \Ifpf_m$.
The set of \emph{(fpf-involution) pipe dreams} for $z $ is then
   $
        \FP(z) := \{D \cap \ltriang : D^T=D \in \RP(y)\}
   $
    where $\ltriang := \{(i,j) \in \NN \times \NN : i > j\}$ and $D^T$ is the transpose of $D$.
\end{defn}
The set $\FP(z)$ where $z = \FPF_6((1,2)(3,6)(4,5))$ is given in Example~\ref{ex:subword-ss} below.

Every element $z \in \Ifpf_\infty$ arises as $z = \FPF_m(y)$ for some choice of  $m\in 2\NN$ and $y \in \Ifpf_m$,
and our definition $\FP(z)$ might seem to depend on this choice. 
However, if $y_1 \in \Ifpf_m$ and $y_2 \in \Ifpf_{m+2}$ are such that $\FPF_m(y_1) = \FPF_{m+2}(y_2)$,
then we must have $y_2 = y_1 s_{m+1}$ and therefore
\[
\RP(y_2) = \left\{ D \sqcup \{(i,j)\} : D \in \RP(y_1)\text{ and } i+j-1 = m+1\right\}.
\]
It follows that 
$
\{D \cap \ltriang : D^T=D \in \RP(y_1)\}
=
\{D \cap \ltriang : D^T=D \in \RP(y_2)\}
$
so
$\FP(z)$ 
is independent of the choice of $y \in \Ifpf_m$ with $z = \FPF_m(y)$.

\begin{rem}
In \cite[Def. 1.4]{HMPdreams}, a set of pipe dreams $\FP(y)$ is defined for each $y \in \Ifpf_m$.
This set is what we call $\FP(z)$ for $z = \FPF_m(y)$.
It is straightforward to translate the results in \cite{HMPdreams} about pipe dreams for elements of $\Ifpf_m$
to statements for elements of $\Ifpf_\infty$ via the preceding discussion, and 
we will not comment further on these unimportant differences in notation.
\end{rem}

 \begin{prop}\label{fp-00-prop}
 If $z \in \FPF_n(\I_n)$ then every $D \in \FP(z)$ is a subset of  $\ltriang_n := \ltriang \cap ([n] \times [n])$
 with size $|D| = |\SSD(z)| = \ellfpf(z)$.
 \end{prop}

 \begin{proof}
Let $z \in \FPF_n(\I_n)$, so that $\SSD(z) \subseteq \ltriang_n$.
By \cite[Thm. 5.15]{HMPdreams},
 there exists an element  $\hat D_{\mathsf{bot}}^{\fpf} \in \FP(z)$ and a partial order $\leq_{\FP}$ on finite subsets of $\ltriang$ such that 
 $\FP(z) = \{ E \subset \ltriang : \hat D_{\mathsf{bot}}^{\fpf}  \leq _{\FP} E\}$.
 The partial order has the property that if $D \subseteq \ltriang_n$ and $D \leq_{\FP} E$ then $|D| =|E|$ and $E \subseteq \ltriang_n$.
But it is clear from \cite[Eq. (5.5)]{HMPdreams}  that $|\hat D_{\mathsf{bot}}^{\fpf}| = |\SSD(z)|$ and $\hat D_{\mathsf{bot}}^{\fpf}  \subseteq \ltriang_n$ if $\SSD(z) \subseteq \ltriang_n$.
 \end{proof}

    Given a set $D \subseteq \ltriang_n $,
    we write $ (u_{ij} : (i,j) \in D)$ to denote the ideal in
    the coordinate ring $\KK[\SSM_n] = \KK[u_{ij} : i,j \in [n], i>j]$ generated by $u_{ij}$ for all $(i,j) \in D$.

\begin{lem} \label{lem:Jw-in-JD-ss} If $z \in \FPF_n(\I_n)$ and $D \in \FP(z)$, then 
   $
    \SSJ_z \subseteq (u_{ij} : (i,j) \in D) \subseteq \KK[\SSM_n].
  $
\end{lem}
    \begin{proof}
        Fix $i,j\in [n]$ with $i>j$. Let $q =   \rank z_{[i][j]}+1$ and choose sets  $(A,B) \in \binom{[i]}{q} \times \binom{[j]}{q}$. Let
        \begin{equation*}
            S = \{(a,b) : \text{$a > b$ and $A \odot B$ contains $(a,b)$ or $(b,a)$}\},
        \end{equation*}
        so that $\SSu_{AB} = \prod_{(a,b) \in S} u_{ab}$ is a generator of $\SSJ_z$. 
        It suffices to show that $\SSu_{AB} \in (u_{ij} : (i,j) \in D)$. If $A \odot B$ contains a diagonal position then this is obvious since $\SSu_{AB}  =0$,
        so assume that $A \odot B$ does not intersect the main diagonal.

         By definition, there exists some $y \in \Ifpf_m$ and $E =E^T\in \RP(y)$  such that $z = \FPF_m(y)$ and $D = E \cap \ltriang$.
The fact that $\prod_{(a,b) \in A \odot B} u_{ab}$ belongs to the ideal of $ \KK[u_{ij} : i,j \in [n]]$ 
generated by all $u_{ij}$ with $(i,j) \in E$ is a consequence of \cite[Theorem B]{KnutsonMiller}, which means that $E \cap (A \odot B)$ is nonempty. 
Since $A \odot B$ contains no element $(i,i)$ by hypothesis,  this  implies that $D \cap S$ is also nonempty,
so $u_{ij} \mid \SSu_{AB}$ for some $(i,j) \in D$, proving that $\SSu_{AB} \in (u_{ij} : (i,j) \in D)$.
    \end{proof}

For us, a \emph{monomial ideal} is an ideal in a commutative polynomial ring generated by monomials.

\begin{lem} \label{lem:intersection-of-sums}  If $I, J, K$ are monomial ideals, then $I + (J \cap K) = (I + J) \cap (I + K)$.  \end{lem}

This lemma follows as a basic exercise; we include a proof for completeness.

    \begin{proof}
        The inclusion $I + (J \cap K) \subseteq (I + J) \cap (I + K)$ always holds. The collection of monomial ideals is closed under sum and intersection, so it suffices to assume that $\mon \in (I + J) \cap (I + K)$ is a monomial and show that $\mon \in I + (J \cap K)$. Since $\mon$ is a monomial and $I, J, K$ are monomial ideals, $\mon \in I + J$ implies that $\mon \in I$ or $\mon \in J$; likewise $\mon \in I+K$ implies that $\mon \in I$ or $\mon \in K$. That is, we either have $\mon \in I$ or $\mon \in J \cap K$, so $\mon \in I + (J \cap K)$. 
    \end{proof}


We arrive at the main theorem of this subsection.

\begin{thm} \label{thm:primary-decomposition-ss}
Let $z \in \FPF_n(\I_n)$.
Then $\SSJ_z = \bigcap_{D \in \FP(z)} (u_{ij} : (i,j) \in D) \subseteq \KK[\SSM_n]$. \end{thm}

\begin{proof}
    Given  $D \subseteq \ltriang_n$, let $\J{D}$ stand for the ideal $(u_{ij} : (i,j) \in D)$ in $ \KK[\SSM_n]$. As in the proof of Theorem~\ref{ss-main-thm}, we use downward induction on the size of $\ltriang_n \setminus \SSD(z)$, 
    starting at the element $z_{\square} \in \FPF_n(\I_n)$ defined by \eqref{z-square-eq}.
    Since $\SSD(z_\square) = \ltriang_n$, it follows from Proposition~\ref{fp-00-prop} that $\FP(z_\square) = \{ \ltriang_n\}$.
    Moreover, as noted in the proof of Theorem~\ref{ss-main-thm}, we have 
    $\SSJ_{z_{\square}} = \J{\ltriang_n}$. Thus, the desired formula holds for $z = z_{\square}$.
    
    
    Assume $z \in \FPF_n(\I_n)\setminus\{ z_{\square}\}$. Then,
    as in the proof of Theorem~\ref{ss-main-thm}, there exists an outer corner $(p,q) \in \ltriang_n$ of $\dom(z)$. 
    By the second formula in Lemma~\ref{lem:transition-containment-ss} and induction, we have
    \[
     \SSJ_z + (u_{pq}) = \bigcap_{y \in \Psi(z,p)} \SSJ_y =  \bigcap_{y \in \Psi(z,p)} \bigcap_{D \in \FP(y)} \J{D}
     \]
     The technical result 
     \cite[Thm. 4.33]{HMPdreams}
     asserts that $(p,q)$ is not in any pipe dream of $z$, and that $D \mapsto D \sqcup \{(p,q)\}$ is a bijection 
$\FP(z) \to \bigcup_{y \in \Psi(z,p)} \FP(y).$ Hence, we can rewrite the right side of the preceding displayed equation as 
    \[
    \bigcap_{y \in \Psi(z,p)} \bigcap_{D \in \FP(y)} \J{D}
    =
    \bigcap_{D \in \FP(z)} \J{D \sqcup \{(p,q)\}} = \bigcap_{D \in \FP(z)} (\J{D} + (u_{pq})).\]
By Lemma~\ref{lem:intersection-of-sums}, the last expression is just $ \bigcap_{D \in \FP(z)} \J{D} + (u_{pq})$.
%
%
  Thus, we have shown that  
    \[\SSJ_z + (u_{pq}) =  \bigcap_{D \in \FP(z)} \J{D} +  (u_{pq}).\] 
    Our last step is to apply Lemma~\ref{lem:hilbert} 
    with $J = \SSJ_z $, $I =  \bigcap_{D \in \FP(z)} \J{D}$, and $f=u_{pq}$ 
    to conclude that $\SSJ_z = \bigcap_{D \in \FP(z)} \J{D}$. 
  To invoke this lemma, the necessary hypotheses are that $J\subseteq I$ and $I \cap (u_{pq}) = u_{pq} I$,
   and checking these conditions is straightforward: the containment $J \subseteq I$ is  Lemma~\ref{lem:Jw-in-JD-ss},
and $(u_{pq}) \cap I = u_{pq} I$ holds
as $(p,q) \notin D$ for all $D \in \FP(z)$.
%
%
\end{proof}

\subsection{Stanley-Reisner ideals and $K$-polynomials}
\label{subsec:K}
In this section we observe that $\SSJ_z$ is the Stanley-Reisner ideal of a shellable simplicial complex, and then use this fact
to give a new proof of a combinatorial formula for the torus-equivariant $K$-theory class of $\SSX_z$ from \cite{MacdonaldNote}. 
Our approach is modeled after Knutson and Miller's study of subword complexes and the resulting combinatorial formulas for Grothendieck polynomials \cite{SubwordComplexes}.

Suppose $R = \KK[u_1, u_2,\ldots, u_N]$ is a polynomial ring that is graded by a free abelian group $G$, so that there is a decomposition $R = \bigoplus_{g \in G} R_g$ where $R_g R_h \subseteq R_{gh}$. We assume each variable $u_i$ is homogeneous and thus contained in $R_g$ for some $g \in G$, in which case we set $\deg (u_i) = g$. We further assume that the grading is \emph{positive}, meaning that there is a homomorphism $\phi : G \to \ZZ$ such that $\phi(\deg(u_i)) > 0$ for all $i$.
This condition ensures that all degrees $\deg(u_1),\deg(u_2), \ldots, \deg(u_N)$ lie strictly on one side of a hyperplane in $G$. 

Now suppose $M = \bigoplus_{g \in G} M_g$ is a finitely generated $G$-graded $R$-module, so that $R_g M_h \subseteq M_{gh}$ for all $g,h \in G$. One can show that $ \dim_\KK(M_g)$ is always finite \cite[Thm. 8.6]{MillerSturmfels}, and we define the \emph{$G$-graded Hilbert series} of $M$ to be $\Hilb(M) = \sum_{g \in G} \dim_\KK(M_g)g$. 

We view $\Hilb(M)$ as a formal generating function in the completion 
of the group ring $\ZZ[G]$. The obvious formula extending the ring structure on $\ZZ[G]$ defines a $\ZZ[G]$-module structure on this completion.
This means it is well-defined to multiply $\Hilb(M)$ by finite linear combinations of elements of $G$.

%

An ideal $I \subseteq R$
 is $G$-homogeneous if it is generated by a set of homogeneous elements. 
 Such an ideal is a $G$-graded abelian group, so the quotient $R/I$ is a $G$-graded $R$-module with its own well-defined $G$-graded Hilbert series.

\begin{defn} The \emph{K-polynomial} of a $G$-homogeneous ideal $I \subseteq R$ is the formal generating function $\K(I) = \Hilb(R/I)\prod_{i=1}^N (1 - \deg(u_i))$. \end{defn}
One can show that $\K(I)$ is a Laurent polynomial in $\deg(u_1), \ldots, \deg(u_N)$ \cite[Thm. 8.20]{MillerSturmfels}.
Fix an arbitrary term order on $R$ and define $\init(I)$ for an ideal $I\subseteq R$ relative to this order. 

\begin{lem} \label{lem:init-K} If $I \subseteq R$ is a $G$-homogeneous ideal, then $\K(I) = \K(\init(I))$. \end{lem}

\begin{proof}
This holds since $\Hilb(I) = \Hilb(\init(I))$ \cite[Thm. 15.3]{Eisenbud}.
\end{proof}

This result is useful since in many cases it is easier to compute $\K(\init(I))$ than $\K(I)$ directly.

\begin{ex}
    Suppose $R = k[x_1, x_2]$, $G =  \ZZ t_1 \oplus \ZZ t_2$, $\deg x_i = t_i$, and $I = (x_1^2)$. 
    A homogeneous basis for $R/I$ is  $\{x_1^i x_2^j + I : 0 \leq i \leq 1, 0 \leq j\}$, so 
$
        \Hilb(R/I) = \sum_{j=0}^{\infty} (1 + t_1) t_2^j = \frac{1+t_1}{1-t_2}
$
    and $\K(I) = \tfrac{1+t_1}{1-t_2}(1-t_1)(1-t_2) = 1-t_1^2$.     If $J = (x_1^2 + x_2)$ and $\init(x_1^2 + x_2) = x_1^2$, then 
    a homogeneous basis for $R/J$ is $\{x_1^i x_2^j + J : 0 \leq i \leq 1, 0 \leq j\}$, so we have $I = \init(J)$ and $\K(J) = \K(I) = \K(\init (J))$.
\end{ex}

Let $T_n \subseteq B_n \subseteq \GL_n$ denote the torus of diagonal matrices.
Write $a_i$ for the character $T_n \to \KK^\times$ sending $t \in T_n$ to $t_{ii}$. 
The subgroup $T_n$ acts on $\SSM_n$ by the formula $t \cdot A := tAt$ for $t \in T_n$.
The corresponding weight space decomposition of $\SSM_{n}$ is $\bigoplus_{n \geq i >j \geq 1} \KK (E_{ij} - E_{ji})$ where $E_{ij}$ is the $n \times n$ matrix with a $1$ in position $(i,j)$ and $0$ in all other positions. Since 
$ t \cdot (E_{ij}-E_{ji})  =
  t_{ii} t_{jj} (E_{ij} - E_{ji}),$
every $A \in \KK (E_{ij} - E_{ji})$ has weight $a_i a_j$. 

Let $G$ be the multiplicative group generated by $a_1,a_2,\dots,a_n$. The distinct elements of $G$ are the products $a_1^{e_1} a_2^{e_2}\cdots a_n^{e_n}$ with $e_1,e_2,\dots,e_n \in \ZZ$,
so $G$ is free abelian and the completion of its group ring may be identified with the formal power series ring 
$\ZZ[[a_1,a_2,\dots]]$.
The coordinate ring $\KK[\SSM_n] = \KK[u_{ij} : n \geq i > j \geq 1]$ has a unique positive $G$-grading in which 
 $\deg (u_{ij}) = a_i a_j$. Since a skew-symmetric matrix Schubert variety $\SSX_z$ is $B_n$-stable, it is also $T_n$-stable, and so its ideal $I(\SSX_z)$ is homogeneous under this grading.

\begin{defn} The \emph{symplectic Grothendieck polynomial} associated to $z \in \FPF_n(\I_n)$ is the $K$-polynomial 
$\Gfpf_z := \K(I(\SSX_z))$. \end{defn}

Although \emph{a priori} one only knows that  $\Gfpf_z \in \ZZ[[a_1,a_2,\dots]]$, it turns out that this formal generating function always has only finitely many nonzero terms.
Moreover, although each $z \in \Ifpf_\infty$ belongs to $\FPF_n(\I_n)$ for infinitely many values of $n$
and the definition of $\SSX_z$ depends on this value, the 
polynomial $\Gfpf_z$ is independent of the choice of $n$ such that $z \in \FPF_n(\I_n)$  \cite[Thm. 4]{WyserYongOSp}.

The functions $\Gfpf_z$ 
 are the same as the polynomials $\Upsilon^K_{\pi,(\GL_{2n},\Sp_{2n})}$ given in \cite{WyserYongOSp} 
as representatives for the ordinary $K$-theory classes of the $\Sp_n(\CC)$-orbit closures on $\GL_n(\CC)/B_n$.
Symplectic Grothendieck polynomials also appear in \cite{MP2019, MP2019b} (after making the change of variables $a_i \mapsto 1- x_i$)
as representatives for the $T$-equivariant $K$-theory classes of the varieties $\SSX_z$.
The latter polynomials (in $\ZZ[x_1,x_2,\dots]$) have well-defined ``stable limits''
converging to symmetric functions that expand positively in terms of Ikeda and Naruse's $K$-theoretic Schur $P$-functions (with parameter $\beta=-1$); see \cite[Thm. 1.9]{Mar} and \cite[Thm. 1.6]{MP2019}.

\begin{ex} \label{ex:groth1-ss} 
Let $D$ be a subset of $\ltriang_n = \{ (i,j) \in [n] \times [n] : i > j\}$.
As in the proof of Theorem~\ref{thm:primary-decomposition-ss}, 
write $\J{D} = (u_{ij} : (i,j) \in D ) \subset \KK[u_{ij} : n \geq i>j\geq 1] =: R.$
The set of monomials 
$ \prod_{(i,j) \in D^c} u_{ij}^{m_{ij}}$ where $D^c:= \ltriang_n \setminus D$ descends to a basis for $R/\J{D}$. Accordingly,
    \begin{equation*}
        \Hilb(R/\J{D}) = \prod_{(i,j)  \in D^c} \sum_{m_{ij} \geq 0} (a_i a_j)^{m_{ij}} 
        = \prod_{(i,j)   \in D^c} \tfrac{1}{1-a_i a_j}
    \end{equation*}
    and
    \begin{equation*}
        \K(\J{D}) = \prod_{(i,j)  \in D^c} \tfrac{1}{1-a_i a_j} \prod_{n \geq i > j \geq 1 } (1-a_i a_j) = \prod_{(i,j) \in D} (1-a_i a_j).
    \end{equation*}
If $z \in \FPF_n(\I_n)$ is \emph{fpf-dominant} in the sense that $\SSD(z) = \dom(z) \cap \ltriang_n$, then $\SSD(z)$ is the unique element of $\RFP(z)$ \cite[Lem. 4.29]{HMPdreams} and  $I(\SSX_z) = \J{\SSD(z)}$ by Theorem~\ref{thm:primary-decomposition-ss}.
We conclude that
   $
        \Gfpf_z = \prod_{(i,j) \in \SSD(z)} (1-a_i a_j),
   $
 which recovers the skew-symmetric half of \cite[Thm. 3.8]{MP2019}.  
    \end{ex}


$K$-polynomials satisfy the inclusion-exclusion formula $\K(I \cap J) = \K(I) + \K(J) - \K(I+J)$. From this and Lemma~\ref{lem:intersection-of-sums} one deduces that $\K(\bigcap_{i=1}^m I_i) = \sum_{S \subseteq [m]} (-1)^{|S|-1} \K(\sum_{i \in S} I_i)$ for any monomial ideals $I_1,I_2, \ldots, I_m$. In particular, Theorem~\ref{thm:primary-decomposition-ss} implies that
\begin{equation} \label{eq:bad-K-formula}
    \Gfpf_z = {\textstyle \K(\bigcap_{D \in \RFP(z)} \J{D})} = \sum_{S \subseteq \RFP(z)} (-1)^{|S|-1} \prod_{(i,j) \in \bigcup S} (1-a_i a_j).
\end{equation}
This formula involves too much cancellation to be efficient: for $z \in \FPF_n(\I_n)$ the size of $\RFP(z)$ can be exponential in $n$, in which case the number of terms in \eqref{eq:bad-K-formula} is doubly exponential.

To get something more practical, we will simplify the formula \eqref{eq:bad-K-formula} using inclusion-exclusion on the poset $\{\bigcup S : S \subseteq \RFP(z)\}$ ordered by inclusion---or more precisely, by taking complements and studying the structure of the simplicial complex generated by facets $\{\ltriang_n \setminus D : D \in \RFP(z)\}$.
For us, a \emph{simplicial complex with vertex set $V$} means a collection of subsets of $V$ (which we refer to as \emph{faces}) closed under taking subsets. We do not require  $\{v\}$ to be a face for every element $v\in V$.
The containment-maximal faces of a simplicial complex are its \emph{facets}.

\begin{defn}
    The \emph{fpf-subword complex} associated to $z \in \FPF_n(\I_n)$ and a subset $Q \subseteq \ltriang_n$ is the simplicial complex with vertices $Q$ and faces
\begin{equation*}
    \subword(z,Q) = \{S \subseteq Q : \text{$Q \setminus S$ contains an fpf-involution pipe dream for $z$} \}.
\end{equation*}
This closely mirrors Knutson and Miller's definition of a \emph{subword complex} in \cite[\S2]{SubwordComplexes}.
\end{defn}

Given a set $V$, let $\KK[x_v : v \in V]$ be the polynomial ring generated by a set of commuting indeterminates indexed by  $V$.
The \emph{Stanley-Reisner ideal} of a simplicial complex $\Delta$ with vertex set $V$ is the ideal in $\KK[x_v : v \in V]$
generated by the elements $\prod_{v \in N} x_v $ for all  $N \subseteq V$ with $N \notin \Delta$. 
\begin{thm} \label{thm:SR-ideal}
If $z \in \FPF_n(\I_n)$ then
    $\SSJ_z$ is the Stanley-Reisner ideal of $\subword(z,\ltriang_n)$.
\end{thm}

\begin{proof} Theorem~\ref{thm:primary-decomposition-ss} says that $\SSJ_z = \bigcap_{D \in \RFP(z)} (u_{ij} : (i,j) \in D)$; now apply \cite[Thm. 1.7]{MillerSturmfels}.
\end{proof}


\begin{ex} \label{ex:subword-ss}
    Let $n=6$ and define $z$ to be the image of $(1,2)(3,6)(4,5)$ under $\FPF_6$, so that
    \[
    z= (1,2)(3,6)(4,5)(7,8)(9,10)(11,12)\cdots \in \FPF_6(\I_6)\subset \Ifpf_\infty.
    \] Suppose $Q = \{(3,1),(3,2),(4,1),(4,2),(5,1)\}$, so that, abbreviating the pair $(i,j)$ as $ij$, we have \[Q = \left\{\begin{smallmatrix}
        \cdot&  & \phantom{13} &  \phantom{14} & \phantom{15} & \phantom{16}  \\
       \cdot &   \cdot&  &  &  &  \\
       31 & 32 &  \,\cdot\, &   &  &  \\
       41 & 42 & \cdot & \cdot &  &  \\
       51 & \cdot & \cdot & \cdot& \cdot &  \\
       \cdot & \cdot & \cdot & \cdot & \cdot & \cdot \\
       \end{smallmatrix}\right\}.\]
       The set of pipe dreams $\FP(z)$ is 
       \begin{equation*}
       \FP(z) = \left\{\left\{\begin{smallmatrix}
            \cdot& \phantom{+}  &  \phantom{+}    &  \phantom{+}    & \phantom{+}     & \phantom{+}    \\
           \cdot &   \cdot&  &  &  &  \\
           \cdot & \cdot & \cdot &  &  &  \\
           + & \cdot & \cdot & \cdot &  &  \\
           + & \cdot & \cdot & \cdot& \cdot &  \\
           \cdot & \cdot & \cdot & \cdot & \cdot & \cdot \\
           \end{smallmatrix}\right\},
           \left\{\begin{smallmatrix}
             \cdot& \phantom{+}  &  \phantom{+}    &  \phantom{+}    & \phantom{+}     & \phantom{+}    \\
           \cdot &   \cdot&  &  &  &  \\
           \cdot & + & \cdot &  &  &  \\
           \cdot & \cdot & \cdot & \cdot &  &  \\
           + & \cdot & \cdot & \cdot& \cdot &  \\
           \cdot & \cdot & \cdot & \cdot & \cdot & \cdot \\
           \end{smallmatrix}\right\},
           \left\{\begin{smallmatrix}
            \cdot& \phantom{+}  &  \phantom{+}    &  \phantom{+}    & \phantom{+}     & \phantom{+}    \\
           \cdot &   \cdot&  &  &  &  \\
           \cdot & + & \cdot &  &  &  \\
           \cdot & + & \cdot & \cdot &  &  \\
           \cdot & \cdot & \cdot & \cdot& \cdot &  \\
           \cdot & \cdot & \cdot & \cdot & \cdot & \cdot \\
           \end{smallmatrix}\right\},
           \left\{\begin{smallmatrix}
             \cdot& \phantom{+}  &  \phantom{+}    &  \phantom{+}    & \phantom{+}     & \phantom{+}    \\
           \cdot &   \cdot&  &  &  &  \\
           + & + & \cdot &  &  &  \\
           \cdot & \cdot & \cdot & \cdot &  &  \\
           \cdot & \cdot & \cdot & \cdot& \cdot &  \\
           \cdot & \cdot & \cdot & \cdot & \cdot & \cdot \\
           \end{smallmatrix}\right\}
           \right\},
        \end{equation*}
        and  the facets of $\Sigma(z,Q)$ are $\{31,32,42\}$, $\{31,41,42\}$, $\{31,41,51\}$, and  $\{41,42,51\}$. The collection of all subsets of these four facets gives 
        $\Sigma(z,Q)$, which we can draw as  
        \begin{center}
            \begin{tikzpicture}[scale=2]
                \filldraw[fill=black!10!white] (0,-0.5) -- (0,0.5) -- (0.866,0) -- (0,-0.5);
                \filldraw[fill=black!10!white] (0,-0.5) -- (-0.289,0) -- (0,0.5);
                \draw (0,0.5) -- (0.289,0) -- (0,-0.5);
                \draw (0.289,0) -- (0.866,0);
                \draw (0,0.5) -- (0,-0.5);

                \filldraw[fill=black] (0,0.5) circle[radius=0.2mm];
                \filldraw[fill=black] (0,-0.5) circle[radius=0.2mm];
                \filldraw[fill=black] (0.866,0) circle[radius=0.2mm];
                \filldraw[fill=black] (0.289,0) circle[radius=0.2mm];
                \filldraw[fill=black] (-0.289,0) circle[radius=0.2mm];

                \draw (-0.289,0) node[left] {$\scriptstyle 32$};
                \draw (0,0.5) node[above] {$\scriptstyle 31$};
                \draw (0,-0.5) node[below] {$\scriptstyle 42$};
                \draw (0.27,-0.02) node[above right] {$\scriptstyle 41$};
                \draw (0,0.5) node[above] {$\scriptstyle 31$};
                \draw (0.866,0) node[right] {$\scriptstyle 51$};
            \end{tikzpicture}
        \end{center}
        The minimal non-faces of $\subword(z,Q)$ are $\{\{32,41\}, \{32,51\}, \{31,42,51\}\}$.  If $v \in \ltriang_n \setminus Q$ then $v$ appears in no fpf-involution pipe dream of $z$, and hence appears in every maximal face of $\subword(z,\ltriang_n)$.
        This implies that $\subword(z,\ltriang_n)$ and $\subword(z,Q)$ have the same minimal non-faces. We can therefore deduce from Theorem~\ref{thm:SR-ideal} that $\SSJ_z = (u_{32}u_{41}, u_{32}u_{51}, u_{31}u_{42}u_{51})$ in accordance with Example~\ref{ex:SSI-generators}. It is clear from the picture above that $\subword(z,Q)$ is homeomorphic to a closed ball as predicted by Lemma~\ref{lem:ball-or-sphere}.
\end{ex}

If $\Delta$ is a simplicial complex with vertex set $V$ and $F \subseteq V$, then the \emph{link} of $F$ in $\Delta$ is 
the simplicial complex $\link_F \Delta := \{E \in \Delta : E \cap F = \emptyset, E \cup F \in \Delta\}$ with vertex set $V \setminus F$.

\begin{lem} \label{lem:subword-link} For any $Q, D \subseteq \ltriang_n$ and $z \in \FPF_n(\I_n)$, we have $\link_D \subword(z, Q) = \subword(z, Q \setminus D)$. \end{lem}

\begin{proof}
    The conditions $F \cap D = \emptyset$ and $F \cup D \in \subword(z, Q)$ hold precisely when $F \subseteq Q \setminus D$ and $(Q \setminus D) \setminus F$ contains a pipe dream in $\RFP(z)$, which says exactly that $F \in \subword(z, Q \setminus D)$.
\end{proof}

 If $\Delta$ is a simplicial complex with vertex set $V$, 
 then the \emph{deletion} of $F \subseteq V$ is the simplicial complex $\del_F \Delta := \{E \in \Delta : E \cap F = \emptyset\}$ with vertex set $V \setminus F$. A complex $\Delta$ with non-empty vertex set $V$ is \emph{vertex-decomposable} if it is \emph{pure}
 (in the sense that all containment-maximal faces have the same dimension) and there is $v \in V$ such that $\del_{\{v\}} \Delta$ and $\link_{\{v\}} \Delta$ are vertex-decomposable; as a base case, the complexes $\{\emptyset\}$ and $\emptyset$ are both vertex-decomposable. 
 Results in \cite{BilleraProvan} show that any vertex-decomposable complex is 
 \emph{shellable} in the sense of \cite[Def. 2.4]{SubwordComplexes}.
 

The next two proofs will use one other concept:
an \emph{fpf-involution word} for $z \in \FPF_n(\I_n)$ is a minimal-length sequence of positive integers $a_1a_2\cdots a_l$
such that $z = s_{a_l}\cdots s_{a_2}s_{a_1} 1_\fpf s_{a_1} s_{a_2} \cdots s_{a_l}$.
Every such word has length $\ellfpf(z)$, and a subset $D \subseteq \ltriang_n$ 
belongs to $\FP(z)$ if and only if its pipe dream reading word is an fpf-involution word
for $z$ \cite[Thm. 3.12]{HMPdreams}. 

\begin{thm} \label{thm:shellable} 
Suppose $z \in \FPF_n(\I_n)$ and $Q \subseteq \ltriang_n$.
Then $\Delta(z, Q)$ is vertex-decomposable and therefore shellable. \end{thm}

\begin{proof}
Our argument is similar to the proof of \cite[Thm. 2.5]{SubwordComplexes}.
Lemma~\ref{lem:subword-link} shows that $\link_{\{\square\}} \subword(z,Q) = \subword(z,Q \setminus \{\square\})$ for any $\square \in Q$, and we will check that also $\del_{\{\square\}} \subword(z,Q) = \subword(z', Q\setminus\{\square\})$ for some $z'\leq z$. By induction on the size of $Q$, this will prove that $\subword(z,Q)$ is vertex-decomposable.

To this end, we  number the elements of $Q = \{(i_1,j_1) \prec (i_2,j_2) \prec\cdots \prec (i_l, j_l)\}$ in the order $\dots  \prec (1,3) \prec (1,2) \prec (1,1) \prec \cdots \prec (2,3) \prec (2,2) \prec (2,1) \prec \cdots$ that makes $(i,-j)$ increase lexicographically. Set $\square = (i_l, j_l)$, $q = i_l+j_l-1$, and $Q' = Q \setminus \{\square\}$. There are two cases:
    \begin{itemize}
        \item Assume $\ellfpf(s_q zs_q) > \ellfpf(z)$. 
        If there exists a pipe dream $D \in \RFP(z)$ with $D \subseteq Q$ and $\square \in D$, then the pipe dream reading word $\word(D)$ ends in $q$, which means $z$ has an fpf-involution word ending with $q$. This contradicts our hypothesis for this case, so $\square$ is not contained in any $D \in \RFP(z)$ with $D \subseteq Q$. We deduce  that 
        $\del_{\{\square\}} \subword(z,Q) =  \subword(z, Q')$.

         \item Assume instead that $\ell(s_q z s_q) \leq \ell(z)$. 
         We wish to show that $\del_{\{\square\}} \subword(z,Q) = \subword(s_q z s_q, Q')$. 
        
         Suppose $E \subseteq Q'$ is a set contained in $\del_{\{\square\}} \subword(z,Q)$. Then $D \subseteq Q \setminus E$ for some $D \in \RFP(z)$, and we have  $Q' \setminus E\supseteq D' := D \setminus \{\square\}$. 
          If $\square \in D$, then $z = s_q y s_q$ where $\ellfpf(y) = \ellfpf(z)-1$. 
          In this case we cannot have $s_q z s_q = z$ so it holds that $D' \in \RFP(s_q z s_q) = \RFP(y)$.
          On the other hand, if $\square \notin D$, then $D = D'$
          contains an element of $\RFP(s_q z s_q)$ by \cite[Lem. 4.31]{HMPdreams}
          since $s_q z s_q \leq z$ in Bruhat order.
          %
         Either way, $Q' \setminus E$ contains an element of $\RFP(s_q z s_q)$, so $E \in \subword(s_q z s_q, Q')$ and we conclude that $\del_{\{\square\}} \subword(z,Q) \subseteq \subword(s_q z s_q, Q')$.
         
         Conversely, suppose $E \in \subword(s_q z s_q, Q')$, so that $D \subseteq Q' \setminus E$ for some $D \in \RFP(s_q z s_q)$. If $\ellfpf(s_q z s_q) < \ellfpf(z)$, then $D \sqcup \{\square\} \in \RFP(z)$ and $D \sqcup \{\square\} \subseteq Q \setminus E$, while if $\ellfpf(s_qzs_q) = \ellfpf(z)$ then $s_q z s_q = z$ and $D \in \RFP(z)$. Thus $Q \setminus E$ contains an element of $\RFP(z)$ and $E \in \del_{\{\square\}} \subword(z, Q)$, so we have $\del_{\{\square\}} \subword(z,Q) = \subword(s_q z s_q, Q')$.
    \end{itemize}
  Thus $\del_{\{\square\}} \subword(z,Q) = \subword(z', Q')$ for $Q ' = Q\setminus\{\square\}$ and some $z'\leq z$.
    As noted above,
    this lets us conclude by induction that $\subword(z,Q)$ is vertex-decomposable and therefore shellable via  \cite{BilleraProvan}.
\end{proof}

\begin{rem} Our definition of ``vertex-decomposable'' differs slightly from that in \cite{BilleraProvan} and \cite{SubwordComplexes}: first, we include not only the complex $\{\emptyset\}$ but also $\emptyset$; second, we allow for the possibility that the vertex $v$ for which $\del_{\{v\}} \Delta$ and $\link_{\{v\}} \Delta$ are vertex-decomposable is not a face of $\Delta$. Except for the inclusion of the complex $\emptyset$, the definitions are equivalent, because if $\{v\} \notin \Delta$, then $\del_{\{v\}} \Delta = \Delta$ and $\link_{\{v\}} \Delta = \emptyset$. 
A close reading of the 
 proof of \cite[Thm. 2.5]{SubwordComplexes} 
 suggests that our convention may in fact be what was originally intended in \cite{SubwordComplexes}, in any case.
 %
  \end{rem}

If $z \in \{1\} \sqcup \Ifpf_\infty$ and $s = s_i = (i,i+1) \in S_\infty$, then we define 
\[ z \ast s = \begin{cases} 1 & \text{if $(i,i+1)$ is a cycle of $z$} \\
z & \text{if }z=1\text{ or }z(i) > z(i+1) \\
szs&\text{otherwise}.
\end{cases}
\]
Note that if $z \in \Ifpf_\infty$ then either $z \ast s \in \Ifpf_\infty$ or $z\ast s = 1$, but we always have $1 \ast s = 1$.
The operation $\ast$ extends to a right action of the \emph{$0$-Hecke monoid} of $S_\infty$ but not to a group action. If $i_1  i_2\cdots i_l$ is a word, define $\deltafpf(i_1 i_2 \cdots i_l) = (\cdots (( 1_{\fpf} \ast s_{i_1}) \ast s_{i_2}) \cdots) \ast s_{i_l}$.

\begin{defn}
An \emph{extended (fpf-involution) pipe dream} for $z \in \Ifpf_\infty$ is a subset $D \subseteq  \ltriang$
whose pipe dream reading word $i_1i_2\cdots i_l$ satisfies $z = \deltafpf(i_1 i_2 \cdots i_l)$.
Let $\NFP(z)$ be the set of all extended pipe dreams for $z$.
\end{defn}

 The set  $\NFP(z)$ is called $\mathsf{InvDreams}(z)$ in \cite{MacdonaldNote}.
Each $D \in \NFP(z)$ has $|D|\geq \ellfpf(z)$, with equality if and only if  
$D \in \FP(z)$  \cite[Thm. 3.12]{HMPdreams}.
The containment $\FP(z) \subset \NFP(z)$ is strict unless $z =1_\fpf$.

In the next lemma, we consider the complex $\{\emptyset\}$ to be a $(-1)$-dimensional sphere and $\emptyset$ to be a ball. These conventions may appear somewhat recherch\'e. Ultimately, however, all we need out of the lemma are the reduced Betti numbers $\dim \tilde{H}_i(\subword(z,Q); \KK)$ as computed via simplicial homology, and one checks that these are $\delta_{i,-1}$ and $0$ when $\subword(z,Q)$ is $\{\emptyset\}$ or $\emptyset$, respectively.

\begin{lem} \label{lem:ball-or-sphere} 
Suppose $z \in \FPF_n(\I_n)$ and $Q \subseteq \ltriang_n$.
Then $\subword(z, Q)$ is homeomorphic to a sphere of dimension $|Q|-\ellfpf(z)-1$
when $Q \in \NFP(z)$ and to a closed ball when $Q \notin \NFP(z)$. \end{lem}

\begin{proof} 
The facets of $\subword(z,Q)$ are the sets $Q \setminus D$ for all $D \in \RFP(z)$ such that $D \subseteq Q$.
Let  $F$ be a codimension one face of $\subword(z,Q)$, that is, a face contained in
a facet of size $|F|+1$. Then 
$Q \setminus F$ is a set of size $\ellfpf(z)+1$ from which one (not necessarily unique) cell can be removed to form a pipe dream in $\FP(z)$.

As explained in \cite[\S 3]{KnutsonMiller}, 
since $\subword(z,Q)$  is shellable by Theorem~\ref{thm:shellable},
it suffices by \cite[Prop. 4.7.22]{BLSWZ99} to check
(a) that $F$ is contained in at most two facets of $\subword(z,Q)$,
 (b) that
$F$ is contained in exactly two facets of $\subword(z,Q)$ if $Q \in \NFP(z)$,
and (c) if $ Q \notin \NFP(z)$ then there is some choice of $F$  that is
contained in a unique facet.

To prove (a), we observe that if $F$ were contained in three distinct facets, then $Q \setminus F$ would contain three distinct pipe dreams in $ \RFP(z)$.
This would mean that
there are three distinct letters that can be deleted from the pipe dream reading word of $F$, which has length $\ellfpf(z)+1$, to form an fpf-involution word for $z$. It is a basic exercise to show that this is impossible; the argument is the same as what is needed to prove that if $w \in S_\infty$, then a word of length $\ell(w)+1$ may contain at most two distinct subwords that are reduced words for $w$.

Claim (b) follows immediately from \cite[Lem. 4.21]{HMP3}, 
which asserts that if 
$Q \in \NFP(z)$ then there are exactly two cells $c \in Q \setminus F$ with $Q \setminus F \setminus \{c\} \in \RFP(z)$. 
Finally, to prove (c), suppose $Q \notin \NFP(z)$. It remains to show that some  $F$ is contained in exactly one facet.

 We have defined the complex $\emptyset$ to be a ball, so we may assume that $\subword(z,Q)$ is non-empty and that $Q$ contains some element of $\FP(z)$. Take $D $ to be maximal under inclusion such that $D\subseteq Q$ and $D \in \NFP(z)$. 
 Since $D \neq Q$ as $Q \notin \NFP(z)$, there exists $D \subseteq D' \subseteq Q$ with $|D' \setminus D| = 1$. By maximality, we have $ z\neq \deltafpf(\word(D')) \in \{1\} \sqcup \Ifpf_\infty$, so there are two possibilities.
         \begin{enumerate}[(i)]
            \item Suppose $ \deltafpf(\word(D'))\neq 1$. Since $\word(D')$ contains an fpf-involution word for $z$ by hypothesis, the subword criterion for the Bruhat order on $S_\infty$ and Lemma~\ref{ss-bruhat-lem} imply that $z< \deltafpf(\word(D')) \in \Ifpf_\infty$. We therefore can choose $z' \in \Ifpf_\infty$ with $\deltafpf(\word(D')) \geq z' > z$ and $\ellfpf(z') = \ell(z) + 1$. Then $D'$ contains some $E \in \RFP(z')$ by \cite[Lem. 4.31]{HMPdreams}, and we 
            take $F = Q\setminus E$. This is a codimension one face
            and we claim that it is contained in a unique facet.
            Indeed, if $F$ were contained in two facets of $\subword(z,Q)$
            then $\word(E)$, which is an fpf-involution word for $z'$,
            would contain two distinct subwords that were fpf-involution words for $z$.
            But this is impossible; the argument is  the same as the proof that if $w,w' \in S_\infty$ and $\ell(w') = \ell(w)+1$, then a reduced word for $w'$ may contain at most one subword that is a reduced word for $w$.
            

            \item Suppose $\deltafpf(\word(D')) = 1$. This only happens if we can write $\word(D') = aqb$ where $a$ and $b$ are subwords with $\word(D) = ab$, $q$ is a positive integer,  and the subset $A \subseteq D'$ contributing the subword $a$ belongs to $\NFP(y)$ for some $y \in \Ifpf_\infty$ with $y(q) = q+1$. Let $\square$ be the element of $D'$ corresponding to the letter $q$ in $\word(D') = aqb$. Since $D \in \NFP(z)$, there is $E \subseteq D$ with $E \in \RFP(z)$, and we can further assume that $E \cap A \in \RFP(y)$. 
            
            The set $F = Q \setminus (E \sqcup \{\square\})$ is a codimension one face contained in  $Q \setminus E$, and we claim it is contained in no other facet. Suppose for the sake of contradiction that $E \sqcup \{\square\}$ contains some $E' \in \RFP(z)$ distinct from $E$. Writing $\word(E \sqcup \{\square\}) = a_1 \cdots a_k q b_1 \cdots b_l$ and $\word(E) = a_1 \cdots a_k b_1 \cdots b_l$, where $\word(E \cap A) = a_1 \cdots a_k$, we have either $\word(E') = a_1 \cdots a_k q b_1 \cdots \widehat{b}_i \cdots b_l$ or $\word(E') = a_1 \cdots \widehat{a}_i \cdots a_k q b_1 \cdots b_l$ for some $i$.  The first case
            implies 
            \[   
            \begin{aligned}
            z &=
            s_{b_l} \cdots  s_{b_1} s_{a_k} \cdots  s_{a_1} 1_{\fpf}  s_{a_1} \cdots s_{a_k}  s_{b_1}  \cdots s_{b_l}
            \\ &
           =  s_{b_l} \cdots \widehat{s}_{b_i} \cdots s_{b_1} s_{q} s_{a_k} \cdots  s_{a_1} 1_{\fpf}  s_{a_1} \cdots s_{a_k} s_q s_{b_1} \cdots \widehat{s}_{b_i} \cdots s_{b_l}
            \end{aligned}
            \]
          but if this holds then it is easy to deduce from  $y=s_{a_k} \cdots  s_{a_1} 1_{\fpf}  s_{a_1} \cdots s_{a_k}$ that
         \[    z = s_{b_l} \cdots \widehat{s}_{b_i} \cdots s_{b_1} s_{a_k} \cdots  s_{a_1} 1_{\fpf}  s_{a_1} \cdots s_{a_k} s_{b_1} \cdots \widehat{s}_{b_i} \cdots s_{b_l}
         \]
         which contradicts the fact that $a_1 \cdots a_k  b_1 \cdots b_l$ is an fpf-involution word for $z$.
     
    Similarly, in the second case we have
            \begin{align*}
                z &= s_{b_l} \cdots s_{b_1} s_{a_k} \cdots s_{a_1} 1_{\fpf} s_{a_1} \cdots s_{a_k} s_{b_1} \cdots s_{b_l}\\
                &= s_{b_l} \cdots s_{b_1} s_q s_{a_k} \cdots \widehat{s}_{a_i} \cdots s_{a_1} 1_{\fpf} s_{a_1} \cdots \widehat{s}_{a_i} \cdots s_{a_k} s_q s_{b_1} \cdots s_{b_l}.
            \end{align*} 
            Canceling the factors $s_{b_j}$ gives 
            \begin{equation*}
            y=
             s_{a_k}\cdots s_{a_1} 1_{\fpf} s_{a_1}\cdots s_{a_k} =
                s_q s_{a_k} \cdots \widehat{s}_{a_i} \cdots s_{a_1} 1_{\fpf} s_{a_1} \cdots \widehat{s}_{a_i} \cdots s_{a_k} s_q .
            \end{equation*}
            As $a_1\cdots a_k$ is an fpf-involution word for $y$, this implies that 
            $a_1 \cdots \widehat{a_i} \cdots a_k q$ is another fpf-involution for $y$.
            But this is also impossible since $s_q y s_q = y$.
        \end{enumerate}
In each case, we have identified a codimension one face 
that is contained in a unique facet of $\subword(z,Q)$.
The verifies claim (c) and completes our proof of the lemma.
\end{proof}


We conclude with this application:

\begin{thm}[{\cite[Thm. 4.5]{MacdonaldNote}}]\label{4.5-thm} Let $z \in \FPF_n(\I_n)$. Then
    \begin{equation*}
        \Gfpf_z = \sum_{D \in \NFP(z)} (-1)^{|D|-\ellfpf(z)} \prod_{(i,j) \in D} (1 - a_i a_j).
    \end{equation*}
\end{thm}

\begin{rem}
 Our first proof of this result  in \cite{MacdonaldNote}
 was by a direct algebraic method,  proceeding from an explicit generating function for $\Gfpf_z$ similar to the  Billey-Jockusch-Stanley formula \cite[Thm. 1.1]{BJS}.  
The identity given here is the specialization of \cite[Thm. 4.5]{MacdonaldNote}
with $x_i = 1-a_i$ and $\beta=-1$.
We recover the original theorem by setting $a_i = 1+\beta x_i$ and then dividing both sides 
by $(-\beta)^{\ellfpf(z)}$.

\end{rem}

\begin{proof}
    We have $\Gfpf_z = \K(I(\SSX_z)) = \K(\init (I(\SSX_z))) = \K(J_z)$ by Theorem~\ref{ss-main-thm} and Lemma~\ref{lem:init-K}. 
    The fact that $J_z$ is the Stanley-Reisner ideal of $\subword(z, \ltriang_n)$ by Theorem~\ref{thm:SR-ideal} implies that
\[
        \Gfpf_z = \sum_{D \subseteq \ltriang_n} -\tilde{\chi}(\link_{D^c} \subword(z, \ltriang_n)) \prod_{(i,j) \in D} (1-a_i a_j),
\]
    where $D^c := \ltriang_n \setminus D$ and $\tilde{\chi}$ is the reduced Euler characteristic over $\KK$, by the discussion in \cite[\S 4]{KnutsonMiller}. We have $\link_{D^c} \subword(z, \ltriang_n) = \subword(z, D)$ by Lemma~\ref{lem:subword-link}. 
The theorem follows by substituting the identity
    \[ -\tilde{\chi}(\link_{D^c} \subword(z, \ltriang_n)) = \begin{cases} 0 &\text{if }D \notin \NFP(z) \\
     (-1)^{|D|-\ellfpf(z)} &\text{if }D \in \NFP(z),\end{cases}
     \]
    which holds since in the first case $\subword(z, D)$  is homeomorphic to a ball and 
     in the second case $\subword(z, D)$ is homeomorphic to a sphere of dimension $|D|-\ellfpf(z)-1$ by Lemma~\ref{lem:ball-or-sphere}.
     %
\end{proof}

\section{New proofs of classical results}\label{new-proofs-sect}

The inductive strategy used to derive our main theorems 
can be adapted to give new proofs of the results of Knutson and Miller  \cite{KnutsonMiller} that we presented in an abridged form as Theorems~\ref{thm:KM-ideal} and \ref{thm:KM-primary-decomp} in the introduction.
We briefly indicate how this is done, omitting the details that are similar to the skew-symmetric case.
This approach is likely already known to experts, but we believe it has not been written down explicitly in any prior literature.

Fix $m,n\in \NN$ and 
identify the coordinate ring $\KK[\Mat_{m\times n}]$ with $\KK[u_{ij} : (i,j) \in [m]\times [n]]$ 
by viewing $u_{ij}$  as the map $A \mapsto A_{ij}$.
Write $\Umat $ for the matrix of variables $[u_{ij}]_{(i,j) \in [m]\times [n]}$.
\begin{defn} Given $w \in S^{m,n}_\infty$, let
$I_w$
be the ideal in $\KK[u_{ij} : (i,j) \in [m]\times [n]]$ generated by the minors $\det(\Umat_{AB})$,
where $(A,B) \in \binom{[i]}{q}\times \binom{[j]}{q}$ for some $(i,j) \in [m]\times [n]$ with $q =  \rank(w_{[i][j]})+1$. 
\end{defn}

 Since 
$\rank A \leq r$ if and only if all size $r+1$ minors of $A$ vanish, $X_w$ is the zero locus of $I_w$. 
In contrast to the skew-symmetric case, where the analogue of $I_w$ is not always prime (see Example~\ref{not-prime-ex}),
the ideal $I_w$ will turn out to be the (prime) ideal of the variety $X_w$ \cite{KnutsonMiller}.


\begin{defn}
For a rank table ${\r}$, let $J_{{\r}}$ be the ideal 
in $\KK[ u_{ij} : (i,j) \in [m]\times [n]]$
generated by 
all monomials  $u_{AB}$ as defined in \eqref{odot-eq}, where 
$(A, B) \in \binom{[i]}{q}\times \binom{[j]}{q}$
for some $(i,j) \in [m]\times [n]$ with $q = \r(i,j)+1$.
If $w \in S^{m,n}_\infty$,
then we define $J_w := J_{\r_w}$.
 \end{defn}

 %
From this point on, we fix an arbitrary antidiagonal term order on $\KK[u_{ij} : (i,j) \in [m]\times [n]]$.
This could be the reverse lexicographic term order from Example~\ref{lex-term2} once again, 
but other choices are possible.
 For any such term order, it holds that $u_{AB} = \init(\det(\Umat_{AB}))$,
so the ideal $J_w$ is automatically a subset of $\init(I_w)$.
We will see shortly that this inclusion is equality  \cite{KnutsonMiller}.

\begin{ex}
Suppose $m=n=3$ and $w=2143 \in S_4 \cap S^{3,3}_\infty$.
It follows from Proposition~\ref{thm:fpf-essential-set} that 
 $X_{2143}$ is the set of matrices $A \in M_{3,3}$ with $ \rank A_{[1][1]} \leq 0$ and $\rank A_{[3][3]} \leq 2$.
One can show that similarly $
    I_w = (\det \Umat_{[1][1]}, \det \Umat_{[3][3]})$ (see \cite[Lem. 3.10]{FultonEssentialSet}) and $J_w = (u_{11}, u_{31}u_{22}u_{13})$.
Our definitions of these ideals involve many other generators that turn out to be redundant.
\end{ex}

The following results have the same proofs 
as their skew-symmetric versions (namely, Propositions~\ref{m-ss-prop}, \ref{ssj-bruhat-prop}, and \ref{prop:min-to-sum-ss}),
\emph{mutatis mutandis}. We omit the details.

\begin{prop} \label{m-prop}
If $M$ is a monomial in $ \KK[u_{ij} : (i,j) \in [m]\times [n]]$
and $\r = \r_M$ is the rank table defined before Proposition~\ref{prop:rM-equals-rw}, then $M \notin J_{\r}$. 
\end{prop}

\begin{prop}\label{j-bruhat-prop}
  If $v,w \in S^{m,n}_\infty$ and $v \leq w$, then $J_v \subseteq J_w$. 
  \end{prop}

    \begin{prop} \label{prop:min-to-sum} 
If ${\r}$ and ${\s}$ are rank tables then $J_{\min({\r}, {\s})} = J_{\r} + J_{{\s}}$. 
\end{prop}

Using these statements in place of Propositions~\ref{m-ss-prop}, \ref{ssj-bruhat-prop}, and \ref{prop:min-to-sum-ss},
while 
swapping out relevant notation, one can 
repeat the proof of 
Theorem~\ref{thm:bruhat-minimal-intersection-ss} verbatim
to derive the following identity.

\begin{thm}\label{thm:bruhat-minimal-intersection}
 Let ${\r}$ be a rank table. Then $J_{\r} = \bigcap_w J_w$ where $w$ runs over the 
 elements of $S^{m,n}_\infty$ that are  Bruhat-minimal among those with ${\r}_w(i,j) \leq {\r}(i,j)$ for all $(i,j) \in [m]\times [n]$.
\end{thm}

Given $w \in S_\infty$ and a positive integer $p$, define 
\begin{equation*}
    \Phi(w,p) = \{w(p,r) \in S_\infty : p < r \text{ and } \ell(w(p,r)) = \ell(w) + 1\},
\end{equation*}
Recall that an outer corner of $\dom(w)$ is a pair 
$(i,j) \in (\NN\times \NN)\setminus \dom(w)$ satisfying \eqref{outer-corner-eq}.

\def\cW{\mathcal{W}}

\begin{lem} \label{lem:bruhat-minimal} Suppose $w \in S_{\infty}$
and $(p,q)$ is an outer corner of $\dom(w)$. Then 
$w(p) = q$ and $\Phi(w,p)$ is the set of Bruhat-minimal elements of
\begin{equation*}
\begin{aligned} 
\cW(w, p) &= \left\{v \in S_\infty : v \geq w, \text{ $\rank v_{[i][j]} = 0$ for $(i,j) \in [p] \times [q]$}\right\} 
\\
&=  \left\{v \in S_\infty : v \geq w, \ \dom(v)\supseteq [p] \times [q]\right\}.
\end{aligned}
\end{equation*}
In addition, if $ w\in S^{m,n}_\infty$ and $(p,q) \in [m]\times [n]$ then $\Phi(w,p) \subset S^{m,n}_\infty$.
 \end{lem}
 
\begin{proof}[Proof sketch]
As in the proof of Lemma~\ref{lem:bruhat-minimal-fpf},
one first checks that $\Phi(w,p) = \{ v \in \cW(w,p): \ell(v) = \ell(w) +1\}$  
 using basic properties of the Bruhat order and Proposition~\ref{prop:bruhat-order}.
 Next, one argues that if $u \in \cW(w,p)$ and $\ell(u) - \ell(w) \geq 2$ then $u$ is not a Bruhat-minimal element.
 One can deduce this claim by repeating the second half of the proof of   Lemma~\ref{lem:bruhat-minimal-fpf},
changing the obvious things needing to be changed, and replacing Lemma~\ref{ijr-prime-lem-fpf}
by Lemma~\ref{ijr-prime-lem} below.\end{proof}
 
 \begin{lem}\label{ijr-prime-lem}
Suppose $w \in S_{\infty}$
and $(p,q)$ is an outer corner of $\dom(w)$.
Assume $i,j,r$ are positive integers with $i<j$, $p<r$,
and 
$w < w(i,j) < w(i,j)(p,r).$
Then there are integers $i',j',r' \in \{i,j,p,r\}$ 
with $i'<j'$, $p<r'$, and
$w < w(p,r') < w(p,r')(i',j') = w(i,j)(p,r).$
\end{lem}

\begin{proof}
Write $\phi : [n]\to \{ p,r,i,j\} $ and $\psi : \{ w(p), w(r), w(i), w(j)\} \to [n]$ 
for the unique order-preserving bijections where $n := |\{ p,r,i,j\}|\leq 4$.
Then $(p',q') := (\phi^{-1}(p), \psi(q)) \in [n]\times [n]$ is an outer corner of $w' := \psi \circ w \circ\phi \in S_n$,
and if the lemma fails for $w$ and $(p,q)$ then it must fail for $w'$ and $(p',q')$.
But a finite calculation shows that the lemma holds whenever $w \in S_n$ and $n\leq 4$.
\end{proof}

\begin{cor} \label{cor:Xw-transitions} Let $w \in S^{m,n}_\infty$ and suppose $(p,q) \in [m]\times [n]$ is an outer corner of $\dom(w)$. Then $X_w \cap X_{p,q} = \bigcup_{v \in \Phi(w,p)} X_v$
for
$X_{p,q}:= \{A \in \Mat_{m\times n} : \text{$A_{ij} = 0$ for all $(i,j) \in [p] \times [q]$}\}$.
 \end{cor}

This follows by the proof of Corollary~\ref{cor:SSXw-transitions}, \emph{mutatis mutandis}, via Proposition~\ref{prop:bruhat-order} and Lemma~\ref{lem:bruhat-minimal}.
The same is true of the following result compared to Lemma~\ref{lem:transition-containment-ss}:
the proof is the same after replacing $\FPF_n(\I_n)$ by $S^{m,n}_\infty$, $\Psi(z,p)$ by $\Phi(w,p)$, and the symbols
$\SSX$, $\SSI$, $\SSJ$ by $X$, $I$, $J$,
and then citing the relevant results in this subsection in place of their skew-symmetric 
predecessors.

\begin{lem}\label{lem:transition-containment}  
 Suppose $w \in S^{m,n}_\infty$ and $(p,q) \in [m]\times [n]$  is an outer corner of $\dom(w)$. Then 
 it holds that
$I(X_w) + (u_{pq}) \subseteq \bigcap_{v \in \Phi(w,p)} I(X_v)
$
and
$
J_w + (u_{pq}) = \bigcap_{v \in \Phi(w,p)} J_v
$,
where $(u_{pq})$ denotes the principal ideal of $\KK[u_{ij} : (i,j) \in [m]\times [n]]$ generated by $u_{pq}$.
\end{lem}

Putting all of this together, we recover the following results from \cite{KnutsonMiller}.

\begin{thm}[{\cite[Thms. A and B]{KnutsonMiller}}] \label{thm:initial-ideal} 
Let $w \in S^{m,n}_\infty$.
Then $J_w = \init (I_w) = \init (I(X_w))$ and $I_w = I(X_w)$.
Thus $I_w$ is a prime ideal, with a Gr\"obner basis given by the set of
minors $\det (\Umat_{AB})$ with $(A,B)$ ranging over all  pairs in
$ \tbinom{[i]}{q}\times \tbinom{[j]}{q} $ for $(i,j) \in [m]\times [n]$ with $q = \rank w_{[i][j]}+1$.
 \end{thm}

 \begin{proof}[Proof sketch]
 The identity $J_w = \init (I_w) = \init (I(X_w))$ is an analogue of Theorem~\ref{ss-main-thm}, and its proof is essentially the same.
 We argue by induction on the size of the complement of $D(w)$ in $[m]\times [n]$.
The base case corresponds to $w_{\square} =(n+1)(n+2)\cdots (n+m)12 \cdots n$,
which has  $D(w_\square) = [m]\times [n]$.
If $w=w_\square$ then it is easy to see that $I_w = J_w = I(X_w)$ is the ideal of $\KK[\Mat_{m\times n}]$
generated by all the variables $u_{ij}$ for $(i,j) \in [m]\times[n]$, so $J_w = \init (I_w) = \init (I(X_w))$.
 The inductive step follows by the same argument as in the proof of Theorem~\ref{ss-main-thm},
 changing what needs to be changed and using Lemma~\ref{lem:transition-containment} in place of Lemma~\ref{lem:transition-containment-ss}.
While in the skew-symmetric case it was nontrivial to show that
$\SSJ_z \subseteq \init(\SSI_z)$,  now we have $J_w \subseteq \init(I_w)$ by definition.

The claim that $I_w = I(X_w)$ is prime corresponds to Theorem~\ref{ss-main-thm2}
and likewise follows from Proposition~\ref{prop:init-facts}(d), since we already know that  $I_w \subseteq I(X_w)$ and that $X_w$ is irreducible.
The last statement, which gives a version of Theorem~\ref{ss-main-thm3}, is equivalent to  $J_w =\init (I(X_w))$.
 \end{proof}

We also sketch an alternate proof of the classical version of Theorem~\ref{thm:primary-decomposition-ss}.
Modifying our earlier notation, given a subset $D \subseteq [m]\times [n]$,
we write $(u_{ij} : (i,j) \in D)$ for the corresponding ideal in $\KK[u_{ij} : (i,j) \in [m] \times [n]]$.
We have such an ideal for each pipe dream of each
 $w \in S^{m,n}_\infty$, since if $D\in\RP(w)$ then
 $D \subseteq [m]\times [n]$ by  \cite[Thm. 3.7]{bergeron-billey},
and by definition $|D| = |D(w)| = \ell(w)$.

\begin{thm}[{\cite[Thm. B]{KnutsonMiller}}]
 \label{thm:primary-decomposition}
 Let  $w \in S^{m,n}_\infty$. Then $J_w = \bigcap_{D \in \RP(w)} (u_{ij} : (i,j) \in D)$.
  \end{thm}
 
 \begin{proof}[Proof sketch]
 We apply our usual inductive strategy.
 The base case corresponds to $w_\square =(n+1)(n+2)\cdots (n+m)12 \cdots n$,
for which $J_{w_\square} = (u_{ij} : (i,j) \in [m] \times [n])$ and 
 $\RP(w_\square) = \{ [m]\times [n]\}$ by \cite[Thm. 3.7]{bergeron-billey}.
If $w \in S_\infty^{m,n}\setminus\{w_\square\}$,
then $\dom(w)$ has an outer corner $(p,q) \in [m]\times [n]$
 and we wish to show that 
 $J_w + (u_{pq}) = \bigcap_{D \in \RP(w)} (u_{ij} : (i,j) \in D) + (u_{pq})$.
 Our argument is the same as in the proof of Theorem~\ref{thm:primary-decomposition-ss},
 except we
 swap out $z$, $\SSJ_z$, $\FP(z)$, $\Psi(z,p)$ for $w$, $J_w$, $\RP(w)$, $\Phi(w,p)$, and then
  replace Lemma~\ref{lem:transition-containment-ss} by Lemma~\ref{lem:transition-containment}
 and  \cite[Thm. 4.33]{HMPdreams} by \cite[Lem. 3.3]{Knutson}.
 The latter result implies that 
$D \mapsto D \sqcup \{(p,q)\}$ is a well-defined bijection 
$\RP(w) \to \bigcup_{v \in \Phi(w,p)} \RP(v)$.

We now wish to apply Lemma~\ref{lem:hilbert} 
    with $J = J_w $, $I =  \bigcap_{D \in \RP(w)} (u_{ij} : (i,j) \in D)$, and $f=u_{pq}$.
 Checking the hypothesis $I \cap (f) = f I$ for this lemma is straightforward.
What remains is to provide an analogue of Lemma~\ref{lem:Jw-in-JD-ss} showing that $J \subseteq I$.
We outline a mostly self-contained proof of this property.
Consider a pipe dream $D \in \RP(w)$ and choose a pair $(A,B) \in \binom{[i]}{q} \times \binom{[j]}{q}$ for some $(i,j) \in [m]\times [n]$ with $q= \rank w_{[i][j]}+1$, so that $u_{AB}$ is an arbitrary generator of $J_w$.
It suffices to show that $u_{ij} \mid u_{AB}$ for some $(i,j) \in D$, and for this is enough to check that 
the largest antidiagonal contained in $([i]\times [j])\setminus D$ has size less than $q$.

To prove this claim,
consider the permutation $v \in S^{m,n}_\infty$ with Rothe diagram 
$[q,i]\times[q,j]$.
The permutation matrix of $v$ is formed from the identity matrix by shifting rows $q,\dots,i$ to the right
by $j+1-q$ columns and then shifting rows $i+1,\dots,i+j+1-q$ to the left by $i+1-q$ columns.
Because $v$ is $321$-avoiding,
none of its reduced words are connected by any nontrivial braid relations,
and it follows from \cite[Thm. 3.7]{bergeron-billey}
that every pipe dream for $v$ is obtained by starting with the \emph{bottom pipe dream} $D_{\mathsf{bot}}(v) := [q,i] \times [j+1-q]$
and applying a sequence of \emph{basic ladder moves}
that each replace a single element of the form $(a+1,a)$ by $(a,a+1)$ if the pairs $(a,a),(a,a+1),(a+1,a+1)$ are all not present.

The largest antidiagonal in $([i]\times [j])\setminus E$ has size less than $q$ if $E=D_{\mathsf{bot}}(v)=[q,i] \times [j+1-q]$,
and this property is preserved by all basic ladder moves.
On the other hand, one can check using Proposition~\ref{ss-bruhat-prop} that $v \leq w$ in Bruhat order,
so the subword characterization of Bruhat order implies that each pipe dream $D \in \RP(w)$ contains some pipe dream $E\in \RP(v)$ as a subset
(see \cite[Lem. 4.8]{HMPdreams}).
But if $E\subseteq D$ then the size of
 largest antidiagonal  in $([i]\times [j])\setminus D$ cannot exceed that of
 $([i]\times [j])\setminus E$, which is already less than $q$ as desired. 
 \end{proof}

\end{document}